\documentclass{article}

\usepackage[margin=2.5cm, marginparwidth=2cm]{geometry}                % See geometry.pdf to learn the layout options. There are lots.
\usepackage[parfill]{parskip}    % Activate to begin paragraphs with an empty line rather than an indent
\usepackage{hyperref}
\hypersetup{
    colorlinks,
    citecolor=black,
    filecolor=black,
    linkcolor=black,
    urlcolor=black
}

\usepackage{amsthm, amsmath}
\usepackage{graphicx, xcolor}
\usepackage{caption, array}

\usepackage{amssymb}
\usepackage{dsfont}
\usepackage{epstopdf}
\usepackage[mathscr]{euscript}
\usepackage{tikz, tikz-cd}
\graphicspath{ {images/} }

\newtheorem{thm}{Theorem}[section]
\newtheorem{lemma}[thm]{Lemma}
\newtheorem*{thm*}{Theorem}
\newtheorem{prop}[thm]{Proposition}
\newtheorem{cor}[thm]{Corollary}
\newtheorem*{cor*}{Corollary}
\newtheorem{claim}{Claim}

\theoremstyle{remark}
\newtheorem{rmk}[thm]{Remark}
\newtheorem{q}[thm]{Question}
\newtheorem*{q*}{Question}

\theoremstyle{definition}
\newtheorem{defn}[thm]{Definition}
\newtheorem{notation}{Notation}

\long\def\Restate#1#2#3#4{
\medskip\par\noindent
{\bf #1 \ref{#2} #3} {\it #4}\par\medskip }

\newcommand{\bouquet}{\vee^g S^1}
\newcommand{\bxb}{(\vee^g S^1) \times (\vee^g S^1)}
\newcommand{\hf}{\hat{f}}

\title{Heegaard splittings and virtually special square complexes}
\author{Chandrika Sadanand}
\begin{document}

\maketitle

\begin{abstract}

We give a new perspective of Heegaard splittings in terms square complexes and Guirardel's notion of a \textit{core} which allows for combinatorial measurement of the obstruction to being a connect sum of Heegaard diagrams.  A Heegaard splitting is a decomposition of a closed orientable $3$-manifold into two isomorphic handle bodies that have a shared boundary surface. Usually, a number of curves on the shared boundary surface, called a Heegaard diagram, are used to describe a Heegaard splitting. We define a larger object, the \textit{augmented Heegaard diagram}, by building on methods of Stallings and Guirardel to encode the information of a Heegaard splitting. 

 \textit{Augmented Heegaard diagrams} have several desirable properties: each 2-cell is a square, they have \textit{non-positive combinatorial curvature} and they are \textit{virtually special}. Restricting to manifolds that do not have $S^1 \times S^2$ as a connect summand, augmented Heegaard diagrams are tied to the decomposition of a $3$-manifold via connect sum as described above. 
\end{abstract}

\tableofcontents

\section{Introduction}

%This work continues the 1965 work of John R. Stallings \cite{hownot}, where he encoded Heegaard splittings of $3$-manifolds as maps between two-dimensional spaces and proceeds to study these from a homological view point. Here, we study them from a more geometric and combinatorial point of view.

Any closed $3$-manifold $M$ is the union of two handle bodies of genus $g$ intersecting along their common boundary surface (a Heegard splitting). The information in this decomposition is typically encoded by a set of curves on the boundary surface called a Heegaard diagram. Stallings shows that it can equally well be encoded by a homotopy class of maps between two $2$-complexes: the boundary surface  and  the product of a bouquet of $g$ circles with it itself. We call this map (up to homotopy) the \textit{Stallings map} (see Section \ref{encoding}).

The \textit{handle slide group} acts transitively on the set of Heegaard diagrams that describe a particular Heegaard splitting, and $Out(F_g) \times Out(F_g)$ modulo a small kernel acts transitively on the Stallings maps describing a particular Heegaard splitting. We find that these two structures are equivalent in the following way. 

\Restate{Theorem}{heegaarddiagram}{}{Stallings maps are in bijection with Heegaard diagrams for a given Heegaard splitting. The handle slide group is isomorphic to $Out(F_g) \times Out(F_g)$ modulo a small kernel. The bijection is equivariant with respect to these actions and this isomorphism.}

A Heegaard splitting is said to be \textit{reducible} if there is an essential simple closed curve on the splitting surface that bounds a disk in each of the handle bodies. If a Heegaard splitting is not reducible it is said to be \textit{irreducible}. A continuous map from a surface to a product of bouquets of circles must by no means be injective; it can exhibit ``folding" and ``crumpling" behaviour. In fact, a Stallings map cannot be an embedding. The first insight of this note is that a Stallings map for an irreducible Heegaard splitting can be simplified (via combinatorial moves) to be a local embedding (see Section~\ref{mainthm}):

\Restate{Theorem}{locinj}{}{An irreducible Heegaard splitting can be encoded by a Stallings map which is locally an embedding.}

More generally, for a Stallings map for any Heegaard splitting, as long as $S^2 \times S^1$ is not a connect summand of $M$, the map can be simplified so that it is locally an embedding everywhere except for a minimal disjoint union of subsurfaces. In this case, the Stallings map factors through a space, which we call a \textit{pinched surface}, resulting from collapsing these subsurfaces to points.

\Restate{Theorem}{thm:pinch}{}{Suppose $M$ does not have $S^2 \times S^1$ as a connect summand, and has a Heegaard splitting $M = \mathcal{H}_1 \cup \mathcal{H}_2$. Then there is a Stallings map for the Heegaard splitting that factors through a pinched surface $\Sigma^*$ such that the quotient map is locally an embedding.}

\medskip\par\noindent
{\bf Corollary}{ \it Local geometric structures on the target (or codomain) $2$-complex can be pulled back  to the domain surface (or pinched suface) by the locally injective Stallings map. 
\par\medskip }

For example the codomain is naturally decomposed into squares with the link of each vertex having cycles with at least four edges. This means the combinatorial curvature of the codomain  is at most zero. These properties can be pulled back to the domain surface (or pinched surface) to give a non-positively curved square complex structure. Namely the domain surface is tiled by squares with at least four squares meeting at each vertex. In fact, more is true about this square complex structure.

\medskip\par\noindent
\Restate{Proposition}{virtuallyspecial}{}{The Stallings map can be encoded by a non-positively curved square complex structure on the quotient of the domain $\Sigma^*$ that is virtually special (in this case, this means the edges can be labelled by V (vertical) or H (horizontal) such that each square has a pair of opposite edges that are V, and the other pair of edges are H, and when a finite sheeted cover is taken, the closed horizontal and vertical strips of squares satisfy some injectivity conditions, see Section~\ref{square complexes})}
%has a vertical-horizontal structure and is clean (vertical-horizontal or VH means the edges can be labelled by V (vertical) or H (horizontal) such that each square has a pair of opposite edges that are V, and the other pair of edges are H, and clean means the map on the boundary of vertical strips of squares and horizontal strips of squares satisfies and injectivity condition, see Section~\ref{square complexes}).

This vertical-horiztontal property on a surface can alternately be understood as a half-translation structure in the case where the domain is a surface (Namely, it is possible to make a developping mapping of the squares to $\mathbb{R}^2$ to make planar charts on the complement of the vertices so that all the transition maps between charts needed to recover the surface are compositions of translations and rotations by $\pi$).

We now add to this square complex structure to form what we call the \textit{augmented Heegaard diagram}. This will allow for a characterization of irreducibility.

Recall that an essential simple closed curve on the splitting surface is said to be \textit{reducing} if it bounds an embedded disk in each of the handle bodies. These disks provide an embedded copy of $S^2 \subseteq M$. If $S^2 \times S^1$ is not a connect summand of $M$, this $S^2$ separates $M$ into two connect summands. Call a Heegaard diagram \textit{simple} if there is a set of disjoint reducing curves that decompose the Heegaard splitting $M= \mathcal{H}_1 \sqcup \mathcal{H}_2$ entirely into minimal genus Heegaard splittings of its prime connect summands. Furthermore, these reducing curves must not intersect any of the Heegaard diagram curves. Given a Heegaard diagram, it is difficult to know whether it is simple. One might hope that the Heegaard diagram with the minimal number of intersections is simple, but Osborne has given counterexamples to this \cite{hempel}. Whitehead and Volodin-Kuztnetsov-Fomenko, each conjectured that some combinatorial properties could be used to detect whether a Heegaard diagram is simple, but their conjectures have proven to be false, outside of specific cases of low genus \cite{morikawa} \cite{vkfcounterexample} \cite{ochiai} \cite{viro}.

Similarly, (assuming $S^2 \times S^1$ is not a connect summand of $M$) in Theorem~\ref{thm:pinch}, we see that the Stallings map for a Heegaard splitting sends many curves to points, all of which are reducing. However, even after simplifying, a Stallings map may not pinch a maximal multicurve of reducing curves, revealing the connect sum structure of $M$. More structure is needed.

With this in mind, we construct the \textit{augmented Heegaard diagram} $\mathscr{C}$ in Section~\ref{sec: VH}. The augmented Heegaard diagram consists of the square complex $\Sigma^*$ with some additional disks (also tiled by squares) attached. We use the virtually special square complex structure of $\Sigma^*$ and Stallings maps to construct $\mathscr{C}$, which is also virtually special. We find that by further simplifying the Stallings map for a given Heegaard splitting, the area of $\mathscr{C}$ can be minimized, and that in this case, $\Sigma^*$ pinches a set of reducing curves that maximally decomposes $M$ with respect to connect sum.

\Restate{Theorem}{thm:augheeg}{}{ If the augmented Heegaard diagram $\mathscr{C}$ is of minimal area in its $Out(F_g) \times Out(F_g)$ orbit, the associated Heegaard diagram is simple.}

The statement of Theorem~\ref{thm:augheeg} can be seen as a generalization and adjustment of conjectures of Whitehead and Volodin-Kuznetsov-Fomenko, which were found to be false \cite{morikawa} \cite{vkfcounterexample} \cite{ochiai} \cite{viro}. 

%This leads to the following characterization of irreducilibity of Heegaard spittings:
%
%\Restate{Corollary}{conditionforirreducibility}{}{Let $M$ be a $3$-manifold which does not have $S^1 \times S^2$ as a connect summand. A Heegaard decomposition for $M$ is irreducible if and only if it has an augmented Heegaard diagram that is a surface.}

In order to constuct $\mathscr{C}$ and prove Theorem~\ref{thm:augheeg}, we consider the virtually special square complex $\Sigma^*$ as a pair of actions of the fundamental group of the splitting surface $\Sigma$ on trees (see Section~\ref{guirardel core}). Each action encodes the attachment of a handle body to the surface, and so the information of the Heegaard splitting and $3$-manifold is contained in the comparison of these actions.

In 2005 Vincent Guirardel introduced a kind of a ``convex" core, comparing two actions of a group on $\mathbb{R}$-trees \cite{guirardel}. The size of the core measures the obstruction to the two actions being equivariant projections of a single action on a tree. This can be interpreted as the core measures how close a Heegaard splitting is to a connect sum of several copies of $S^2 \times S^1$. When two actions of $\pi_1(\Sigma)$ on a tree are exactly the same, they encode the attachment of two handle bodies to $\Sigma$ in the same way. This constructs a connect sum of several copies of $S^2 \times S^1$. In this case, the Guirardel core is not even full-dimensional.

More generally the Guirardel core has been used to study the orbit of cofinite actions on an $\mathbb{R}$-tree by subgroups of $Out(F_g)$ \cite{BBC}, \cite{CBP}, \cite{CP}, \cite{CP2}, \cite{horbez}. This makes it particularly well-suited for the problem in this paper, since the Stallings maps for a given Heegaard splitting form an $Out(F_g) \times Out(F_g)$ orbit.

We find that the Guirardel core is indeed intimately related to the Stallings map of a Heegaard splitting. Let $f: \Sigma \rightarrow G_1 \times G_2$ be a Stallings map encoding a Heegaard splitting (where $\Sigma$ is the splitting surface and $G_1$ and $G_2$ are bouquets of circles). Let $S$ be the cover of $\Sigma$ corresponding to $\ker f_*$. Let $T_1 \times T_2$ be the universal cover of $G_1 \times G_2$ and let $\hat{f}: S \rightarrow T_1 \times T_2$ be the lift of $f$. The fundamental group of $\Sigma$ acts on $S$ and on $T_1 \times T_2$ (it acts on the latter in the form of $f_* \pi_1(\Sigma)$). It is these actions of $\pi_1(\Sigma)$ on $T_1$ and $T_2$ that encode the Heegaard splitting.

\[
\begin{tikzcd}
S \arrow[r, "\hat{f}"] \arrow[d] & T_1 \times T_2  \arrow[d]\\
\Sigma \arrow[r, "f"] & G_1 \times G_2
\end{tikzcd}
\]

\Restate{Theorem}{thm:coreandhatf}{}{Let $p_1$ and $p_2$ be the coordinate projections of $T_1 \times T_2$. The Guirardel core of the actions of $\pi_1(\Sigma)$ on $T_1$ and $T_2$ is the smallest set $X$ such that $\hat f(S) \subseteq X$ and ${p_1}^{-1} (t)$ is connected for every $t \in T_1$ and ${p_2}^{-1}(t)$ is connected for every $t \in T_2$.}

This leaves us with four interrelated ways to capture the information of a Heegaard splitting and they are each in 1-1 corresponance: Heegaard diagrams, Stallings maps, virtually special square complex structures, pairs of actions of $\pi_1(\Sigma)$ on trees. There are combinatorial moves corresponding to handle slides for Heegaard diagrams for each of the other three structures. Just as there are many Heegaard diagrams for a single Heegaard splitting, there are many homotopy classes of maps, many square complex structures, and many pairs of actions of $\pi_1(\Sigma)$ on trees all encoding the same Heegaard splitting.

This geometric picture leads to the question of characterizing the relationship between properties of these structures, properties of the Heegaard splitting and the nature of the $3$-manifold they describe. This work is completed for genus one Heegaard splittings in Section~\ref{lensspaces}, which is devoted to Lens spaces. The tractability of special cube complexes for geometric group theory methods led to a proof of the virtual Haken conjecture \cite{agol}.  A program exploring the higher genus cases might give a $2$-dimensional and similarly tangible perspective on the classification of $3$-manifolds. 

More specifically, one might look for algorithms to minimize the area of $\mathscr{C}$ in terms of the core, relating properties of $\mathscr{C}$ for irreducible 3-manifolds to properties of its JSJ-decomposition, and exploring how to leverage methods of special square complexes to prove theorems about $3$-manifolds via the augmented Heegaard diagram.

\subsubsection*{Outline} Section~\ref{prelim} lays out the definitions and prelimaries of Heegaard spltitings as well as the geometric group theory tools that are used to understand Stallings maps, such as special square complexes and the Guirardel core. 
Section~\ref{encoding} defines Stallings maps, and details their relationship to Heegaard diagrams. Section~\ref{mainthm} relates Stallings maps and irreducibility of Heegaard splittings, proving Theorem~\ref{thm:pinch}. Section~\ref{sec: VH} constructs the augmented Heegaard diagram and proves Theorem~\ref{thm:augheeg}. Section~\ref{directions} lists questions for future study. Section~\ref{lensspaces} works through the example of Heegaard splittings of genus one, finding their Stallings maps, and Guirardel cores.

\subsubsection*{Acknowledgements}
The author acknowledges Dennis Sullivan for suggesting this problem and guiding the author through a PhD. Great thanks are due to Mark Hagen and Chris Leininger, who listened to endless iterations of this work, providing suggestions, feedback and support, and to Scott Taylor for reading through an initial draft in detail and providing some corrections. Thank you to Moira Chas for repeatedly insisting that this work should be published. Finally, the author acknowledges Tali Pinsky, Tarik Aougab, Nissim Ranade, Cameron Crowe and Michael Wan for helpful conversations that brought clarity.
\section{Preliminaries}
\label{prelim}
We recall some definitions and facts that will be useful in later sections.

\subsection{Heegaard splittings and diagrams}
\begin{defn}
A \textit{Heegaard splitting} or \textit{Heegaard decompostion} of $M$ is a decomposition $M=\mathcal{H}_1 \cup \mathcal{H}_2$, where $\mathcal{H}_1$ and $\mathcal{H}_2$ are homeomorphic $3$-manifolds that intersect in their shared boundary, which is a surface of genus $g$. Further, each $\mathcal{H}_i$ is a \textit{handle body}. Namely, $\mathcal{H}_1$ and $\mathcal{H}_2$ are the result of taking a closed $3$-ball and attaching $g$ $1$-handles, or solid cylinders.

\begin{figure}[h]
\centering \includegraphics[width=1.5in]{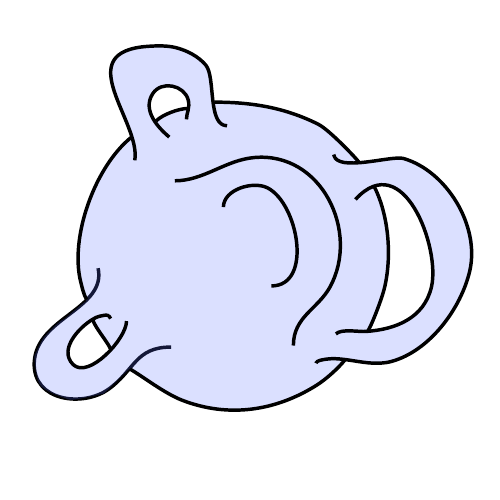}
\caption{A handle body of genus four.}
\end{figure}

The shared boundary $\partial \mathcal{H}_1 = \partial \mathcal{H}_2$ will be called the \textit{splitting surface} $\Sigma_g$, where $g$ is the genus.\\
\end{defn}

Every compact $3$-manifold without boundary has infinitely many Heegaard decompositions (of higher and higher genus).

\begin{defn}
A \textit{Heegaard diagram} describing a Heegaard splitting is a triple $(\Sigma_g, \mathscr{A}, \mathscr{B})$, where $\mathscr{A}$ and $\mathscr{B}$ are sets of $g$ curves on a surface of genus $g$, $\Sigma_g$, with the following properties.
\begin{enumerate}
\item The curves in $\mathscr{A}$ represent a basis for a Lagrangian subspace of homology (they are a maximal set of homologically non-trivial, independent, disjoint representatives) each with null homotopic image under $\iota_{1}$, the inclusion of $\Sigma_g$ into the handle body $\mathcal{H}_1$.

\item the curves in $\mathscr{B}$ represent a basis for a Lagrangian subspace of homology (they are a maximal set of homologically non-trivial, independent disjoint representatives) each with null homotopic image under $\iota_{2}$, the inclusion of $\Sigma_g$ into the handle body $\mathcal{H}_2$.
\end{enumerate}
It is common to call the free homotopy classes in $\mathscr{A}$ \textit{blue diagram curves} and those in $\mathscr{B}$ \textit{red diagram curves}. They will be depicted in these colours in figures.
\end{defn}

The curves in $\mathscr{A}$ dictate the attachment of $\mathcal{H}_1$ to $\Sigma_g$ with the following procedure. Cut $\Sigma_g$ along the curves of $\mathscr{A}$. This yields a sphere with $2g$ disks removed, because the curves are homologically non-trivial and independent and because there are $g$ of them. Label all the boundary components by the curve in $\mathscr{A}$ they were cut from. Glue $2g$ disks to this space to form a sphere. Attach a $3$-ball to the sphere. Finally identify pairs of the glued on disks according to the labeling on their boundary. Similarly, the curves in $\mathscr{B}$ dictate the attachment of $\mathcal{H}_2$ to $\Sigma_g$, and thus one can reconstruct a Heegaard splitting from a diagram.

\begin{defn} \label{handleslide}
Two Heegaard diagrams $(\mathscr{A}, \mathscr{B})$ and $(\mathscr{A}', \mathscr{B}')$ on a surface of genus $g$ differ by a \textit{handle slide} if $\mathscr{B}=\mathscr{B}'$, and $\mathscr{A} \setminus \{ \gamma \} = \mathscr{A}' \setminus \{ \gamma ' \}$ such that for some $\delta \in \mathscr{A}$, $\delta$, $\gamma$, and $\gamma '$ bound a pair of pants. The roles of $\mathscr{A}$ and $\mathscr{B}$ may be reversed as well.
\end{defn}

\begin{thm}
(Reidemeister-Singer \cite{reidemeister} \cite{singer}) Given two Heegaard diagrams for a Heegaard decomposition, one can be obtained from the other by performing a finite sequence of handle slides.
\end{thm}

\subsection{VH complexes and special square complexes}
\label{square complexes}
Here, the definitions from ``VH complexes, towers and subgroups of $F \times F$" by Bridson and Wise~\cite{cleanVHcomplexes} as well as ``Special Cube complexes" by Haglund and Wise~\cite{haglundwise} are introduced.\\
\begin{defn}
A \textit{combinatorial map} of CW complexes is a continuous map between two CW complexes such that each open cell is sent homeomorphically onto an open cell. A CW complex is called \textit{combinatorial} if all its attaching maps are combinatorial. Finally a \textit{square complex} is a combinatorial CW complex of dimension two, where the top dimensional cells are attached along loops that are four edges long.\\
\end{defn}

\begin{defn}
\label{vh}
A square complex is VH if the edges can be partitioned into two sets: vertical (V) and horizontal (H) so that the attaching map of each  $2$-cell alternates between vertical and horizontal edges.\\
\end{defn}

 Let $X$ be a VH complex. Let $X_H$ be the $0$-skeleton $X^{(0)}$ with the horizontal edges attached and let $X_V$ be the $0$-skeleton $X^{(0)}$ with the vertical edges attached.\\
 
 \begin{defn}
  The connected components of the complement $X \setminus X_H$ are called \textit{horizontal corridors}. Similarly, the connected components of the complement $X \setminus X_V$ are called the \textit{vertical corridors}. Consider how the horizontal corridors must attach to $X_H$ to reconstruct $X$ (and similarly for the vertical corridors and $X_V$). A VH complex is \textit{clean} if the attaching maps of the horizontal (and respectively vertical) corridors to $X_H$ (and respectively $X_V$) are injective.\\
 \end{defn}
 
 VH complexes are the predecessors of \textit{special cube complexes} which are cube complexes that don't exhibit certain pathological behaviours. We recall the following definition, but ask the reader to see \cite{haglundwise} for the details.\\
 \begin{defn}
 A square complex is \textit{special} if:
 \begin{enumerate}
 \item Each hyperplane embeds.
 \item No hyperplane directly self osculates.
 \item No hyperplane indirectly self osculates.
 \end{enumerate}
 
 A square complex is \textit{virtually special} if it has a finite-sheeted cover that is special.
 \end{defn}
 See \cite{haglundwise} for definitions of hyperplane and direct and indirect self osculation.

\subsection{Guirardel core}
\label{prelim:guirardel} 
Below, the definitions of ``Core and intersection number for group actions on trees" by Guirardel \cite{guirardel} are introduced. The reader is encouraged to consult this paper as a reference for more details.

Begin with two actions of a group $G$ on a trees $T_1$ and $T_2$, or equivalently an action of $G$ on $T_1 \times T_2$. In this paper we will be looking at the special case where $T_1 =T_2 = T$ and it is an infinite tree where every vertex is of the same finite valence, and edges are of length 1. Below, we give Vincent Guirardel's definitions for this case \cite{guirardel}\\

\begin{defn}
A \textit{direction} $\delta$ is a connected component of $T \setminus \{x\}$ for some point $x$ in $T$. A \textit{quadrant} $Q = \delta_1 \times \delta_2$ is the cartesian product of two directions (and therefore is a subset of $T\times T$).\\
\end{defn}

\begin{defn}
Let $(t_1, t_2) \in T \times T$. A quadrant $Q =  \delta_1 \times \delta_2$  is \textit{heavy} if there exists a sequence $\{\gamma_n\} \subseteq G$ such that
\begin{itemize}
\item $\gamma_n \cdot (t_1,t_2) \in Q$ for all $n$,
\item $d_{\delta_1}(t_1, \gamma_n \cdot t_1) \rightarrow \infty$ as $n \rightarrow \infty$ (where $d_{\delta_1}(t_1, \gamma_n \cdot t_1)$ is the distance between $t_1$ and $\gamma_n \cdot t_1$ in $\delta_1$) and
\item  $d_{\delta_2}(t_2, \gamma_n \cdot t_2) \rightarrow \infty$ as $n \rightarrow \infty$ (where $d_{\delta_2}(t_2, \gamma_n \cdot t_2)$ is the distance between $t_2$ and $\gamma_n \cdot t_2$ in $\delta_2$).
\end{itemize}
$Q$ is \textit{light} if it is not heavy.
\end{defn}

Now we can define the core, which is used to compare the pair of actions on $T$.\\
\begin{defn} \label{defn:guirardel}
The \textit{core} $C$ of $T \times T$ with respect to the action of $G$ is the complement of the union of light quadrants in $T \times T$. The \textit{intersection number} of the two actions is the co-volume of the action of $G$ on $C$. In other words, the intersection number is the number of 2-cells in $C/G$, or the area of $C/G$.
\end{defn}

The larger the intersection number, the larger the obstruction to finding a third tree $T_3$ with an action of $\pi_1(\Sigma_g)$ and surjective equivariant maps onto $T_1$ and $T_2$. In the case where $G$ acts on $T_1$ and $T_2$ in the same way, $C/\pi_1(\Sigma_g)$ has zero area. It is common for the area to be infinite, and we will confront this fact in Section~\ref{sec: VH}.

We finish by stating a characterization of the core that will be helpful later on.\\
\begin{thm}[Proposition 5.1 in \cite{guirardel}]
Let $G$ act on two trees $T_1$ and $T_2$ such that $C \neq \emptyset$. Let $F \subseteq T_1 \times T_2$ be a non-empty closed connected $G$-invariant subset with connected fibres. Then $F$ contains $C$. Moreover, $C$ is the intersection of all such sets $F$.
\end{thm}

\section{Heegaard splittings as Stallings maps}
\label{encoding}

\subsection{From Heegaard splittings to two-dimensional maps}
\label{splittingstomaps}
Given a Heegaard splitting $M=\mathcal{H}_1 \cup \mathcal{H}_2$, with a splitting surface $\Sigma_g = \partial \mathcal{H}_1 = \partial \mathcal{H}_2$, there are two inclusion maps $\iota_i: \Sigma_g \longrightarrow \mathcal{H}_i$. Knowing these two maps is sufficient to reconstruct $M$. The two maps can be expressed as a single product map $\iota_1 \times \iota_2: \Sigma_g \longrightarrow \mathcal{H}_1 \times \mathcal{H}_2$.\\

Now, we simplify by considering this map up to homotopy, and considering the spaces up to homotopy equivalence. We post-compose $\iota_1$ and $\iota_2$ by a deformation retraction that takes $\mathcal{H}_i$ to a graph of genus $g$, $\vee^g S^1$. Let $\iota'_1$ and $\iota'_2$ be the respective compositions. This gives a map from a surface to a product of graphs.

\[f= \iota'_1 \times \iota'_2: \Sigma_g \longrightarrow G_1 \times G_2
\]

There was some choice in the construcion of $f$, above, and to account for this, we consider all possible outcomes together as a class of maps $\mathcal{A}_{M,\mathcal{H}_1, \mathcal{H}_2}$.\\

\begin{defn}
\label{defn: describe}
Let $M= \mathcal{H}_1 \cup \mathcal{H}_2$ be a Heegaard splitting. Let $\Sigma_g$ be the splitting surface and let $\iota_i$ be in the inclusion maps $\Sigma_g \rightarrow \mathcal{H}_i$ for $i=1, 2$. The set $\mathcal{A}_{M,\mathcal{H}_1, \mathcal{H}_2}$ is defined as $\{ f: \Sigma_g \rightarrow G_1 \times G_2  \}$, where $G_i$ are graphs of genus $g$ and each coordinate $f_i$ is homotopic to $\iota_i$, up to homotopy equivalence of the target, $\mathcal{H}_i$.

We say that a map $f  \in \mathcal{A}_{M,\mathcal{H}_1, \mathcal{H}_2}$ \emph{encodes} or \textit{describes} the Heegaard splitting $M=\mathcal{H}_1\cup \mathcal{H}_2$.
\end{defn}

Note that if $f \in \mathcal{A}_{M, \mathcal{H}_1, \mathcal{H}_2}$ and $f'$ is homotopic to $f$, $f'$ must also be in $\mathcal{A}_{M, \mathcal{H}_1, \mathcal{H}_2}$. Therefore $\mathcal{A}_{M, \mathcal{H}_1, \mathcal{H}_2}$ can be partitioned into homotopy classes.

These maps are truly two-dimensional in the sense that both their domain and target spaces are two-dimensional and they display properties such as folding, and branch covering behaviour (locally homotopic to the holomorphic map $z \mapsto z^n$).

\begin{rmk}
\label{properties}
Note that $\iota'_i$ has the following properties:
\begin{enumerate}
\item The homology kernel of $\iota'_i$  has rank $g$ and has a basis with representatives that are simple, closed, non-intersecting curves.
\item These representatives have null homotopic image (because being null-homologous and null-homotopic  are equivalent in a graph).
\item The image of the homology under ${\iota'_i}_*$ must be rank $g$ as well, meaning that ${\iota'_i}_*$ is surjective on homology. This implies that  the map induced on fundamental groups must also be surjective, since the codomain is a graph.
\end{enumerate}
\end{rmk}

The first property implies the other two above. These properties are not changed up to homotopy nor homotopy equivalence, and therefore any map in $\mathcal{A}_{M,\mathcal{H}_1, \mathcal{H}_2}$ will bear these properties in each of its coordinates.

\subsection{From two-dimensional maps to Heegaard splittings}

Begin with a map $f$ from a surface of genus $g$ to the product of two bouquets of $g$ circles, whose coordinates' homology kernel is of rank $g$, and each have a basis with representatives that are simple, closed, non-intersecting curves (property (1) from Remark~\ref{properties} above). One may construct a Heegaard splitting $M=\mathcal{H}_1 \cup \mathcal{H}_2$. This is done simply by taking the two sets of $g$ curves representing a basis for the homology kernels as a Heegaard diagram. Indeed, $f$ describes $M = \mathcal{H}_1 \cup \mathcal{H}_2$ uniquely.
%Notice the role of each of the conditions in the proof below. Condition (1) guarantees that these two sets of curves are indeed a Heegaard diagram, while conditions (2) and (3) guarantee that map we began with is in $\mathcal{A}_{M,\mathcal{H}_1, \mathcal{H}_2}$. 

\begin{prop}
\label{factorthrough}
Consider a map $f: \Sigma_g \rightarrow G_1\times G_2$, where $G_i$ are graphs of genus g. Suppose $f=f_1 \times f_2$ and the homology kernels of ${f_i}_*$ with $i=1,2$ have bases that can each be represented by a set of $g$ disjoint simple closed curves. Then there exists a unique Heegaard splitting $M=\mathcal{H}_1 \cup \mathcal{H}_2$ such that $f \in \mathcal{A}_{M,\mathcal{H}_1, \mathcal{H}_2}$.
\end{prop}

%This proposition is similar to and possibly follows from the the work of Papakyriakopoulos (see the third theorem in \cite{papa}).
\begin{proof}
This proof considers one coordinate of the map $f$ at a time. We show that there is a unique inclusion of $\Sigma_g$ into a handle body $\mathcal{H}_1$ as its boundary that is equivalent to $f_1$ up to homotopy and homotopy equivalence of the target. The proof for $f_2$ follows similarly.

\paragraph{Existence} Consider $f_1: \Sigma_g \rightarrow G_1$, the first coordinate of $f$.  Recall that property (1) implies properties (2) and (3). Properties (1) and (2) guarantee the existence of  $g$ simple, closed, non-intersecting curves with null homotopic image (which represent a basis of the homology kernel of $f_1$). Therefore, the map $f_1$ factors through a space constructed by gluing $g$ disks to $\Sigma_g$ along these curves. We note that this space has non-trivial $\pi_2$, but that the target $G_1$ does not, and so $f_1$ must in fact factor through inclusion to the handle body $\mathcal{H}_1$ obtained by gluing a $3$-ball along the $\pi_2$ class of the previous space. See Figure~\ref{fig: factor through} below. Note that the map $d$ (in Figure~\ref{fig: factor through}) induces an isomorphism on homology and inherits the property (3) from $f_1$ ($\pi_1$-surjectiveness).

\begin{figure}[h]
\[
\begin{tikzcd}
\Sigma_g \includegraphics[width=4cm]{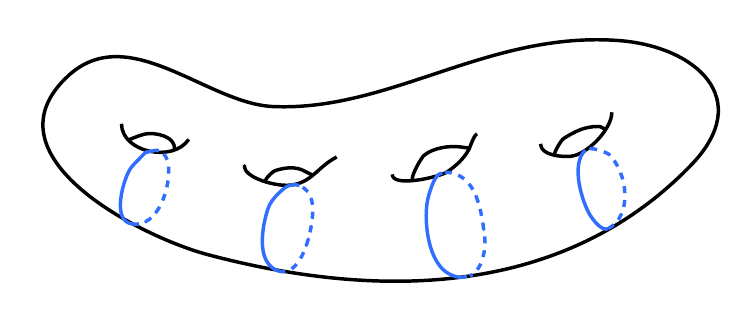} \arrow[rr, yshift=0.8cm, "f_1"] \arrow[dr] & & \includegraphics[width=2cm]{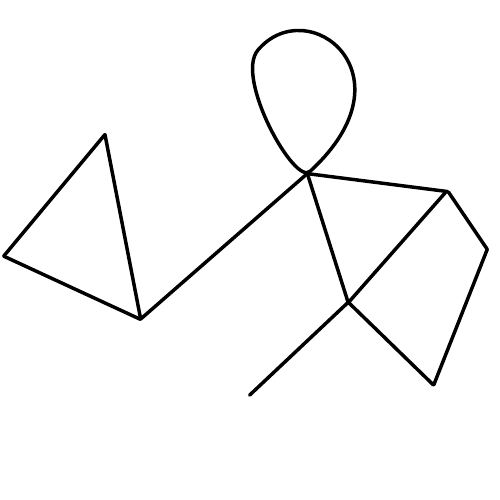} G_1\\
 & \mathcal{H}_1 \includegraphics[width=4cm]{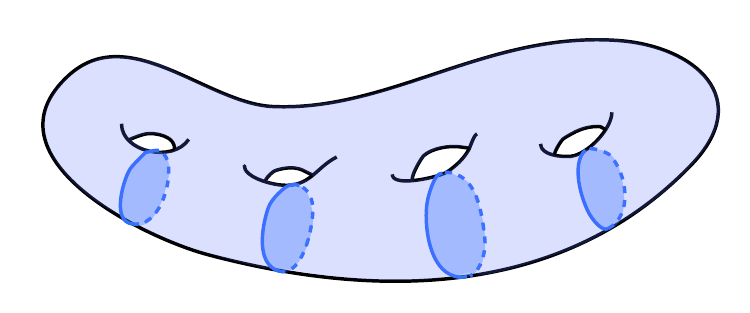} \arrow[ur, "d"] & \\
\end{tikzcd}
\]
\caption{The map $f_1$ must factor through a handle body $\mathcal{H}_1$.}
\label{fig: factor through}
\end{figure}

It remains to show that the map $d$ is part of a homotopy equivalence, and also the uniqueness of $\mathcal{H}_1$ with this property.

To show that $d$ is part of a homotopy equivalence, consider the following diagram.
\[
\begin{tikzcd}
A:= \vee^g S^1 \arrow[rr, bend left = 10, "h"] & & \mathcal{H}_1 \arrow[ll, bend left = 10] \arrow[r, "d"] & G_1 \arrow[rr, bend left = 10, "j"] & & \vee^g S^1 =: B \arrow[ll, bend left = 10] \arrow[lllll, dashed, bend right=50, "e"] 
\end{tikzcd}
\]

The diagram shows homotopy equivalences $h$ and $j$ between $\vee^g S^1$ and $\mathcal{H}_1$ and a graph of genus $g$ respectively. It suffices to show that $j \circ d \circ h$ is part of a homotopy equivalence. To this end, a map, $e$ will be constructed that forms a homotopy equivalence together with $ j \circ d \circ h$. The labels $A$ and $B$ help clarify discussion about the map $e$.

The fundamental groups of these spaces will be discussed assuming the chosen base point of both copies of $\vee^g S^1$ is the wedge point. 

Consider a copy of $S^1$ in $B$. It represents a generator of the fundamental group, $[\alpha]$. The map $d$ inherits $\pi_1$-surjectiveness from $f_1$, and so there is a preimage of $[\alpha]$ by $j \circ d \circ h$. Call this preimage $[a]$. Pick some $a \in [a]$. Now we construct $e$ by sending $\alpha$ to the image of $a$ linearly, and sending each copy of $S^1$ in $B$ similarly to the representative chosen by the procedure outlined in this paragraph.

To show that $e$ and $j \circ d \circ h$ are homotopy inverses, consider a copy of $S^1$ in $B$. Consider it as a based loop $\beta$. By construction, $(j \circ d \circ h) \circ e(\beta)$ and $\beta$ are base-point homotopic. Turning attention to $e \circ (j \circ d \circ h)$, consider a copy of $S^1$ in $A$. Consider it as a loop, $\gamma$, that represents an element of $\pi_1(A)$. Note that the construction of $e$ guarantees that both $[\gamma]$ and its image by $e \circ (j \circ d \circ h)$ have the same image in $\pi_1(A)$.

Simultaneously applying these based homotopies, we get a homotopy between identity and $(j \circ d \circ h) \circ e$ and $e \circ (j \circ d \circ h)$ respectively.

%We show that the induced map $(j \circ d \circ h)_*$ on fundamental groups between $A$ and $B$ is injective, and thus $[\gamma] = [e \circ (j \circ d \circ h)(\gamma)]$ for every copy of $S^1$ in $A$., which implies that $e \circ (j \circ d \circ h)$ is homotopic to identity. Consider a generating set of $\pi_1(A)$ with each element supported on a copy of $S^1$. None of these have null-homotopic image, because they are non trivial homologically, and we know that $d$ is a homology isomorphism. Now, because the fundamental groups of $A$ and $B$ are free, $(j \circ d \circ h)$ must be an injection. This completes the proof of exsitence.

\paragraph{Uniqueness} Uniqueness follows rather quickly. Suppose there were inclusions into two handle bodies $\mathcal{H}_1$ and $\widetilde{\mathcal{H}_1}$ that were equivalent to $f_1$. These are shown below with $d$ and $\tilde{d}$ both homotopy equivalences.
\[
\begin{tikzcd}
 & \widetilde{\mathcal{H}_1} \arrow[rd, bend left = 10, "\tilde{d}"]& \\
\Sigma_g \arrow[ru] \arrow[rr, "f_1"] \arrow[rd]& & G_1 \arrow[lu, bend left = 10] \arrow[ld, bend left = 10]\\
 & \mathcal{H}_1 \arrow[ru, bend left = 10, "d"]
\end{tikzcd}
\]
One can simply compose $d$ and $\tilde{d}$ to show that in fact the two inclusions are equivalent (up to homotopy and homotopy equivalence of the target). Thus each coordinate of  $f$ encodes the attachment of a handle body to $\Sigma_g$ in a unique way, encoding a unique Heegaard splitting.
\end{proof}

From this proof, it is clear that the Heegaard splitting $M=\mathcal{H}_1 \cup \mathcal{H}_2$ can be reconstructed from any map in $\mathcal{A}_{M, \mathcal{H}_1, \mathcal{H}_2}$. This justifies the terminology of Definition \ref{defn: describe} (that a map $f \in \mathcal{A}_{M, \mathcal{H}_1, \mathcal{H}_2}$ \emph{encodes} the Heegaard splitting $M=\mathcal{H}_1 \cup \mathcal{H}_2$).

We summarize the situation with the following corollary.

\begin{cor}
A map $f: \Sigma_g \rightarrow G_1 \times G_2$ encodes a Heegaard splitting iff it satisfies property (1) from Remark~\ref{properties}, Section~\ref{splittingstomaps}.
\end{cor}

%\begin{defn}
%A map $f$ from a surface $S$ of genus $g$ to the cartesian product of two bouquets of $g$ circles \textit{describes a Heegaard splitting} $M=\mathcal{H}_1 \cup \mathcal{H}_2$ if it can be written as $\iota'_1 \times \iota'_2$ up to homotopy. Here $\iota_i$ are the inclusions of the separating surface into the handle bodies, followed by some deformation retraction of the handle bodies to a bouquet of circles.
%\end{defn}

\subsection{Heegaard diagrams and handle slides}
\subsubsection*{Heegard diagrams and Stallings maps}
In order to study $\mathcal{A}_{M, \mathcal{H}_1, \mathcal{H}_2}$ more systematically and relate it to Heegaard diagrams, we focus on the subset $\mathcal{B}_{M, \mathcal{H}_1, \mathcal{H}_2} \subseteq \mathcal{A}_{M, \mathcal{H}_1, \mathcal{H}_2}$  where $G_1$ and $G_2$ are bouquets of $g$ circles. Let $\overline{\mathcal{B}}_{M, \mathcal{H}_1, \mathcal{H}_2}$ denote the homotopy classes of maps in $\mathcal{B}_{M, \mathcal{H}_1, \mathcal{H}_2}$. Next, consider the equivalence relation on $\overline{\mathcal{B}}_{M, \mathcal{H}_1, \mathcal{H}_2}$  generated by $[f] \sim [f']$ if and only if there exists $f \in [f]$, $f' \in [f']$ and $k$ such that $k \circ f = f'$, where $k: \vee^g S^1 \rightarrow \vee^g S^1$ is a homeomorphism. Finally, let $B_{M, \mathcal{H}_1, \mathcal{H}_2} := \overline{\mathcal{B}}_{M, \mathcal{H}_1, \mathcal{H}_2} / \sim$.

Notice that $g$ has the effect of permuting the circles of $\vee^g S^1$ or reversing the orientation of some of the circles of $\vee^g S^1$.  Therefore, elements in the quotient, $B_{M, \mathcal{H}_1, \mathcal{H}_2}$ broadly keep track of which part of $\Sigma_g$ wrap around a circle of $\vee^g S^1$ without recording which circle is which, or the direction in which the surface wraps around it. It turns out that this information is equivalent to a Heegaard diagram.

\begin{thm} \label{heegaarddiagram}
There is a bijection $\mathscr{S}: \{ \text{Heegaard diagrams for } M = \mathcal{H}_1 \cup \mathcal{H}_2 \} \rightarrow B_{M, \mathcal{H}_1, \mathcal{H}_2}$. Any map that is a representative of $\mathscr{S}(x)$ is called a \textit{Stallings map} for the Heegaard diagram $x$.
\end{thm}

\begin{proof}
We define a bijection $\mathscr{S}^{-1}: B_{M, \mathcal{H}_1, \mathcal{H}_2} \rightarrow \{ \text{Heegaard diagrams for } M = \mathcal{H}_1 \cup \mathcal{H}_2\}$ as follows.

\paragraph{Defining $\mathscr{S}^{-1}$} Begin with  $[f] \in B_{M, \mathcal{H}_1, \mathcal{H}_2}$ and pick any representative $f \in [f]$. Let $f_1$ and $f_2$ be the two coordinates of $f$ (in other words, $f= f_1 \times f_2$). By Proposition~\ref{factorthrough}, one can think of the first coordinate $f_1$, as inclusion of $\Sigma_g$ into a handle body $\mathcal{H}_1$ as its boundary, followed by a homotopy equivalence $f_1': \mathcal{H}_1 \rightarrow \vee^g S^1$. Let $h: \vee^g S^1 \rightarrow \mathcal{H}_1$ be the other part of this homotopy equivalence. There are many Heegaard diagrams for the attachment of $\mathcal{H}_1$ to $\Sigma_g$ by this inclusion. However, $[f]$ determines one of them canonically using the following idea, and this canonical diagram is taken to be $\mathscr{S}^{-1}([f])$. 

Let $b \in \vee^g S^1$ be the bouquet point.  Pick distinct points $s_1, \ldots, s_g$ in $(\vee^g S^1)$, such that $s_i \neq b$, and there is one point on each circle.  Pick generators $\{ \sigma_i\}$ for $\pi_1(\vee^g S^1)$ such that the circle that contains $s_i$ is a representative for $\sigma_i$. The idea is that the points $\{s_i\}$ determine disks in $\mathcal{H}_1$, whose boundaries form the desired Heegaard diagram. The details of this process are laid out below.

Choose smooth structures on $\mathcal{H}_1$ and on the components of $\vee^g S^1 \setminus \{b\}$. Then modify $f_1$ by a homotopy so that it is smooth with respect to these structures (using Whitney's approximation theorem) and so that it is transverse to the points $s_i$ (using Thom transversality). Now $\{f_1^{-1}(s_i)\}$ is a collection of simple closed curves. Label each curve in $f_1^{-1}(s_i)$ by $\sigma_i$. This way, one can associate an intersection word in the letters $\{\sigma_i \}$ for curves on $\Sigma_g$.

Suppose there is an $i$ such that $f_1^{-1}(s_i)$ consists of at least two simple closed curves $\gamma$ and $\delta$. If there is an oriented simple curve arc $\alpha$ in $\Sigma_g \setminus \{f_1^{-1}(s_i)\}$ intersecting $\gamma$ and $\delta$ whose intersection word is $\sigma_i \sigma_i^{-1}$, then a neighbourhood of  $\alpha \cup \gamma \cup \delta$ is a pair of pants $\mathcal{P}$. One boundary component of the pants is homotopic to $\gamma$, one is homotopic to $\delta$, and the third is homotopic to a concatenation of the form $\gamma \cdot \alpha \cdot \delta \cdot \alpha^{-1}$ (with appropriate choices of orientation for $\gamma$ and $\delta$). In this case, $f_1$ can be modified by a homotopy supported on $\mathcal{P}$ so that $f_1^{-1}(\sigma_i)$ is a single curve tracing $\gamma \cdot \alpha \cdot \delta \cdot \alpha^{-1}$. Each such modification reduces the number of curves in $\{f_1^{-1}(s_i)\}$, and so after a point, there are no more arcs travelling between two curves in $\Sigma_g$ with intersection word $\sigma_i \sigma_i^{-1}$ for all $i$. Next, modify $f_1$ by another homotopy to remove any null homotopic curves in $\{f_1^{-1}(s_i)\}$ (this is possible because $\vee^g S^1$ is aspherical). Both these homotopies can be realized by modifying $f_1$ by a homotopy to that it is locally a Morse function, and then pushing parts of the surface above or below the height value $s_i$.

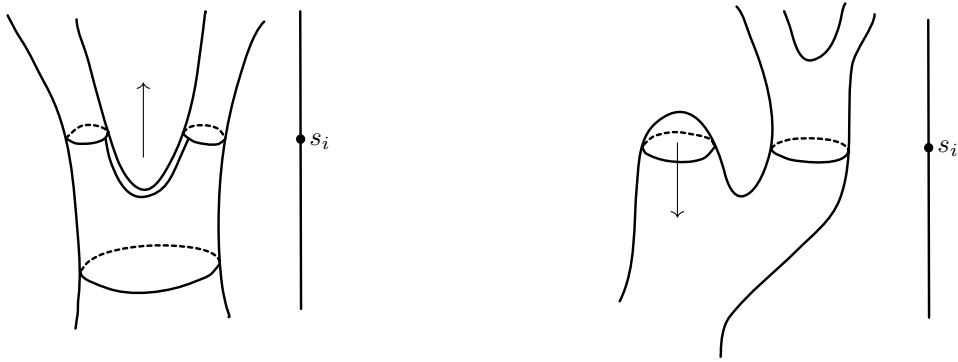
\begin{figure}[h]
\begin{tikzpicture}
\node (image) at (0,0) { \includegraphics[width=\textwidth]{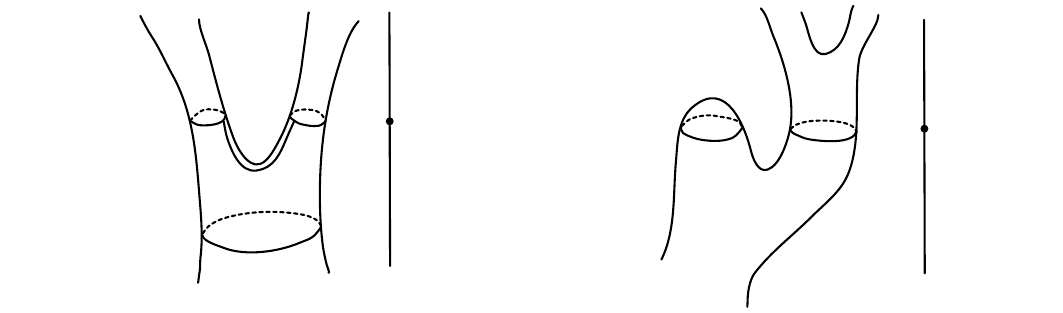}};
\draw[->] (2.8,0.5) -- (2.8,-0.5);
\draw[->] (-4.3, 0.3) -- (-4.3, 1.3);
\node at (-1.95,0.52) {$s_i$};
\node at (6.4,0.42) {$s_i$};
\end{tikzpicture}
\caption{The coordinates of the Stallings map locally as a Morse function.}
\label{fig: morse}
\end{figure}

After these modifications, $\{f_1^{-1}(s_i)\}$ is a set of disjoint essential simple closed curves, and there are no cancelling homotopic pairs. This means that these curves fall into at most $3g-3$ homotopy classes.

In the proof of Proposition~\ref{factorthrough}, we see that $f_1$ factors through a map $d: \mathcal{H}_1 \rightarrow \vee^g S^1$, and $d$ is part of a homotopy equivalence, together with a map $h: \vee^g S^1 \rightarrow \mathcal{H}_1$. This implies that   the curves $\{f_1^{-1}(s_i)\}$ are null homotopic when included into $\mathcal{H}_1$.  Therefore they bound (immersed) disks in $\mathcal{H}_1$. In fact, it is possible to find embedded disks that these curves bound. This will be proved in Proposition~\ref{propirred} and does not depend on any additional hypotheses.

Consider $h (\vee^g S^1)$. Each circle must be sent to a curve in $\mathcal{H}_1$ that intersects these disks in some pattern.  Modifying $h$ by a homotopy so that these intersections are transverse, ${d}_* \circ h_* (\sigma_i)$ can be read off as the intersection word of the $h$-image of the $i^{\text{th}}$ circle with the disks. Since $d$  and $h$ form a homotopy equivalence, ${d}_* \circ h_* (\sigma_i)= \sigma_i$. So, the $i^{\text{th}}$ circle must have intersection word that reduces to $\sigma_i$ in the free group $\pi_1(\vee^g S^1)$. However, since it is not possible for an arc to run between two distinct curves with intersection subword $\sigma_i \sigma_i^{-1}$, the only possibility is that $h$ can be modified by a homotopy so that each circle in  $h (\vee^g S^1)$ intersects the corresponding disk exactly once. This implies that $f^{-1}(s_i)$ contains a non-separating curve. If $f^{-1}(s_i)$ contained a second curve as well for some $i$, then there would be an arc $\gamma$ intersecting these two distinct curves with intersection word $\sigma_i \sigma_i^{-1}$. This is a contradiction. Therefore $f^{-1}(s_i)$ consists of a single non-separating curve.

Thus, there are exactly $g$ disjoint non-separating simple closed curves in $\{f_1^{-1}(s_i)\}$. Because these $g$ curves are null homotopic when included into $\mathcal{H}_1$, are homologically independent, and we know ${f_1}_*$ is $\pi_1$-surjective; from Proposition~\ref{factorthrough}, we know they must form a Heegaard diagram for the attachment of $\mathcal{H}_1$ to $\Sigma_g$.

To show that this half Heegaard diagram is unique, consider that if there was another half Heegaard diagram obtained in the way above, the curves of this diagram would be the preimages of $s_1, \ldots s_g$ by a map $\tilde f_1$ that is homotopic to $f_1$. Label the curves of this second diagram $\sigma_1, \ldots \sigma_g$. Since the two diagrams are not the same, there must be an $i$ such that the $\sigma_i$ two curves are not homotopic. Consider the homotopy between $f_1$ and $\tilde f_1$. The preimage of $s_i$ by this homotopy sweeps out a surface immersed in $\Sigma_g$, bounded by the two $\sigma_i$ curves. This implies that the two $\sigma_i$ curves are homologous. Two non-separating curves are homologous if and only if they are homotopic. This is a contradiction. Therefore the half Heegaard diagram obtained above is unique. 

In the same way, the second half of the Heegaard diagram is uniquely determined by $f_2$. We let this Heegaard diagram be $\mathscr{S}^{-1}([f])$.

\paragraph{Checking $\mathscr{S}^{-1}$ is surjective} Given a Heegaard diagram for $M = \mathcal{H}_1 \cup \mathcal{H}_2$, a map $f: \Sigma_g \rightarrow \vee^g S^1 \times \vee^g S^1$ that is sent to this Heegaard diagram by $\mathscr{S}^{-1}$ can be constructed. To construct the first coordinate $f_1$, attach disks and a $3$-ball to $\Sigma_g$ according to the blue Heegaard diagram curves to obtain $\mathcal{H}_1$. Then take a smooth structure on $\mathcal{H}_1$, and choose a tubular neighbourhood of the disks (each neighbourhood homeomorphic to $\mathbb{D} \times (-1,1)$ with $\mathbb{D} \times \{ 0 \}$ being the disk around which the neighbourhood is chosen).

Let $f_1$ send each disk to a distinct point in $\{s_1, s_2, \ldots s_g\}$. The tubular neighbourhood around each disk is mapped to the interval of $\vee^g S^1 \setminus \{b\}$ containing the corresponding $s_i$ by forgetting the first coordinate of $\mathbb{D} \times (-1,1)$, and applying a homeomorphism from $(-1,1)$ to the desired interval. The complement of the tubular neighbourhoods is mapped to the bouquet point $b$. The second coordinate, $f_2$ is constructed similarly.

Note that $\{f_j^{-1}(s_i)\}_{i, j}$ are the diagram curves and $\mathscr{S}^{-1}([f])$ is the Heegaard diagram we began with. the second coordinate $f_2$ is constructed similarly.

\paragraph{Checking $\mathscr{S}^{-1}$ is injective} Suppose $[f]$ and $[f']$ are in $B_{M, \mathcal{H}_1, \mathcal{H}_2}$ and that $\mathscr{S}^{-1}([f])=\mathscr{S}^{-1}([f'])$. Let $f \in [f]$ and $f' \in [f']$. Then after modifying $f$ and $f'$ by a homotopy if necessary, $\{(f')^{-1}(s_i)\}_i$ is a set of $g$ simple closed curves that homotopic to the $g$ curves $\{f^{-1}(s_i)\}_i$ for all $i$. After possibly composing $f$ with a homeomorphism $k: \vee^g S^1 \rightarrow \vee^g S^1$, we may assume that for each $i$, the curve $(f')^{-1}(s_i)$ is homotopic to $f^{-1}(s_i)$. Modifying $f$ further by homotopy, we may assume $f^{-1}(s_i) = (f')^{-1}(s_i)$ for all $i$.
%By choosing $k$ appropriately, we may further assume that for all $i$, if $\gamma$ is a path in $\Sigma_g$ transversally intersecting $f^{-1}(s_i) (=(f')^{-1}(s_i))$ once each, $f(\gamma)$ and $f'(\gamma)$ both pass through the point $s_i$ in the same direction in the $i^{\text{th}}$ circle of $\vee^g S^1$.

Since $\vee^g S^1\setminus \cup_i \{s_i\}$ is contractible, $f$ can be modified by a homotopy that is supported on the complement of $\cup_i f^{-1}(s_i)$ so that the resulting map sends every point outside a tubular neighbourhood of $\cup_i f^{-1}(s_i)$ to the bouquet point $b$. The components of the tubular neighbourhood of $\cup_i f^{-1}(s_i)$ are each homeomorphic to $S^1 \times (-1,1)$, and the modified $f$ sends the tubular neighbourhood of $f^{-1}(s_i)$ to the maximal interval in $\vee^g S^1$ containing $s_i$. The map $f'$ can be modified by a homotopy to have the same properties.

The modified maps $f$ and $f'$ now agree on the complement of a tubular neighbourhood of $\cup_i f^{-1}(s_i)$. It remains to show that they agree up to homotopy on each component of the tubular neighbourhood of $\cup_i f^{-1}(s_i)$. This amounts to showing that the path $\{\pi\} \times (-1, 1)$ in the tubular neighbourhood of $f^{-1}(s_i)$ is mapped by $f$ and $f'$ to wrap around the $i^{\text{th}}$ circle in $\vee^g S^1$ in the same direction. However, this can be guaranteed by an appropriate choice of $k$ earlier. Therefore $[f]=[f']$. 
\end{proof}

\subsubsection*{Handle slides and Stallings maps}
Given any two Heegaard diagrams for $M=\mathcal{H}_1 \cup \mathcal{H}_2$, there is a finite sequence of handle slides that can be applied to one to obtain the other. Now that we know that The set of Stallings maps $B_{M, \mathcal{H}_1, \mathcal{H}_2}$ is in bijection with the set of Heegaard diagrams for $M=\mathcal{H}_1 \cup \mathcal{H}_2$, one might wonder about the relationship between these Stallings maps and how it relates to the notion of handle slide.

The group $Out(F_g) \times Out(F_g)$ acts on $\mathcal{B}_{M, \mathcal{H}_1, \mathcal{H}_2}$ by the action $(a, b) \cdot f_1 \times f_2=a\circ f_1 \times b\circ f_2$ where $a$ and $b$ denote elements in $Out(F_g)$ as well as maps $\vee^g S^1 \rightarrow \vee^g S^1$ inducing these elements on the fundamental group (these maps must be part of homotopy equivalences). Note that up to homotopy, there is only one choice of the maps $a$ and $b$. One choice is made for each element of $Out(F_g)$ so that this action is well defined.

Note that any map that is part of a homotopy equivalence of $\vee^g S^1$ to itself determines an element of $Out(F_g)$ by its action on the fundamental group. Recall that any two maps in $\mathcal{B}_{M, \mathcal{H}_1, \mathcal{H}_2}$ differ by a homotopy equivalence of each of the two coordinates of $\vee^g S^1 \times \vee^g S^1$. Therefore, given $f, f' \in \mathcal{B}_{M, \mathcal{H}_1, \mathcal{H}_2}$, there is an element $(a,b)$ of $Out(F_g) \times Out(F_g)$ such that $(a,b) \cdot f$ is homotopic to $f'$. This induces a transitive action on the quotient $B_{M, \mathcal{H}_1, \mathcal{H}_2}$. The action is well defined because given $a \in Out(F_g)$, for every homeomorphism $k:\vee^g S^1 \rightarrow \vee^g S^1$, there is another homeomorphism $j:\vee^g S^1 \rightarrow \vee^g S^1$ such that the map  $a\circ f_1$ is homotopic to $j \circ (a \circ k \circ f_1)$. Namely, $j$ can be chosen to be a homeomorphism in the homotopy class of $a \circ k^{-1} \circ a^{-1}$ (recall $a$ is part of a homotopy equivalence, so there must be a homeomorphism in this class). 

Since $B_{M, \mathcal{H}_1, \mathcal{H}_2}$ is in bijection with the set of Heegaard diagrams for $M= \mathcal{H}_1 \cup \mathcal{H}_2$ (Theorem~\ref{heegaarddiagram}), the action of $Out(F_g) \times Out(F_g)$ must translate into an action by handle slides. Let $K'$ be the normal closure of the elements of $Out(F_g)$ that permute the generators or replace a generator with its inverse. Then let $K_1 \leq Out(F_g) \times Out(F_g)$ be the subgroup $K' \times \mathds{1}$ and let $K_2 \leq Out(F_g) \times Out(F_g)$ be the subgroup $\mathds{1} \times K'$. Then $K:=K_1 \cap K_2$ is the kernel of the action of $Out(F_g) \times Out(F_g)$ on $B_{M, \mathcal{H}_1, \mathcal{H}_2}$.\\

\begin{prop}
\label{onetoone}
The group $Out(F_g) \times Out(F_g) / K$ is isomorphic to $H\times H$ where $H$ is the handle slide group acting on a handle body of genus g. The actions of $Out(F_g) \times Out(F_g) / K$ on $B_{M, \mathcal{H}_1, \mathcal{H}_2}$ and $H\times H$ on the set of Heegaard diagrams of a Heegaard splitting of genus $g$ are equivariant with respect to the bijection $\mathscr{S}$ via this isomorphism.
\end{prop}

\begin{proof}
The handle slide group $H$ is generated by individual handle slides. Beginning with a generator $(a,b)$ of $Out(F_g) \times Out(F_g)$, we would like to show that the map $\mathscr{S}([f]) \mapsto \mathscr{S}((a,b) \cdot [f])$ acts on the set of Heegaard diagrams in one of two ways on each set of $g$ curves. It either as the identity in $H$, or it acts as a single handle slide. We will see that this defines a homomorphism $Out(F_g) \times Out(F_g) \rightarrow H \times H$, and show that its kernel is $K$. As with most proofs in this paper, we prove this statement for one coordinate, and the second coordinate follows similarly. 

Let $j$ be an automorphism of the free group $F_g$. Then $j$ induces a homotopy equivalence of the bouquet of $g$ circles to itself. In the case where $j$ is an inner automorphism, the homotopy equivalence is homotopic to identity. If not, $j$ represents an element of $Out(F_g)$. The group $Out(F_g)$ is generated by automorphisms that send a single generator to its inverse, and act on the other generators as identity, as well as automorphisms that send a generator to the product of two generators and act on the other generators as identity. The first type fixes every element of $B_{M, \mathcal{H}_1, \mathcal{H}_2}$, and therefore the Heegaard diagram obtained is unchanged as well. If $j$ is of the second type, its action on $f_1$ can be realized gradually using a path $p(t)$ in the outer space $\mathcal{O}_g$ such that $p(0) \circ f_1 = f_1$ and $p(1) \circ f_1 = j \circ f_1$. Intuitively, $p(t)$ folds one circle around another as shown in the comic strip in Figure~\ref{fig: handleslide}.  

\begin{figure}[h]
\centering
\includegraphics[width= 6 in]{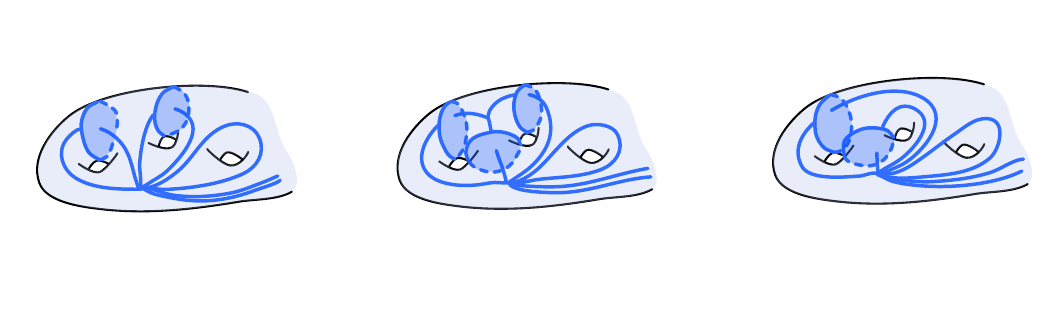}
\caption{A generator of $Out(F_g)$ is shown inducing a handle slide}
\label{fig: handleslide}
\end{figure}

At each moment in time $t$, Let $Y(t)$ be the target space of $p(t) \circ f_1$. Choose points $\{s_i \}$ one at the midpoint of each edge of $Y(t)$. Take the preimage of these points and simplify by a homotopy as in the proof of Proposition~\ref{heegaarddiagram}, so that there is a single curve in the preimage of each $s_i$. When $t\neq 0$ or $1$, one connected component $P$ of $Y(t) \setminus \cup_i \{s_i\}$ contains a single trivalent vertex. The preimage of $P$ must be a subset of $\Sigma_g$ bounded by three curves in the preimage of $\{s_i\}$. This surface cannot have positive genus because $p(t)$ is a homotopy equivalence. If the surface had positive genus, there would be two intersecting degree one homology classes that have trivial image by $p(t) \circ f_1$. This contradicts Property (1) of Stallings maps. Therefore, $P$ must be a pair of pants. When $t=0$ or $1$, the preimage of $\{s_i\}$ are Heegaard diagrams (as in Proposition~\ref{heegaarddiagram}) and one of the cuffs of the pair of pants $P$ disappears. This is exactly the combinatorial move known as a handle slide. 

This defines a homomorphism $\varphi: Out(F_g)  \rightarrow H$. A composition of generators of $Out(F_g)$ may act by permuting the generators. These compositions must have trivial image by $\varphi$ since the set of midpoints of edges is unchanged by a permutation of the copies of $S^1$. Thus the kernel of $\varphi$ is $K'$. Then $\varphi \times \varphi$ induces an isomorphism between the quotient $Out(F_g) \times Out(F_g) / K$ and $H \times H$, as desired.
\end{proof}

\subsection{A picture of Stallings maps}
We would like to study the nature of these maps. It is natural to ask whether they can be homeomorphisms (thinking of these as the nicest possible map between two-dimensional spaces). We will see in the final section about Lens spaces that there is one Heegaard splitting that can be encoded as a homeomorphism between tori (the genus one Heegard splitting of $S^3$). However, outside the genus one case, this is not possible, since the target space is no longer a manifold.

It is valuable to spend a moment understanding the topology of a cartesian product of two bouquets of $g$ circles. There are three types of points: manifold points (which have neighbourhoods homeomorphic to discs), book-like points (which have neighbourhoods homeomorphic to several disks intersecting in a curve), and the singular point (which is the product of the two wedge points). All together this forms a non-positively curved square complex with one vertex, $2g$ edges and $g^2$ squares. It is comprised of many one-square tori that intersect pairwise in an edge or a vertex. The link of the vertex is homeomorphic to the complete bipartite graph $K_{2g, 2g}$. Note that this graph is not planar, and so $(\vee^g S^1) \times (\vee^g S^1)$ cannot be embedded in $\mathbb{R}^3$.\\

\begin{figure}[h]
\centering
\includegraphics[width= 5 in]{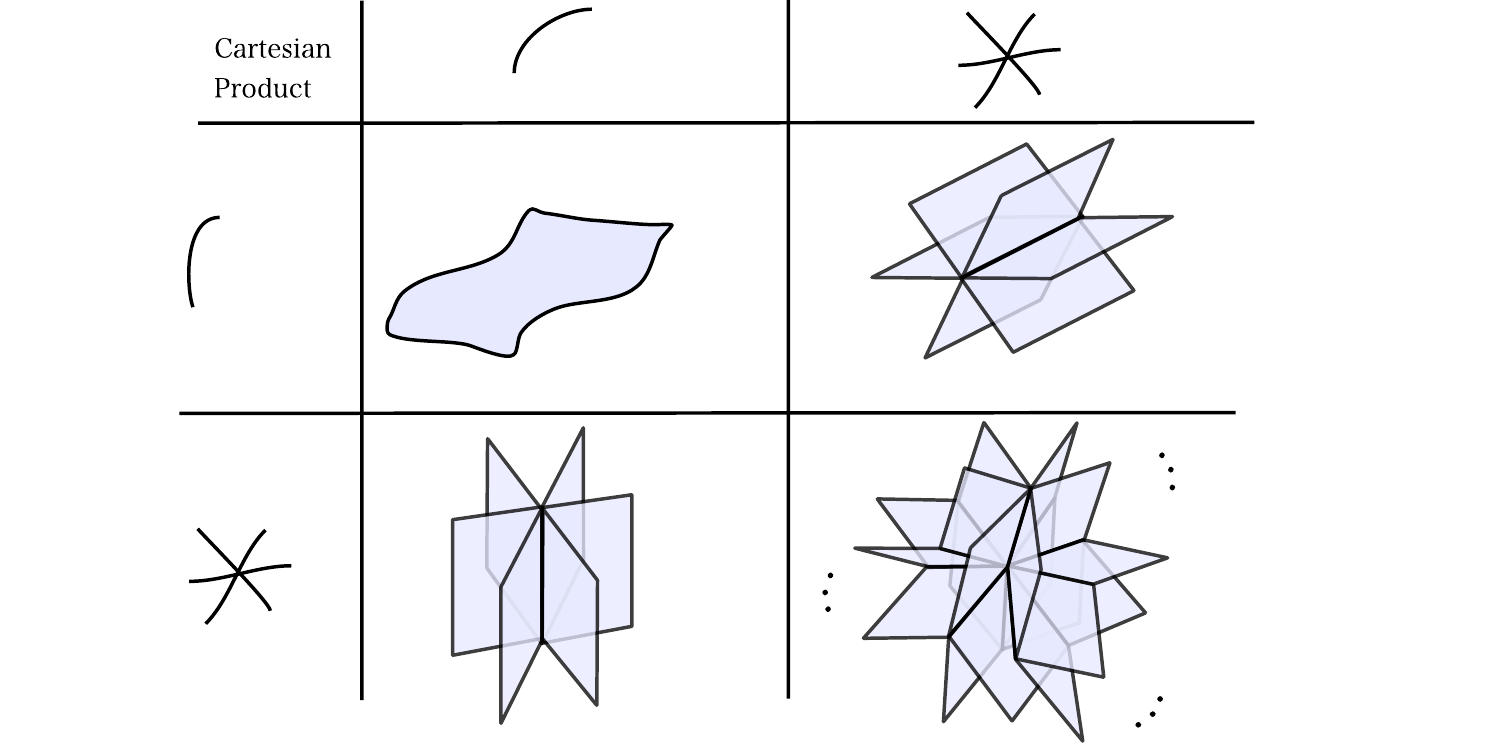}
\caption{topology of product neighbourhoods of $G_1 \times G_2$.  The bottom left neighbourhood cannot be embedded in $\mathbb{R}^3$, due to additional square ``tiles" that are not shown here. The ellipses indicate that these are missing.}
\label{fig: table}
\end{figure}

%%%%%%%%%%%%%%%%%%%%%%
%%%%%%%%%%%%%%%%%%%%%%%%
\section{Irreducibility, immersions and square complexes}
\label{mainthm}

We will now focus on the case where the Heegaard splitting is irreducible. Here, is it possible to pull our maps tight in some sense, removing folding and crumpling, and this gives an even more concrete picture of the maps, allowing understanding of the reducible case through decomposition by connect sum. Some set up is needed to formalize and prove this result. We begin with some definitions\\

\begin{defn}
Consider an embedded curve $\gamma$ on $\Sigma_g$. Take an open tubular neighbourhood $T$ of $\gamma$ of width small enough that $T$ is homeomorphic to an open cylinder. Consider a tubular neighbourhood $U$ of $\gamma$ that is contained in $T$. Removing the image of $\gamma$ from $U$ leaves two connected components. One of these components will be called $A_U$ and the other will be called $B_U$. Now we define an equivalence relation on the set of tubular neighbourhoods of $\gamma$ contained in $T$. Any two components of (possibly distinct) deleted tubular neighbourhoods are equivalent if and only if their intersection is non-empty. Notice that there are two equivalence classes and $A_U$ and $B_U$ are in different classes for any given $U$. Call these two classes the two \textit{sides of} $\gamma$.\\

A set \textit{intersects a side of } $\gamma$ if it intersects all the sets in this equivalence class. Consider an arbitrary constant curvature metric on $\Sigma_g$. Let $U_{\epsilon}$ be the tubular neighbourhood of $\gamma$ of width $\epsilon$ (with $\epsilon$ small enough so that $U_{\epsilon}$ is contained in $T$). Then  the number of \textit{connected components} of the intersection of a set $X$ with a side of $\gamma$ is defined as the limit  of the number of connected components of $X\cap A_{U_{\epsilon}}$ as $\epsilon \rightarrow 0$ (respectively $B_{U_{\epsilon}}$, depending on which side).
\end{defn}
%%%%%%%%%%%%%%%%%%%%%%%%%%%%%%%%%%
%%%%%skip definition of corner%%%%%%%%%%%%%%
%A corner is similarly defined for a cell decomposition of a surface
%\begin{defn}
%Given a cell decomposition of the surface $\Sigma_g$, and an arbitrary constant curvature metric, let $d$ be less than half of the minimum distance between a pair of vertices. Consider disk neighbourhoods of each vertex $v$, of radius less than $d$. Let $D_{v, \epsilon}$ be such a disk of radius $\epsilon <d$. Define an equivalence relation on connected components of intersections of cells with these disk neighbourhoods. If $c$ is an open 2-cell, a connected component of $c \cap D_{v, \epsilon}$ is equivalent to a connected component of $c \cap D_{v', \epsilon'}$ iff  the two components intersect. The equivalence classes are called \textit{corners} of $c$. We make a similar definition for edges.  If $e$ is an open 1-cell, a connected component of $e \cap D_{v, \epsilon}$ is equivalent to a connected component of $e \cap D_{v', \epsilon'}$ iff  the two components intersect.

%A set \textit{intersects a corner} if it intersects all the sets in this equivalence class.
%\end{defn}

A wave move  is a combinatorial change that one can make to a Heegaard diagram without changing the $3$-manifold it describes. The proof of this is originally written by Volodin, Kuznetzov and Fomenko in 1974 \cite{wavemove}, but is explained below.

\begin{defn}Given a half Heegaard diagram on the surface $\Sigma_g$, a \textit{wave move} consists of the following steps.
	\begin{enumerate}
	\item Pick a simple arc $\gamma$ that begins and ends on a diagram curve $\delta$, does not intersect any other diagram curves, and intersects only one side of $\delta$ in two components;
	\item Thicken the union of $\delta$ and $\gamma$ in $\Sigma_g$ to obtain a pair of pants (see Figure~\ref{wavemove});
	\item Replace $\delta$ with another cuff of this pair of pants (if it creates a new Heegaard diagram).
	\end{enumerate}
\end{defn}

\begin{figure}
\centering \includegraphics[width=3in]{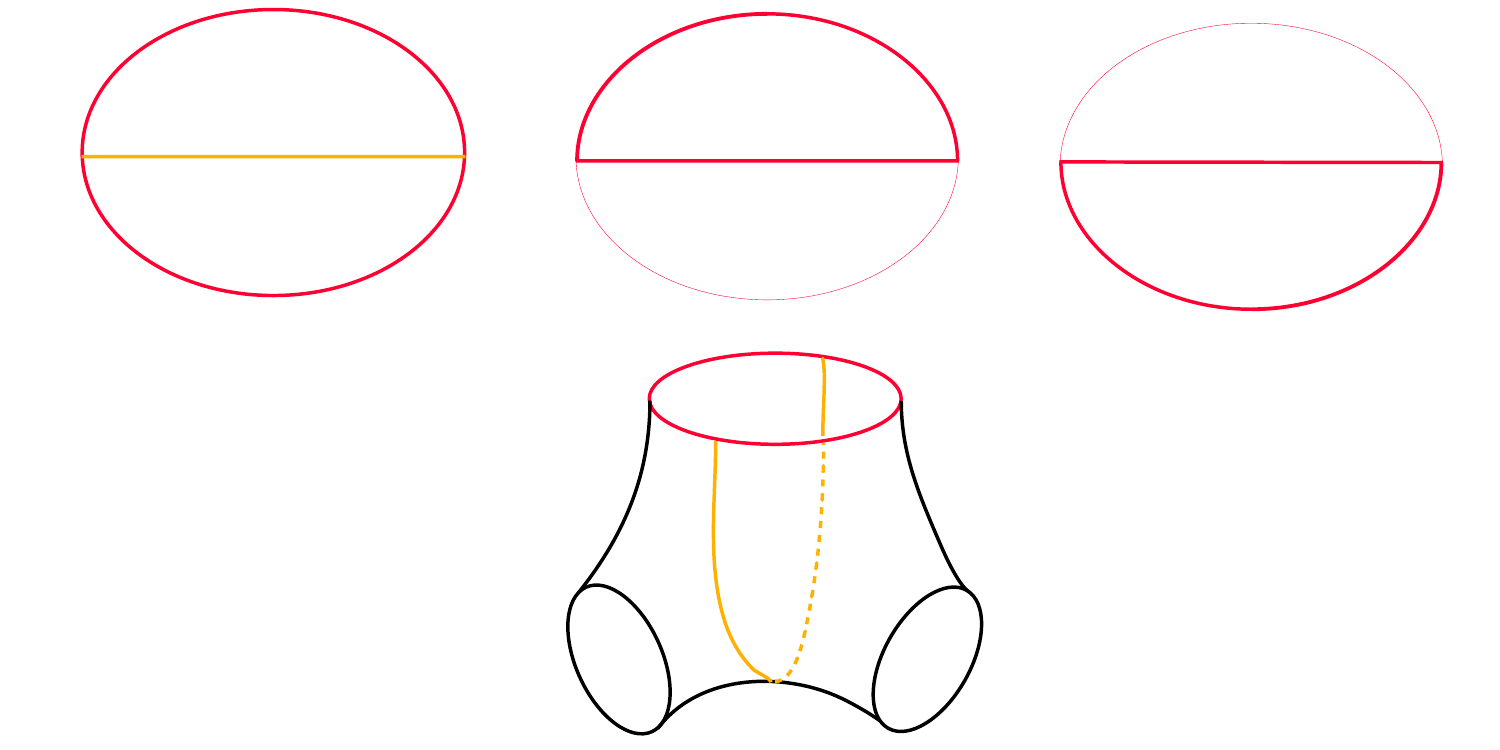}
\caption{The diagram curve $\delta$ is shown in red and the arc $\gamma$ is shown in yellow. Representatives of the free homotopy classes $[\alpha]$ and $[\beta]$, shown in two ways: (i) as a union of segments of $\delta$ and $\gamma$, and (ii) as the cuffs of a pair of pants constructed by $\delta$ and $\gamma$.}
\label{wavemove}
\end{figure}

Call the two curves forming the other cuffs of the pair of pants $\alpha$ and $\beta$. We shall see in the proof of the next theorem that either replacing $\delta$ with $\alpha$ or $\delta$ with $\beta$ always results in another Heegaard diagram. Further, we shall see that the new Heegaard diagram describes the same Heegaard decomposition. 

\begin{thm}[Volodin, Kuznetzov, Fomenko \cite{wavemove}] \label{thm:wavemove}
Applying a wave move to a half Heegaard diagram does not change the attachment of a handle body described.
\end{thm}
\begin{proof}
To show that one of these swaps must give a new Heegaard diagram for the same Heegaard decomposition, two cases must be explored combinatorially.\\

\paragraph{Case 1: $\alpha$ or $\beta$ is isotopic to a diagram curve.} If $\alpha$ is isotopic to another curve in the original half Heegaard diagram, then $\beta$ must not be, since $\delta$ must be homologically independent from the other curves. Then, by a single handle slide, $\delta$ can be replaced by $\beta$ to give a new Heegaard diagram, describing the same attachment of a handlebody to $\Sigma_g$. Similarly, if $\beta$ is isotopic to a curve in the original half Heegaard diagram, then $\delta$ can be replaced with $\alpha$.\\

\paragraph{Case 2: Neither $\alpha$ nor $\beta$ is isotopic to a diagram curve} Assume neither $\alpha$ nor $\beta$ is isotopic to a curve in the original half Heegaard diagram. In this case, cut $\Sigma_g$ along all the diagram curves, and label them to recall the identifications needed to recover $\Sigma_g$. This is a sphere with $2g$ discs removed. Note that, by the way it is defined, $\gamma$ does not intersect any diagram curves. Therefore, the graph given by the union of $\delta$ and the $\gamma$ can be embedded in the sphere with $2g$ discs removed. By thickening this graph in this space, the pair of pants bounded by $\delta$, $\alpha$, and $\beta$ are embedded as well. 

\begin{figure}
\centering \includegraphics[width=3in]{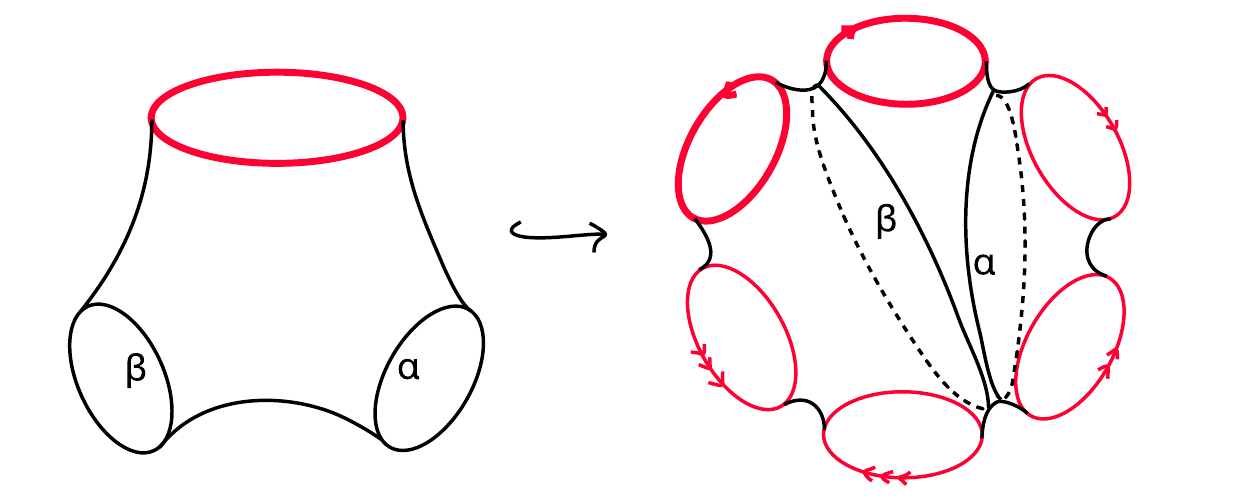}
\caption{The embedding of the pair of pants introduced in Figure \ref{wavemove}.}
\end{figure}

Either $\alpha$ or $\beta$ must separate one copy of $\delta$ from the other. Without loss of generality let it be $\beta$. The aim is to show that $\delta$ can be replaced by $\beta$ to give a new Heegaard diagram associated to the same attachment of a handle body.\\

This will be shown by constructing a sequence of embedded pairs of pants that dictate a sequence of handle slides that ultimately make the desired replacement.\\

Cut along $\beta$. This produces two connected components. Pick either one of these to work with. Call this space $S$. We will embed our pairs of pants that dictate handle slides in this space. Identify as many of the diagram curves as possible to each other. Each identification creates a handle with a diagram curve wrapped around it.

\begin{figure}
\centering \includegraphics[width=2in]{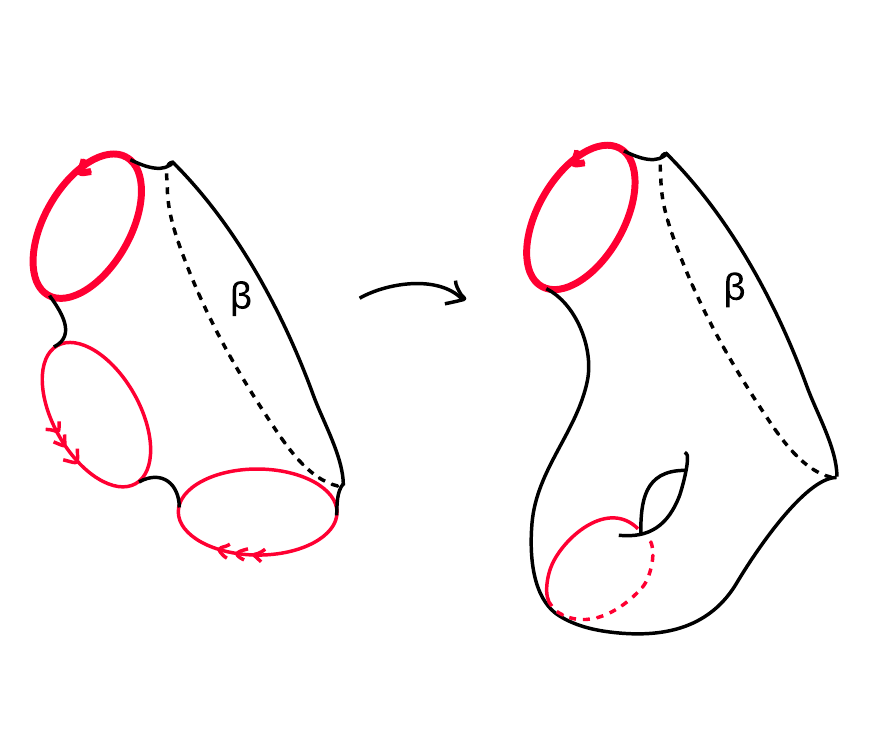}
\caption{$S$ is shown before and after identifications.}
\end{figure}

In general $S$ is now a surface of genus $p$ with $n+2$ boundary components ($n$ boundary components besides $\delta$ and $\beta$). Further, $S$ has a diagram curve around each handle, and isotopic to each of the boundary components except for $\beta$. Using the classification of surfaces, the surface $S$ and the diagram curves are homeomorphic to the surface and curves shown in Figure~\ref{pantsdecomp}. Consider the following pair of pants decomposition by curves: $\delta,\gamma_1 \ldots, \gamma_{2p}, \ldots, \gamma_{2p+n} = \beta$.

\begin{figure}
\centering \includegraphics[width=4in]{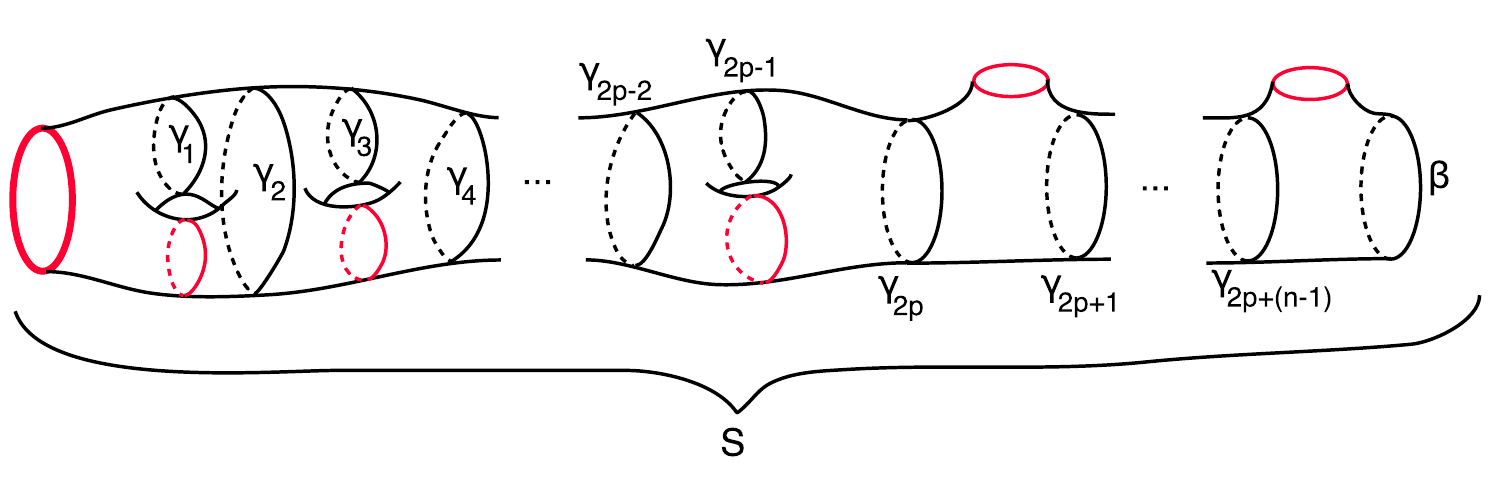}
\caption{A pair of pants decomposition of $S$.}
\label{pantsdecomp}
\end{figure}

By a handle slide, the problematic curve can be replaced by $\gamma_1$ to give a new Heegaard diagram associated to the same Heegaard decomposition. A second handle slide allows $\gamma_1$ to be replaced by $\gamma_2$. Next $\gamma_2$ is replaced with $\gamma_3$, and so on until $\gamma_{2p+(n-1)}$ is replaced with $\beta$, as desired.

\end{proof}

The author rediscovered the wave move for herself during her graduate studies, and found that it is incredibly useful! It is the key to almost every proof in this paper.
%In fact, the proof of Proposition~\ref{heegaarddiagram} seen earlier, uses a slight modification of the wave move.

\subsection{Irreducibility}

Recall that  Heegaard splitting is \emph{irreducible} if no essential simple closed curve on the splitting surface bounds a disk in $\mathcal{H}_1$ and in $\mathcal{H}_2$.

\begin{prop} \label{propirred}
Let $\mathcal{H}$ be a handle body and let $\Sigma$ be the surface that is its boundary. Let $\gamma$ be an essential simple closed curve on $\Sigma$, that is nullhomotopic in $\mathcal{H}$. Then $\gamma$ bounds an embedded disk in $\mathcal{H}$.
\end{prop}

\begin{proof}
Choose half a Heegaard diagram, that describes the attachment of $\mathcal{H}$ to $\Sigma$. Since $\gamma$ is nullhomotopic in $\mathcal{H}$, the cyclic word produced by its intersection with the diagram curves must reduce to identity in the free group on $g$ generators. Thus, the arcs of $\gamma$ that run between diagram curves dictate a number of wave moves that change the diagram so that it does not intersect $\gamma$ at all, while still describing the attachment of $\mathcal{H}$ to $\Sigma$.

Cut $\mathcal{H}$ along the embedded disks bounded by diagram curves. Note that $\gamma$ is not cut since it does not intersect any diagram curves. The remaining space is homotopic to a closed $3$-ball with $2g$ disks removed on the boundary. $\gamma$ is an embedded curve on the boundary of this $3$-ball. It is clear that $\gamma$ bounds an embedded disk in this space. Identifying pairs of disks to recover $\mathcal{H}$, we see that $\gamma$ bounds an embedded disk in $\mathcal{H}$.
\end{proof}

\begin{rmk}
Proposition \ref{propirred} shows that irreducibility of a Heegaard splitting is equivalent to the following: no essential simple closed curve in $\Sigma$ is nullhomotopic in both $\mathcal{H}_1$ and $\mathcal{H}_2$. This property is much more easily verified than the definition, and will be used extensively in what follows.
\end{rmk}

\begin{defn}
Let $S$ be a closed orientable surface and let $G_1$, $G_2$ be graphs. A Stallings map $f: S \longrightarrow G_1 \times G_2$ is said to be \textit{irreducible} if it encodes an irreducible Heegaard decomposition.
\end{defn}

\begin{rmk}
Proposition \ref{propirred} implies that if $f$ is an irreducible map, then $\text{ker}f_*$ has no non-trivial element with a representative that is a simple closed curve.
\end{rmk}

\begin{notation}
The traditional picture of a Heegaard diagram is of red and blue curves.  Following this tradition, we will refer to the two sets of curves as though one set is red and the other is blue.
\end{notation}

\begin{prop} \label{filling}
Let $M=\mathcal{H}_1 \cup \mathcal{H}_2$ be an irreducible Heegaard decomposition. The complement of the curves of any associated Heegaard diagram in $\Sigma_g$ is homeomorphic to a finite union of discs. In other words, the Heegaard diagram fills the surface.
\end{prop}

\begin{proof}

Consider the handle body $\mathcal{H}_i$ and the $g$ curves of a Heegaard diagram that encode its attachment to $\Sigma_g$. Notice that the complement of these curves in $\Sigma_g$ is sent to a null-homotopic subset of $\mathcal{H}_i$ by $\iota_i$. 

%\begin{figure}
%\centering \includegraphics[width=3in]{halfdiagram.pdf}
%\caption{Half of a Heegaard diagram.}
%\end{figure}

Now suppose the proposition were false. That is, suppose that there was a Heegaard diagram for $M=\mathcal{H}_1 \cup \mathcal{H}_2$ that has a complement with a connected component that is not a disk. This implies that there is an essential simple closed curve in the complement of the diagram. Since it does not intersect any blue curves, it must have null-homotopic image by $\iota_1$ and since it does not intersect any red curves, it must have null-homotopic images by $\iota_2$. This contradicts the irreducibility of of $M=\mathcal{H}_1 \cup \mathcal{H}_2$.
\end{proof}

\begin{notation}
The discs described above can be thought of as \emph{polygons} with alternating red and blue edges made from segments of Heegaard diagram curves.
\end{notation}

\subsection{Tautness}

\begin{defn}
 \label{taut}
A Heegaard diagram is \textit{taut} if every connected component of the complement intersects a side of a diagram curve in at most one connected component.
\end{defn}

A region of $\Sigma_g$ is shown in Figure~\ref{problemandnoproblem}, with two complementary polygonal components of two different Heegaard diagrams. The first picture is not taut. The open polygon intersects the same side of a curve in two connected components. The second picture is taut. Here, while the open polygon does intersect the sides of a curve in two connected components, they are the opposite sides.

\begin{figure}
\centering \includegraphics[width=5in]{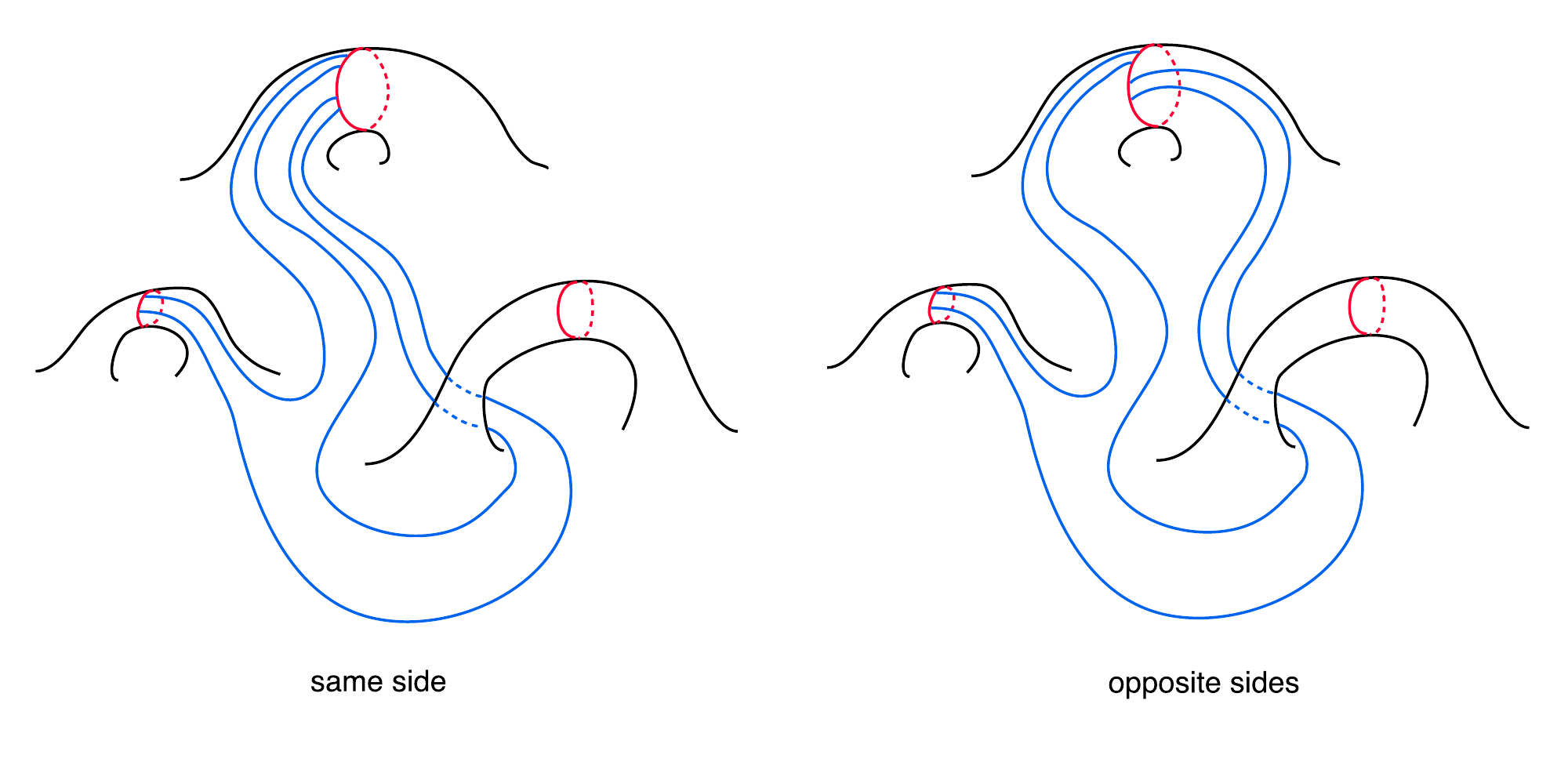}
\caption{Two regions of Heegaard diagrams. The left is not taut, while the right is taut.}
\label{problemandnoproblem}
\end{figure}

\begin{thm}
\label{thm:existtaut}
Every Heegaard decomposition has at least one associated taut Heegaard diagram.
\end{thm}

\begin{proof}
The proof is constructive, beginning with any Heegaard diagram and modifying it until a taut Heegaard diagram is achieved.\\

If this Heegaard diagram is already taut, we are finished.\\

If not, there exists a connected component of the complement that intersects some diagram curve twice on the same side. Such incidences will be called \textit{problems}.
%By Proposition~\ref{filling}, all connected components of the complement must be open topological disks. Since the boundary is formed by a finite number of segments of Heegaard diagram curves, we will refer to these disks as polygons.
Notice that multiple problems may occur on a single complementary component, or along a single curve.\\

Given a problem, without loss of generality, let the problematic curve be red and call it $\delta$. Consider a curve segment $\gamma$ with one end point on each of the two segments shared by the problematic curve and problematic component, and with interior inside the problematic component. Such a segment is shown in Figure~\ref{problem} for the problem first seen in Figure~\ref{problemandnoproblem}.

\begin{figure}
\centering \includegraphics[width=2in]{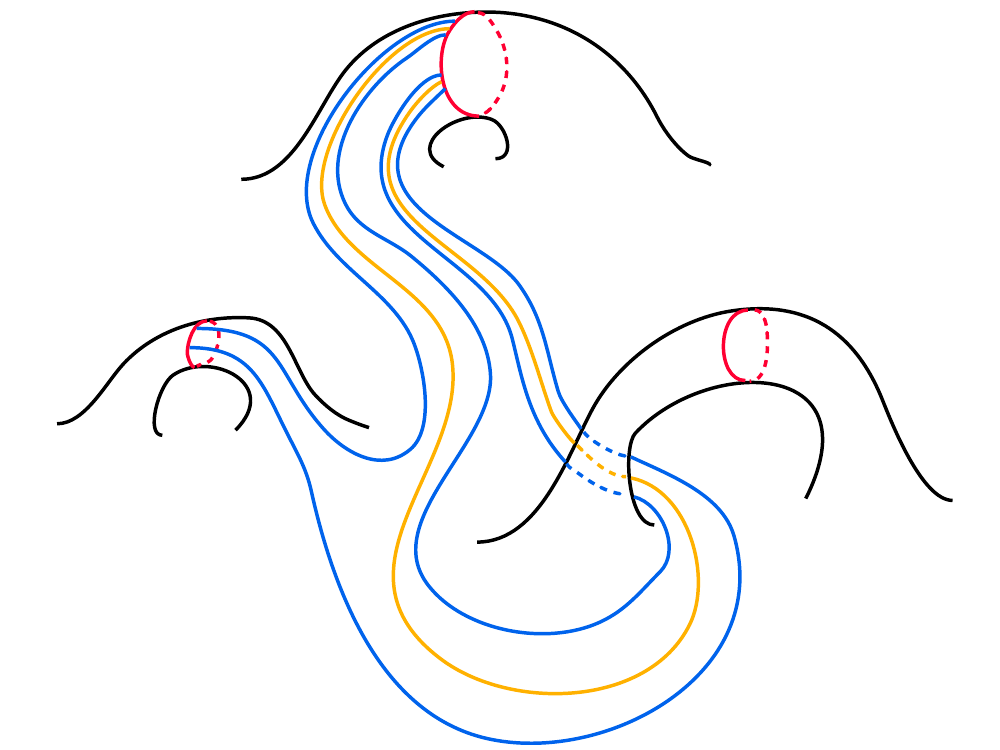}
\caption{An orange segment that intersects the same side of a diagram curve in two connected components.}
\label{problem}
\end{figure}

Thicken the union of $\delta$ and $\gamma$ in $\Sigma_g$ to obtain a pair a pants. One cuff of the pair of pants is isotopic to $\delta$, call the other two cuffs $\alpha$ and $\beta$. Replace the problematic curve $\delta$ with $\alpha$ or $\beta$ to produce a new Heegaard diagram. This Heegaard diagram describes the same Heegaard splitting by Theorem~\ref{thm:wavemove}, since we have just applied a wave move.\\

This swap strictly reduces the number of intersections between red and blue curves.\\

Figure~\ref{reduceintersection} includes the minimal number of blue arcs possible so that the complementary region containing $\gamma$ intersects a side of $\delta$ in two connected components . The elimination of an intersection point is seen in either case.

\begin{figure}
\centering 
\begin{tikzpicture}
\node (image) at (0,0) {\includegraphics[width=4in]{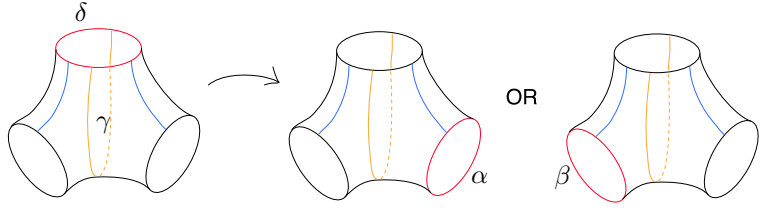}};
\node at (-4,1.5) {$\delta$};
\node at (1.3,-0.7) {$\alpha$};
\node at (2.4,-0.7) {$\beta$};
\node at (-3.7,0) {$\gamma$};
\end{tikzpicture}
\caption{Applying wave move to the red half of the Heegaard diagram, seen in the pair of pants.}
\label{reduceintersection}
\end{figure}

\begin{figure}
\centering \includegraphics[width=4in]{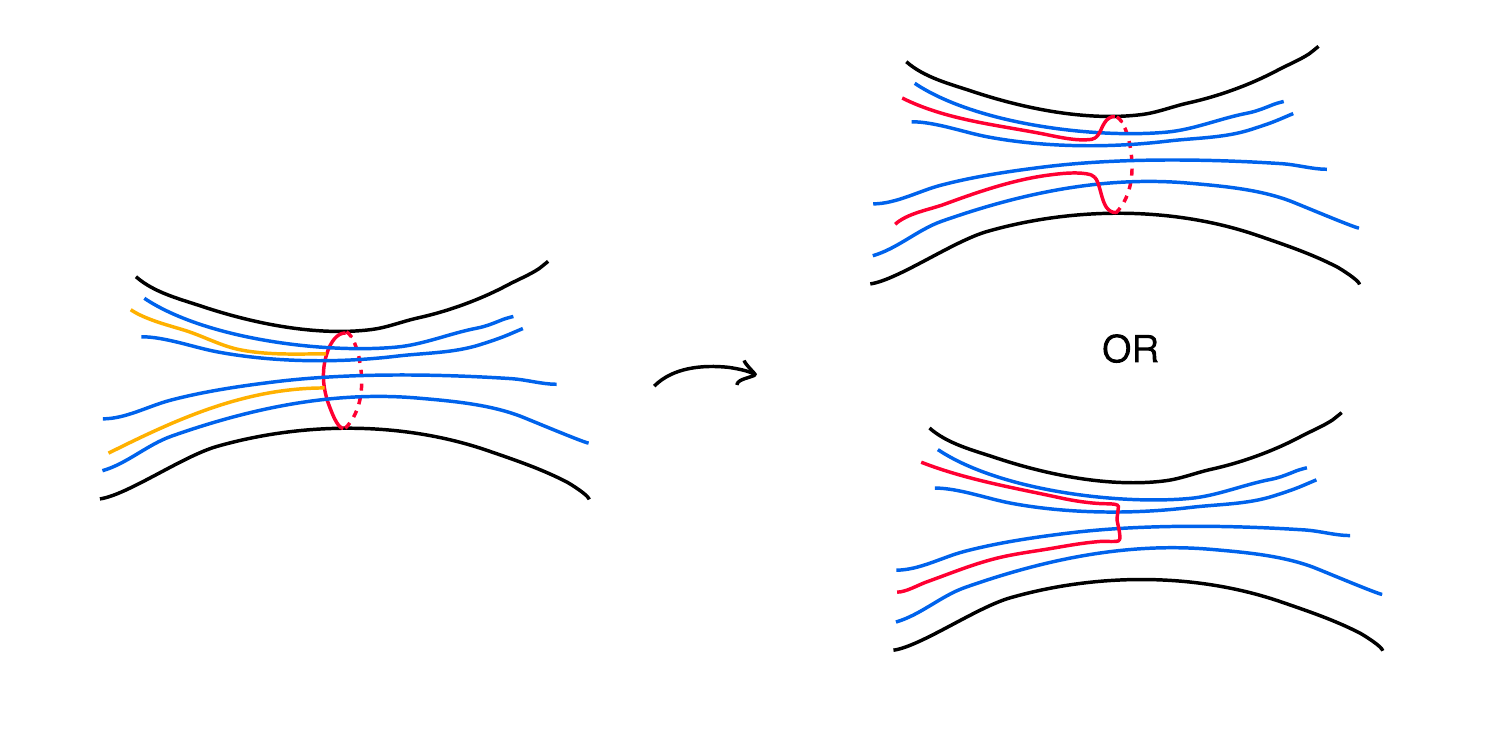}
\caption{A neighbourhood of the problematic curve $\delta$ before and after the wave move, this time with a few more blue arcs.}
\end{figure}

Repeat this process as long as there is a problem, each time reducing the intersections between red and blue curves by at least one. The number of intersections is finite and bounded below by zero. Therefore, this process must end with a taut configuration.
\end{proof}
 
 Applying Theorem~\ref{thm:existtaut} to irreducible Heegaard splittings, where every Heegaard diagram is filling (Proposition~\ref{filling}) gives the following result immediately.
 
 \begin{thm}
 \label{irredtautfilling}
 Every irreducible Heegaard splitting has a taut and filling Heegaard diagram.
 \end{thm}

\subsection{Immersions}
The goal of this section is to prove the following theorem.

\begin{thm} \label{locinj}
There is a locally injective map $f: \Sigma_g \longrightarrow (\vee^g S^1) \times (\vee^g S^1)$ encoding every irreducible Heegaard decomposition of genus greater than zero.
\end{thm}

This theorem is a corollary of the following lemma, which we will prove in detail.
\begin{lemma}
\label{tautfilling}
The Stallings map corresponding to a taut and filling Heegaard diagram without bigons is locally injective.
\end{lemma}

Taking Lemma~\ref{tautfilling} to be true, all Heegaard diagrams of irreducible Heegaard splittings are filling by Proposition~\ref{filling}, and we have seen in Theorem~\ref{irredtautfilling} that there is a taut diagram for each irreducible decomposition. Therefore Theorem~\ref{locinj} follows.

To prove Lemma~\ref{tautfilling}, consider a taut, filling Heegaard diagram without bigons for some Heegaard decomposition $M=\mathcal{H}_{1} \cup \mathcal{H}_{2}$. Let the curves describing the attachment of $\mathcal{H}_{1}$ to $\Sigma_g$ be blue. Similarly, let the curves describing the attachment of $\mathcal{H}_{2}$ to $\Sigma_g$ be red.\\

This diagram will be used to give a particular map $f: \Sigma_g \longrightarrow (\vee^g S^1) \times (\vee^g S^1)$. We will see that $f$ is locally injective and then the proof of Proposition~\ref{tautfilling}, will be reduced to checking that $f$ describes $M=\mathcal{H}_{1} \cup \mathcal{H}_{2}$.\\

\paragraph{Defining $f$:}
In the target, $(\vee^g S^1) \times (\vee^g S^1)$, consider each coordinate $(\vee^g S^1)$.
For simplicity (and foreshadowing), call one the blue coordinate and the other the red coordinate. After removing the wedge point of $\vee^g S^1$, $g$ open intervals remain. Orient these intervals arbitrarily. In the blue copy of $\vee^g S^1$, Pick $g$ points, one at the midpoint of each of these intervals. Call these blue points. Pick a bijection between the set of blue points and the set of blue diagram curves. Repeat the process on the red coordinate, and call the chosen points red points. Now orient the diagram curves on $\Sigma_g$ arbitrarily. \\
	
\begin{figure}
\centering \includegraphics[width=4in]{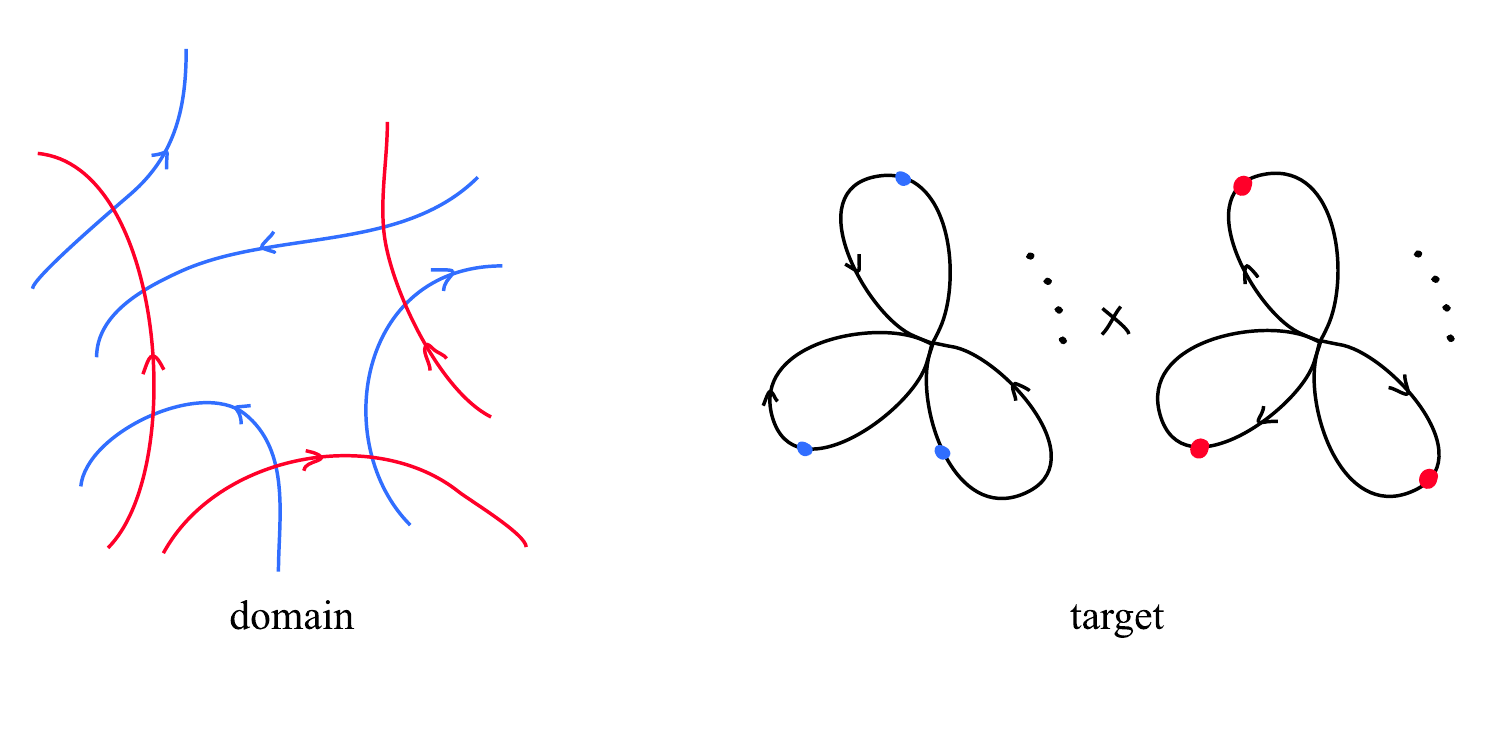}
\caption{A region of $\Sigma_g$ is shown with orientations chosen on the Heegaard diagram curves. The two coordinates of $(\vee^g S^1) \times (\vee^g S^1)$ are shown with blue and red points chosen as well as orientations on the edges.}
\end{figure}

Recalling the definition of ``side," there is now a right side (and left side) of each curve using these orientations. In the bouquet of circles, there is also a right side of each of the points in the intervals picked out in the preceding paragraph.\\

Due to Proposition \ref{filling}, this Heegaard diagram induces a cell decomposition structure on $\Sigma_g$. Every vertex in this cell decomposition is quadrivalent. Consider the dual cell decomposition. Every face in the dual decomposition must have four edges (it is a cell decomposition of squares). We metrize $\Sigma_g$ so that these dual $2$-cells are squares with the diagram curves intersecting the midpoint of the edges of each square. Each square has one blue curve and one red curve running through it, cutting it precisely into quarters. The two diagram curves that run through a given square are associated to two points: one in the blue coordinate $\vee^g S^1$ and one in the red coordinate $\vee^g S^1$. We would like to construct a map $f: \Sigma_g \longrightarrow (\vee^g S^1) \times (\vee^g S^1)$. The idea is that $f$ sends the interior of each square homeomorphically to the product of maximal intervals which contain these two points. The interior of the squares are mapped by $f$ such that the right side of the red curve is sent to the right side of the red point and the right side of the blue curve is sent to the right side of the blue point.\\
	
\begin{figure}
\centering \includegraphics[width=5in]{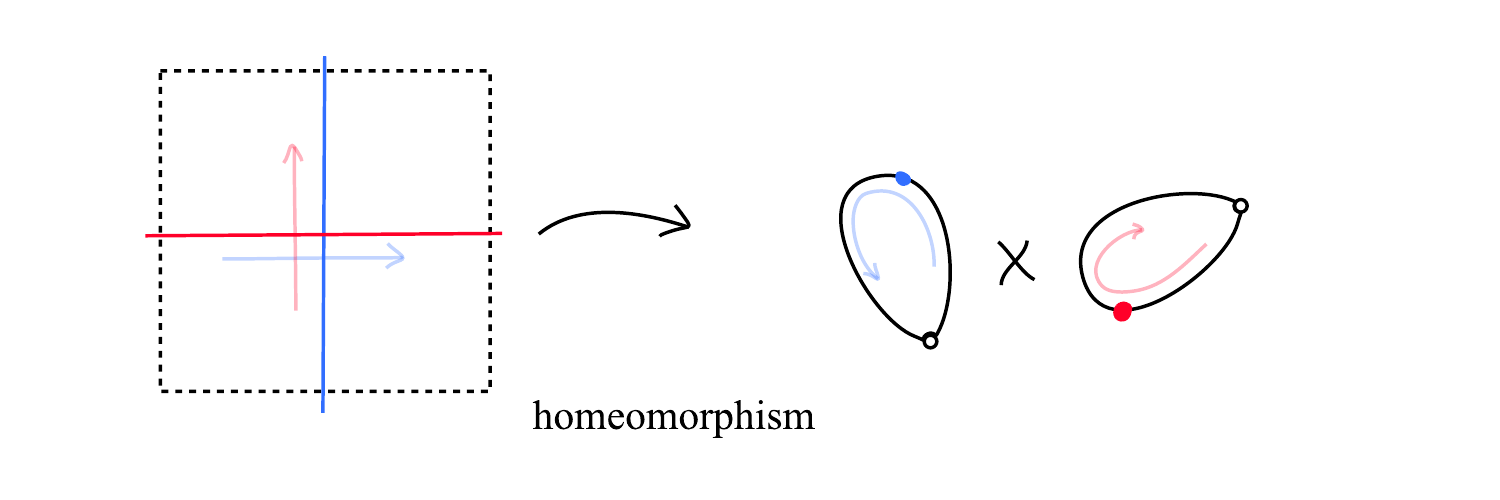}
\caption{The definition of $f$ on the interior of a square.}
\end{figure}

Now we say precisely how to set this up. Consider $f$ as two coordinates $f=f_1 \times f_2$, where $f_i:\Sigma_g \longrightarrow \vee^g S^1$. Each of the $1$-cells in the square cell decomposition of $\Sigma_g$ intersects exactly one diagram curve in one point. If this diagram curve is blue, let $f_1$ send the open $1$-cell to the maximal interval subset of $\vee^g S^1$ that contains the blue point assigned to the diagram curve. This should be done in such a way so that the half interval that intersects the right side of the diagram curve is sent to the right of the blue point. Let the boundary of this $1$-cell be sent to the wedge point. If the diagram curve is red, let $f_2$ send the open $1$-cell to the maximal interval subset of $\vee^g S^1$ that contains the red point assigned to the diagram curve, with the boundary of the $1$-cell being sent to the wedge point. Now consider a closed $2$-cell in the square cell decomposition of $\Sigma_g$. $f_1$ and $f_2$ are defined on alternating boundary edges of this cell. To define $f_1$, we consider the square as $[0,1] \times [0,1]$, where $f_1$ is defined on $[0,1] \times \{0\}$ and on  $[0,1] \times \{1\}$. The rest of the square is sent by $f_1$ to a straight-line homotopy between the definition(s) of $f_1$ on  $[0,1] \times \{0\}$ and  $[0,1] \times \{1\}$. The map $f_2$ is defined on the closed $2$-cell similarly. The map $f$ is defined this way on all of $\Sigma_g$ and it is by definition continuous. It is also clear that $f$ restricted to the interior of each cell is a homeomorphism onto a cell of the square complex structure of $(\vee^g S^1) \times (\vee^g S^1)$. Therefore $f$ is a combinatorial map of square complexes (see Section~\ref{square complexes} for definitions).\\
%is it "clear"?

In order to prove Proposition~\ref{tautfilling}, we must check that $f$ is locally injective and that it describes the Heegaard decomposition. In order to do this, we need two lemmas.\\

\begin{lemma} \label{cornersinj}
$f$ is locally injective.
\end{lemma}

\begin{proof}
In order to prove that $f$ is an immersion, or locally injective, it suffices to show that $f$ is locally injective at the vertices of $\Sigma_g$. This is because $f$ is a combinatorial map of square complexes (See \cite{haglundwise}).

Consider a vertex, $v$, of the square complex structure on $\Sigma_g$. Note that it is contained in the interior of a single polygon of the dual cell decomposition induced by the Heegaard diagram curves (since the squares are dual to this decomposition). Consider a small disk neighbourhood $D$ of $v$ and the intersection of the cells of the square complex structure with $D$. Recall that $f$ sends $1$-cells homeomorphically to $1$-cells, and $2$-cells homeomorphically to $2$-cells. Therefore, in order for $f$ to be locally injective at $v$, it is sufficient that $f(D \cap \sigma) \cap f(D \cap \delta) = \emptyset$ where $\sigma, \delta$ are cells in $\Sigma_g$ of the same dimension. 

Suppose two $1$-cells $\sigma$ and $\delta$ meeting at $v$ have  $f(D \cap \sigma) \cap f(D \cap \delta) \neq \emptyset$. It follows that $f(\sigma) = f(\delta)$. By the definition of $f_1$ and $f_2$ on the $1$-skeleton of $\Sigma_g$, this means that $\sigma$ and $\delta$ intersect the same Heegaard diagram curve. (Recall that each $1$-cell intersects exactly one diagram curve.) This means that the polygon dual to the vertex $v$ has two edges that are intervals on a single diagram curve. The hypothesis that the images of the intersection with $D$ intersect implies that one side of this diagram curve intersects the open polygon in two connected components, one along each of these edges. Thus the Heegaard diagram is not taut. This is a contradiction.\\
%check that this is enough detail

Suppose there are two $2$-cells $\sigma$ and $\delta$, that meet at $v$ and $f(\sigma) =f(\delta)$. Suppose further that $f(D \cap \sigma) \cap f(D \cap \delta) \neq \emptyset$. There must be two $1$-cells $e_1$ and $e_2$ that lie consecutively in the boundary of $\sigma$ that intersect $D$.  The edges $e_1$ and $e_2$ must be sent to consecutive $1$-cells in the boundary of $f(\sigma)$, incident to $f(v)$. This is because $f$ is a continuous combinatorial map of square complexes. Similarly, There are two $1$-cells, $e_3$ and $e_4$ that lie consecutively in the boundary of $\delta$ that must be sent to two consecutive edges in the boundary of $f(\delta)$ incident to $f(v)$.  The set $\{e_i\}$ must contain at least three elements. If it contained 2 or fewer elements, it would imply that only two squares meet at $v$ (or that $v$ does not have a surface neighbourhood). If only two squares meet at $v$, then the Heegaard diagram dual to the square complex structure contains a bigon. This contradicts the assumptions of our setting, see Lemma~\ref{tautfilling}.

 Since $f(\delta) = f(\sigma)$, it is possible to find three or four connected components of $D\cap (\cup_i e_i) \setminus \{ v \}$ that are sent by $f$ to two distinct connected components of edges in $f(D) \setminus f(v)$. We have already shown this is not possible earlier in this proof, as this means that there are two distinct $1$-cells $e_k$ and $e_j$ meeting at $v$ that have  $f(D \cap e_j) \cap f(D \cap e_k) \neq \emptyset$.
 
 Since $f$ is a combinatorial map, these are the only pathologies that would prevent $f$ from being injective on $D$.  Therefore $f$ is locally injective at the vertices of $\Sigma_g$ and therefore an immersion.
%\begin{figure}
%\centering \includegraphics[width=5in]{corners.pdf}
%\caption{The behaviour of edge corners and square corners.}
%\end{figure}
\end{proof}

%Assuming Lemma~\ref{cornersinj} is true for the time being, consider a small star shaped neighbourhood of a vertex, $v$, of the square complex structure of $\Sigma_g$.  There is only one vertex in this neighbourhood, and it is the only point in the neighbourhood sent to a vertex by $f$ (since $f$ is a combinatorial map of square complexes). The rest of the neighbourhood can be decomposed into open intervals and open disks (one in each of the corners that meet at $v$) that do not intersect pairwise and result from intersecting open cells with the star shaped neighbourhood. The above claim proves that the images of these sets do not intersect pairwise. Therefore the map $f$ is an embedding on an open star neighbourhood of each vertex.\\

We are now ready to prove Proposition~\ref{tautfilling}. The remaining work consists of checking that $f$ describes the given Heegaard decomposition.

\begin{proof}[Proof of Lemma~\ref{tautfilling}]
Begin with a taut filling Heegaard diagram on a surface of genus greater than zero. Construct the locally injective map $f:\Sigma_g \rightarrow (\vee^g S^1) \times (\vee^g S^1)$ as above. It remains to show that $f$ describes the Heegaard decomposition we began with. 

Recall that $f$ describes $M = \mathcal{H}_{1} \cup \mathcal{H}_{2}$ if, up to homotopy equivalence of the domain, and product homotopy equivalence of the target, $f= \iota_{1} \times \iota_{2} : \Sigma_g \longrightarrow \mathcal{H}_{1} \times \mathcal{H}_{2}$, where $\iota_i$ are the inclusion maps from the separating surface to $\mathcal{H}_{1}$ and $\mathcal{H}_{2}$ respectively. First, we check that $f_{1} = \iota_{1}$ up to homotopy equivalence of the target space. Consider the fibres of $f_{1}$. In what follows, refer to the example shown in Figure~\ref{fibres}. Recall that $f_{1}$ sends the interior of each square in $\Sigma_g$ to a maximal interval in $\vee^g S^1$. The blue curve running through the square is sent to the blue point in this interval. In fact, the full pre-image of each blue point in $\Sigma_g$ is a blue diagram curve. The full pre-image of any other point in the same maximal interval as a blue point is an isotopic curve. It may help to think about these statements on one closed square at a time first, and then put them together. The full pre-image of the wedge point is graph on $\Sigma_g$ given by the $1$-skeleton of the square complex structure, with the edges that intersect blue diagram curves removed. This can be seen by thinking about each square first. The pre-image of the wedge point in each closed square is the boundary edges that do not intersect a blue diagram curve. Notice that all the fibres except the pre-image of the wedge point are topological circles. All this follows from the definition of $f$.\\

\begin{figure}
\centering \includegraphics[width=5in]{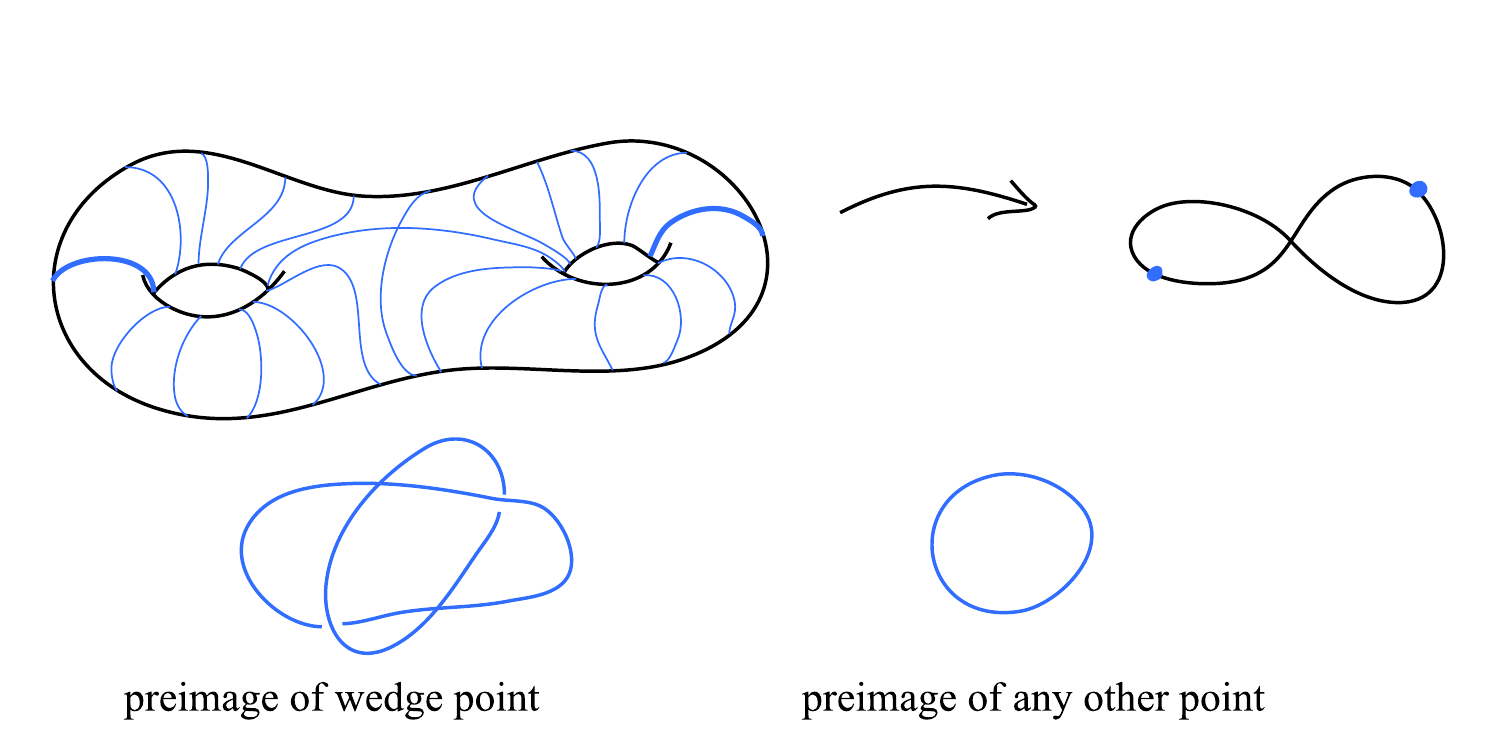}
\caption{A genus two example of one coordinate of $f$.}
\label{fibres}
\end{figure}

Consider including $\Sigma_g$ into a larger space $\mathcal{H}$, defined as follows.  Remove the $f_{1}$ pre-image of the wedge point from $\Sigma_g$. Each connected component of the remaining space is made up of the squares through which a given blue diagram curve runs. It is a cylinder of squares of ``height" $1$. There are $g$ such cylinders remaining, and they have a product structure of $(0,1) \times S^1$ induced by the fibres of $f$.  Include these cylinders into solid cylinders, $(0,1) \times D^2$ so that the $S^1$ fibres are included as boundaries to the $D^2$ fibres of the solid cylinder. The idea is to glue $\Sigma_g$ back together, and glue the solid cylinders back together accordingly. We now formalize this idea as follows. First glue two disks to each solid cylinder, so that it is $[0,1] \times D^2$ instead of $(0,1) \times D^2$. Parametrize the $D^2$ coordinate of these cylinders as a unit disk using polar coordinates $(r, \theta)$, $0\leq r\leq1$. Now glue $\Sigma_g$ back together along the pre-image of the wedge point. If two boundary disks ($\{0\} \times D^2$ or  $\{1\} \times D^2$) are identified along an arc on their $S^1$ boundary, each point in this arc belongs to a ray in each of the two disks. Identify these rays. Call the resulting space $\mathcal{H}$. Note that the centres of all the boundary disks are identified in this process. This is because the pre-image of the wedge point is connected. If it were not connected, then a collection of Heegaard diagram curves would separate the connected components. The union of these curves would be separating on $\Sigma_g$ and this contradicts the Heegaard diagram curves being homologically independent.\\

\begin{claim} \label{homeo}
$\mathcal{H}$ is a handle body and there is a homeomorphism $\phi: \mathcal{H} \longrightarrow \mathcal{H}_{1}$ so that the inclusion $\iota: \Sigma_g \hookrightarrow \mathcal{H}$ is equal to $\phi \circ \iota_{1}$.
\end{claim}

Recall that in order to complete the proof of Proposition~\ref{tautfilling}, we wanted to show that $f_{1} = \iota_{1}$ up to homotopy equivalence. If Claim~\ref{homeo} is true, the proof is completed by noting that $\vee^g S^1$ is a deformation retract of $\mathcal{H}$. The deformation retraction is realized by taking each point on each disk in the solid cylinders $(r, \theta)$, and continuously shrinking the radius r  to $0$. Call this deformation retraction $d$. Each cylinder becomes an interval, and the way the solid cylinders were glued together leads to these intervals being glued together into a bouquet of circles. This can be seen by considering the glueing locus under $d$, and noting that it is all sent to a single point in the end. Then note that $f_{1} = d \circ \iota$ because we set up the product structure on the hollow cylinders so that each $S^1$ fibre was the $f_{1}$ pre-image of a point in $\vee^g S^1$ (not the wedge point). So we have $f_{1} = d \circ \iota = d \circ \phi \circ \iota_{1}$, and $d \circ \phi$ is a homotopy equivalence as desired (since the composition of a homeomorphism and deformation retraction gives a homotopy equivalence together with inclusion composed with a homeomorphism). This completes the proof of Lemma~\ref{tautfilling}.
\end{proof}

Now we prove Claim~\ref{homeo}.

\begin{proof}[Proof of Claim~\ref{homeo}]
Recall how $\mathcal{H}_{1}$ is attached to $\Sigma_g$: a disk is attached along each blue diagram curve on $\Sigma_g$, and then a $3$-ball is attached to the resulting $2$-complex. We claim that the method of attaching $\mathcal{H}$ is equivalent to this. Consider, in the construction of $\mathcal{H}$, the disk $\{\frac{1}{2} \} \times D^2$ in each solid cylinder. The $S^1$ boundary of such a disk is a blue diagram curve. So both constructions attach disks along the blue diagram curves.  Let $A$ be the space remaining when we remove these disks and $\Sigma_g$ from $\mathcal{H}$.  We will prove that $A$ is homeomorphic to a $3$-ball that is attached to the rest of $\mathcal{H}$ in the same way that the $3$-ball is attached in the construction of $\mathcal{H}_{1}$.\\

The middle disk $\{\frac{1}{2} \} \times D^2$ cuts each solid cylinder in half. Each half is a cone on a disk given by the ``bottom" and ``sides" of the cylinder. This is shown below.

\begin{figure}
\centering \includegraphics[width=4.5in]{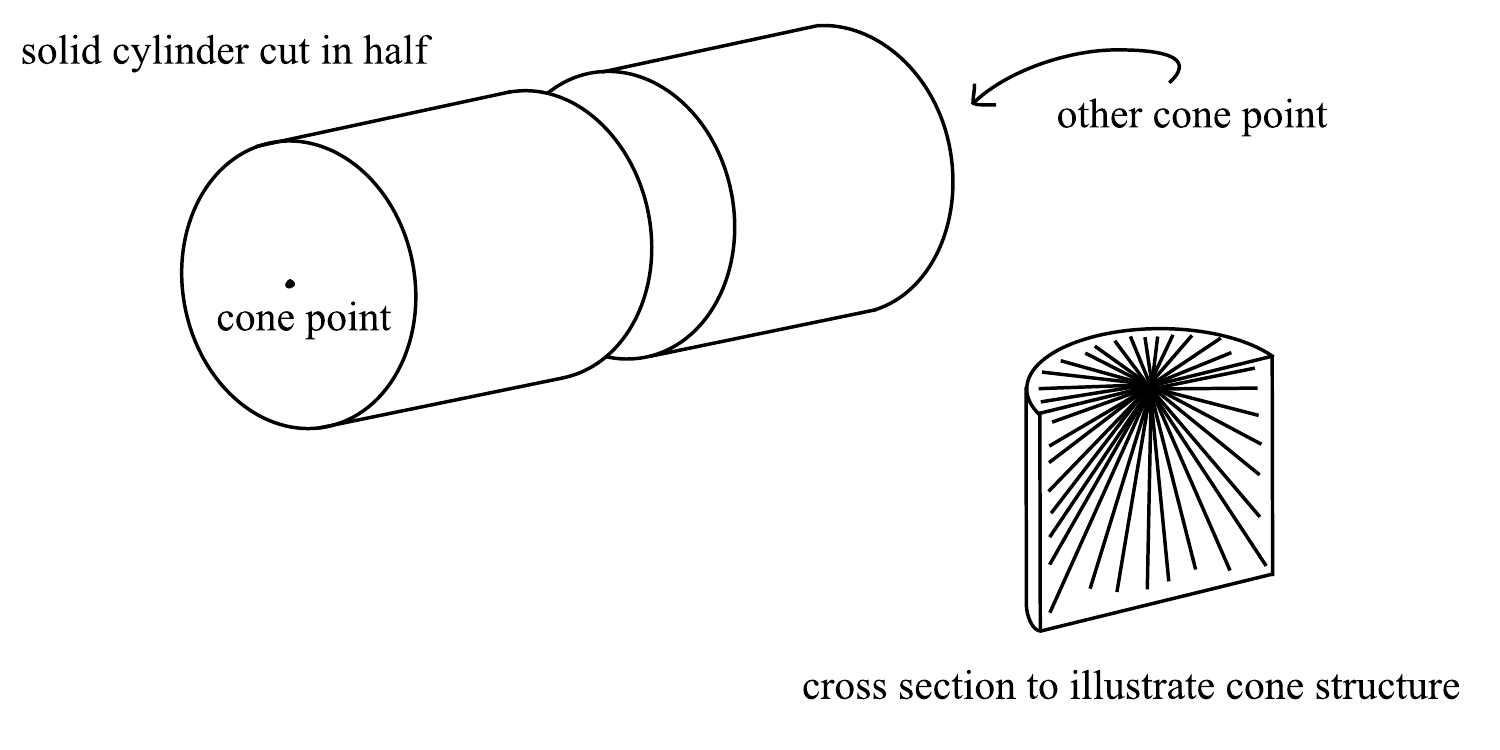}
\caption{The cone structure of half of a solid cylinder.}
\end{figure}

The two cones share the disk $\{\frac{1}{2} \} \times D^2$. Note that the cone structure restricted to $\{0\} \times D^2$ and $\{1\} \times D^2$ is compatible with the polar coordinates chosen on these disks. Recall that when the solid cylinders are glued together to form $\mathcal{H}$, rays of one of these disks are identified with rays of another. This implies that $A$ can be expressed as a cone relating to the cone structure we have placed on the half cylinders. Cut $\Sigma_g$ along the blue diagram curves (cut the cylinders in half). Glue in disks along all boundary components. This is simply $S^2$. $A$ is the interior of the cone on this space.  That is, an open $3$-ball. In this way, it is seen that $\mathcal{H}$ and $\mathcal{H}_{1}$ are attached to $\Sigma_g$ in the same way, completing the proof.
\end{proof}

\begin{rmk}
\label{rmk:osborne}
It is natural to ask whether Theorem~\ref{locinj} has a converse. Namely if $f$ being locally injective implies that the Heegaard splitting it encodes is irreducible. Osborne has provided examples that show that this is not the case \cite{osborne}. Figure~\ref{Osborne} depicts one such diagram for a genus two Heegaard splitting for a lens space with fundamental group $\mathbb{Z}_{13}$. The diagram is taut and filling, and therefore the map encoding it is locally injective, however it must be reducible since the genus is greater than one.

\begin{figure}
\centering \includegraphics[width=4.5in]{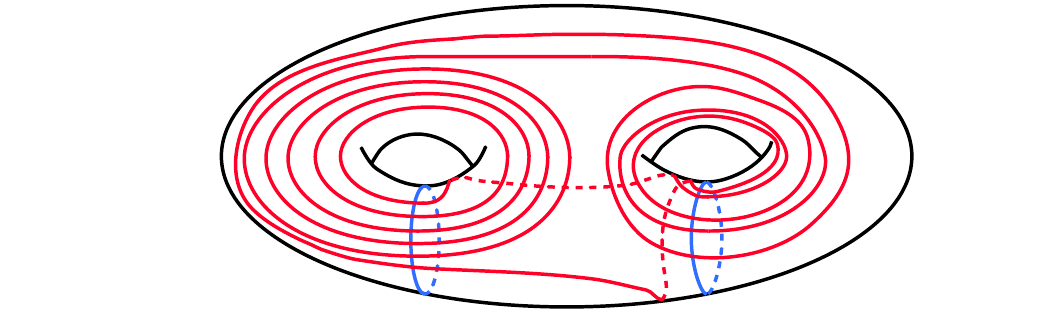}
\caption{Heegaard diagram for a lens space with fundamental group $\mathbb{Z}_{13}$}
\label{Osborne}
\end{figure}
\end{rmk}

\section{Virtually special square complexes and the augmented Heegaard diagram}
\label{sec: VH}

We have seen the emergence of square complex structures as we study Stallings maps. We now discuss the properties of these square complexes, and interpret them further in the language of geometric group theory.

\subsection{Non-positively curved square complexes that are virtually special}
Note that the cartesian product of two graphs naturally has the structure of a non-positively curved square complex. A square complex is non-positively curved if cycles in the link of every vertex have length at least four. The precise definitions and terminology about square complex structures can be found in \cite{haglundwise}.

When $f$ is locally injective, a non-positively curved square complex structure can be pulled back to the surface $\Sigma_g$, giving $\Sigma_g$ a tiling by squares where at least four squares meet at each vertex. The hyperplanes of this square complex structure (the curves dual to the square tiling) are precisely the Heegaard diagram curves. This is seen in the constructive proof of Theorem~\ref{locinj} and Proposition~\ref{tautfilling}. Therefore the following properties of the square complex structure are immediate.
\begin{itemize}
\item Hyperplanes do not self intersect (the hyperplanes are Heegaard diagram curves, and these do not self intersect)
\item The square complex is a VH-complex (V stands for vertical and H stands for horizontal. The blue curves of the Heegaard diagram are the \textit{vertical} hyperplanes, and the red curves are the \textit{horizontal} hyperplanes, see Definition~\ref{vh}).
\item Hyperplanes are two-sided (because all tubular neighbourhoods of simple closed curves on $\Sigma_g$ are orientable)
\item Hyperplanes do not directly self-osculate (since the Heegaard diagram is taut). Therefore the VH-complex is \textit{clean}.
\end{itemize}

The definitions of clean VH-complexes can be found breifly in Section~\ref{square complexes}. For more detail, see \cite{cleanVHcomplexes}. Note that the hyperplanes may indirectly self-osculate and they may also interosculate, so the square complex structure is not \textit{special} and so, while $f$ is a local embedding, it is not a local isometry.\\

\begin{prop}
\label{virtuallyspecial}
The square complex structure on $\Sigma_g$ dual to a taut Heegaard diagram is virtually special.
\end{prop}
\begin{proof}
This is a corollary of Theorem 5.7 in \cite{haglundwise}, which states that any compact VH-complex with clean horizontal hyperplanes is virtually special. The square complex structure here has both horizontal and vertical hyperplanes being clean. The surface $\Sigma_g$ is certainly compact. Thus, we have the desired result.
\end{proof}

The connection between the square complex structure and geometric group theory methods can be used to understand $f$ better (see Section~\ref{guirardel core}).

To end the section, we state a generalized version of Theorem~\ref{locinj} for all Heegaard splittings.\\

\begin{defn}
\label{pinch}
A \textit{pinched surface} $\mathcal{P}$ is given by the quotient $S / \sim$ of a surface by an equivalence relation of the following form.  Given a finite set of compact connected subsurfaces of $S$ with boundary, $x \sim y \iff x=y$ or $x$ and $y$ are in the same component of the chosen subsurfaces. We call the images of the non-trivial equivalence classes \textit{pinching points}\\
\end{defn}

We distinguish pinched surfaces from singular surfaces with nodes, since a subsurface that has more boundary components than an annulus can be pinched to a point in the quotient in Definition~\ref{pinch}, leading to multiple connected compact surfaces identified at a single point in the quotient.\\

\begin{defn}
\label{graphpinch}
The \textit{graph of a pinched surface} $S / \sim$ is given as follows. Let $X$ be the union of the finite set of subsurfaces of $S$ in the definition of $\sim$. There is one vertex for each connected component of $X$, as well as one vertex for each connected component of $S \setminus X$. There is an edge connecting two vertices if the regions of $S$ they represent share a boundary.
\end{defn}

Note that this graph is bipartite. A connected component of $X$ cannot share a boundary with another connected component of $X$ since they are disjoint compact subsurfaces. Therefore, the corresponding vertices cannot have an edge connecting them. If two connected components of the complement of $X$ shared a boundary, this would mean that the boundary curve is a connected component of $X$. However, $X$ is a union of subsurfaces, so this is not possible. Therefore the corresponding vertices cannot have an edge connecting them.

\begin{thm}
\label{thm:pinch}
Suppose that $M$ does not have $S^2 \times S^1$ as a connect summand, and has Heegaard splitting $M=\mathcal{H}_1\cup \mathcal{H}_2$. Then there is a Stallings map $f$ for this splitting that factors through a quotient to a pinched surface $\Sigma^*$. The pinched surface $\Sigma^*$ has a square complex structure such that the the quotient map is a combinatorial map of square complexes that is locally injective. Further, the Heegaard splitting can be recovered from the square complex structure on $\Sigma^*$.
\end{thm}

\begin{proof}
By Theorem~\ref{thm:existtaut}, a taut Heegaard diagram for $M=\mathcal{H}_1\cup \mathcal{H}_2$ can be chosen. Then let $f$ be a Stallings map descibing it. Next, take a the union of some tubular neighbourhoods of the curves. Any component of the complement must have null homotopic image by $f$. The pinched surface $\Sigma^*$ is constructed by making each non-simply connected component of the complement a single equivalence class, and making any other point a singleton equivalence class, and then taking the quotient. Up to homotopy, $f$ must factor through $\Sigma^*$.

The graph of $\Sigma^*$ cannot have any cycles. A cycle consists of a set of vertices and edges arranged in a circle. Since the graph is bipartite, if there was a cycle, the vertices must alternate in whether they correspond to a subsurface of $\Sigma_g$  that is pinched, or not. In particular, this implies that the cycle cannot consist exclusively of vertices that correspond to subsurfaces of $\Sigma_g$ that are not pinched. Let $A \subseteq \Sigma_g$ be a subsurface that is pinched, and whose corresponding vertex is part of the cycle. Any essential simple closed curve $\gamma$ on $A$ is a reducing curve. Further, it is non separating in $\Sigma_g$ since a path can be found from one side of $\gamma$ to the other that projects to a path that goes around the cycle in the graph of $\Sigma^*$. By Proposition~\ref{propirred}, the reducing curve $\gamma$ bounds an embedded disk in $\mathcal{H}_1$ and another disk in $\mathcal{H}_2$. This forms a copy of $S^2 \subseteq M$ that is not separating (since $\gamma$ is not separating). This implies that $S^2 \times S^1$ is a connect summand of $M$, a contradiction.

Note that a closed surface subset of $\Sigma^*$ cannot be homeomorphic to a sphere. This could only occur when in $\Sigma_g$, a connected component $C$ of the tubular neighbourhood of the Heegaard diagram has genus 0. Then in the quotient, every complementary surface region bordering the image of $C$ is simply connected, forming a sphere. However, then each diagram curve in $C$ would be separating on the sphere in $\Sigma^*$, and therefore separating in $\Sigma_g$ (since the graph of $\Sigma^*$ is a tree). This is a contradiction.

Further, there are no open surface subsets of $\Sigma^*$ whose closure is a pseudo-manifold. If this were the case, there would be a connected component $C$ of the tubular neighbourhood of the Heegaard diagram with two boundary curves that bound the same complementary region. In this case, any one of those boundary curves is a reducing, non-separating curve on $\Sigma_g$, a contradiction (again, because it implies that $S^2 \times S^1$ is a connect summand of $M$). 

Therefore $\Sigma^*$ consists of several closed surfaces of positive genus identified along points in a tree like fashion (more than two may be identified at a point).

Consider the image of the Heegaard diagram curves in $\Sigma^*$. If a diagram curve passed from one surface subset of $\Sigma^*$ to another, it must have passed through one of the pinching points. This is not possible by the definition of $\Sigma^*$. If a number of curves are homologous in one of the surface subspaces, they must have been homologous in $\Sigma_g$. If any two quotient curves intersect in $\Sigma^*$, the original two curves must have intersected in $\Sigma_g$. Finally if a quotient diagram curve $\overline{\gamma}$ were to separate the surface subspace it lies in, the original curve $\gamma$ must have been separating in $\Sigma_g$. This is because the complement of every subsurface of $\Sigma_g$ used in the definition of $\Sigma^*$ is not connected (since the graph of $\Sigma^*$ is a tree, removing any vertex leaves a disconnected space). Therefore there cannot be any curve from one side of $\gamma$ to the other side, since it would project to a cycle in the graph of $\Sigma^*$.

Therefore the image of the Heegaard diagram curves on each surface subspace of $\Sigma^*$ are disjoint, as well as homologically non-trivial and independent. If any surface subset of $\Sigma^*$ contains fewer red or blue curves than its genus, then $\Sigma_g$ had fewer red or blue curves respectively than its genus. If a surface subset of $\Sigma^*$ had more blue (respectively red) curves than its genus, since they do not intersect each other, one must be separating on the subsurface, which in turn implies that it is separating on $\Sigma_g$, since the graph of $\Sigma^*$ is a tree. This is a contradiction. Therefore the quotient of the diagram curves form a Heegaard diagram on each surface subset of $\Sigma^*$. This Heegaard diagram must be taut and filling because the diagram on $\Sigma_g$ was taut, and the quotient procedure described precisely takes each component of the complement that was not homeomorphic to a disk, and makes it so. Therefore the Lemma~\ref{tautfilling} can be applied to each surface subset to show that the quotient map is homotopic to a locally injective map.

Note that any simple closed curves on $\Sigma_g$ in the pre-image of a point of $\Sigma^*$ are reducing curves (that is, they are null homotopic in $\mathcal{H}_1$ and in $\mathcal{H}_2$). Thus if $\Sigma_g \neq \Sigma^*$, $M=\mathcal{H}_1 \cup \mathcal{H}_2$ is reducible and the restriction of $f$ to the surface subspaces of $\Sigma^*$ are encodings of Heegaaard splittings of connect summands. 

The square complex structure on each closed surface subset of non-positive Euler characteristic is dual to the Heegaard diagram on the subsurface, so it contains all the information of the Heegaard splitting.
\end{proof}
%%%%%%%%%%%%%%%%%%%%%%%%%%%%%%%%%%%%%%%%%%%%
%%%%%%%%%%%%%%%%%%%%%%%%%%%%%%%%%%%%%%%%%%%%%
%%%%%%%%%%%%%%%%%%%%%%%%%%%%%%%%%%%%%%%%%%%%

In the case where $S^2 \times S^1$ is a connect summand, the portion of $\Sigma^*$ corresponding to $S^2 \times S^1$ may not be a surface subspace, but rather a pseudo manifold of genus 0, on which $f$ cannot be locally injective.
In the cases where $\Sigma_g = \Sigma^*$ , $f$ itself is locally injective and $\Sigma_g$ inherits a square complex structure so that $f$ is a combinatorial map of square complexes.

There have been several proposed methods to combinatorially reduce or simplify a Heegaard diagram to a normal form from which the topology of the 3-manifold is more readily recognized ( for example the Whitehead conjecture \cite{whitehead} and Volodin-Kuznetsov-Fomenko conjecture \cite{wavemove} for $S^3$). Each is true in certain cases, but not generally \cite{vkfcounterexample}. At this point, it is natural to ask whether Theorem~\ref{thm:pinch} gives one such method to recognize the connect sum structure of the 3-manifold. However, Remark~\ref{rmk:osborne} describes an example of the failure of taut Stallings maps to pinch every reducing curve. Nevertheless, the structure of Stallings maps suggest that there is something \textit{beyond} the surface that can be manipulated and simplified. We expose and explore this in the remainder of the paper, where we will use the following definition.\\

\begin{defn}
\label{def:simple}
A Heegaard diagram is \textit{simple} if there is a maximal disjoint set of separating reducing curves on $\Sigma_g$ that do not intersect the Heegaard diagram (these reducing curves must induce the prime decomposition of the $3$-manifold, possibly with additional $S^3$ connect summands if the Heegaard splitting is stabilized). A Heegaard diagram is not simple if every maximal disjoint set of separating reducing curves on $\Sigma_g$ intersects the Heegaard diagram. A Stallings map describing a simple Heegaard diagram is also called \textit{simple}.
\end{defn}

Note that if a Heegaard diagram for the $3$-manifold $M$ is simple and $f$ is a Stallings map for the Heegaard diagram, the surface subspaces of $\Sigma^*$ each form minimal genus ($\geq 1$) Heegaard diagrams for the prime connect summands of $M$.

%%%%%%%%%%%%%%%%%%%%%%%%%%%%%%%%%%%%%%%%%%%%%%%%%%%%%%
%%%%%%%%%%%%%%%%%%%%%%%%%%%%%%%%%%%%%%%%%%%%%%%%%%%%%%
%%%%%%%%%%%%%%%%%%%%%%%%%%%%%%%%%%%%%%%%%%%%%%%%%%%%%%
\subsection{Group Actions on Trees and the Guirardel Core}
\label{guirardel core}
Let $M=\mathcal{H}_1\cup \mathcal{H}_2$ be a $3$-manifold with a Heegaard decomposition. Let $f$ be a Stallings map for this Heegaard splitting. By the proof of Theorem~\ref{thm:existtaut}, we may assume that $f$ is taut after possibly replacing it with another Stallings map in its $Out (F_g) \times Out (F_g)$ orbit. In this section, we uniformly assume that $S^2 \times S^1$ is not a connect summand of $M$. By Theorem~\ref{thm:pinch}, $f$ may be modified by a homotopy so that it factors through a pinched surface $\Sigma^*$ with a square complex structure, and $f$ is a combinatorial map of square complexes that is locally injective on each surface subspace of $\Sigma^*$. The Heegaard diagram corresponding to $f$ is taut but not necessarily filling.
%At this point, we focus on the restriction of $f$ to each surface subset. In this way, we assume that $f$ is a locally injective map (dropping to a connect summand of $M$ as necessary).
This is the setting for the analysis that follows.

Consider the following commutative diagram of covering spaces, where $T_1$ and $T_2$ are each a copy of the universal cover of $\bouquet$, $S$ is the cover of $\Sigma_g$ corresponding to the kernel of $f_*: \pi_1(\Sigma_g) \rightarrow \pi_1(\bxb)$, and $\hat{f}$ is the corresponding lift of $f$.
\[
\begin{tikzcd}
S \arrow[r, "\hat{f}"] \arrow[d] \arrow[loop left, "\pi_1(\Sigma_g)"]& T_1 \times T_2 \arrow[d] \arrow[loop right, "\pi_1(\Sigma_g) \times \pi_1(\Sigma_g) \geq \Delta_{\pi_1(\Sigma_g)}"]\\
\Sigma_g \arrow[r, "f"] & (\vee^g S^1) \times (\vee^g S^1)
\end{tikzcd}
\]

Consider the action of $\pi_1(\Sigma_g)$ on this diagram. The group $\pi_1(\Sigma_g)$ acts on $S$ by deck transformations (with kernel equal to $\ker (f_*)$). It also acts on $T_i$ by deck transformations. This is because $T_i$ is the universal cover of $\bouquet$ with the free group on $g$ generators $F_g$, acting as deck transformations. The inclusion of $\Sigma_g$ as the boundary of a handle body is $\pi_1$-surjective, and therefore $f_{i*} \colon \pi_1(\Sigma_g) \rightarrow F_g$ is surjective as well. Thus each element $\alpha \in \pi_1(\Sigma_g)$ acts by the deck transformation $f_{i*}(\alpha)$ on $T_i$. This creates an action of $\pi_1(\Sigma_g)\times\pi_1 (\Sigma_g)$ on $T_1 \times T_2$, and the diagonal subgroup $\Delta_{\pi_1(\Sigma_g)}$ (elements of the form $(x,x)$) gives an action of $\pi_1(\Sigma_g)$ on $T_1 \times T_2$. The map $\hf$ is equivariant with respect to the actions of $\pi_1(\Sigma_g)$ on $S$ and on $T_1 \times T_2$.

We would like to understand the diagonal action of $\pi_1(\Sigma_g)$ on $T_1 \times T_2$ because this is where the information of the $3$-manifold $M$ is hidden. The Guirardel core (Definition~\ref{defn:guirardel}, and defined with more details in \cite{guirardel}) provides a way to understand the compatibility of the two coordinates of this action. The larger the volume of the core (a subset of $T_1 \times T_2$), the larger the obstruction to finding a third tree $T_3$ with an action of $\pi_1(\Sigma_g)$ and surjective equivariant maps onto $T_1$ and $T_2$.

For instance, if $\pi_1(\Sigma_g)$ acts on $T_1$ and $T_2$ in the same way, the Guirardel core has zero volume and every simple closed curve in $\Sigma_g$ is reducing.  The corresponding Heegaard diagram is the connect sum of several genus one diagrams for $S^2 \times S^1$. This diagram completely reveals the connect sum structure of this particular $3$-manifold.

The precise relationship between the Guirardel core and Heegaard decompositions can be elucidated using Stallings maps, as we shall see in the following theorem. Recall that the Guirardel core is the subset of $T_1 \times T_2$ given by removing all the light quadrants (see Definition~\ref{defn:guirardel}, which recounts the definition in \cite{guirardel}). Let $\hat{f}_1$ and $\hat{f}_2$ be the coordinates of $\hat{f}$.

\begin{lemma}
\label{lem_heavy_quad}
A quadrant $\delta_1 \times \delta_2$ of $T_1 \times T_2$ is heavy iff it intersects $\hat{f}(S)$.
\end{lemma}

\begin{proof}
Let $\hat{f}_1$ and $\hat{f}_2$ be the coordinates of $\hat{f}$. Suppose $\delta_1 \times \delta_2 \cap \hat{f}(S)= \emptyset$. Using any point in $\hf(S)$ as the base point, its orbit must be contained in $\hf(S)$ because the map $\hf$ is equivariant with respect to the group action. For this reason, there can be no sequence of elements in the orbit of the base point that is contained in $\delta_1 \times \delta_2$. Therefore  $\delta_1 \times \delta_2$ must not be heavy (and therefore must be light).

Conversely, suppose $\delta_1 \times \delta_2 \cap \hat{f}(S)\neq \emptyset$. This is equivalent to $\hf_1^{-1}(\delta_1) \cap \hf_2^{-1}(\delta_2) \neq \emptyset$. We show that $\delta_1 \times \delta_2$ must be heavy. Consider the covers $S_1$ and $S_2$ of $\Sigma_g$ corresponding to the subgroups $\ker {f_1}_*$ and $\ker {f_2}_*$ of $\pi_1(\Sigma_g)$ respectively. These covers are planar surfaces that are the boundaries of the universal covers of the handle bodies $\mathcal{H}_1$ and $\mathcal{H}_2$ respectively. In fact, $\hat{f}_1$ can be interpreted as the projection from $S$ to $S_1$ followed by the inclusion of $S_1$ into the universal cover of $\mathcal{H}_1$, followed by a deformation retraction $d_1$ of this space to $T_1$ (a lift of the deformation retraction of $\mathcal{H}_1$ to $\vee^g S^1$). The blue Heegaard diagram curves lift to separating curves in $S_1$ and the red Heegaard diagram curves lift to separating curves in $S_2$. This is because on $\Sigma_g$ they are the pre-images of midpoints of edges in $\vee^g S^1$ by $f_1$ and $f_2$ respectively, and so in $S_i$, we may think of them as the pre-images of the mid points of edges of $T_i$ by the deformation retraction $d_i$ alluded to above. Recall that these curves are taut but not filling.

To understand how a path $\gamma$ in $\hat{S}$ is sent to a path in $T_1$ by $\hat{f}_1$, we may project $\gamma$ to $S_1$ and observe the algebraic intersection with blue diagram curves. (Recall that the blue diagram curves are the pre-images of the midpoints of the edges in $T_1$). The sum of the absolute value of the algebraic intersection numbers of $\gamma$ with each of the blue diagram curves is the distance between the end points of $\hat{f}_1(\gamma)$ in $T_1$.

Let $x \in \hf_1^{-1}(\delta_1) \cap \hf_2^{-1}(\delta_2)$. If $x$ is on (the lift of a) Heegaard diagram curve, it can be perturbed slightly to be off all diagram curves while remaining in $x \in \hf_1^{-1}(\delta_1) \cap \hf_2^{-1}(\delta_2)$ (since this set is open and the curves are not full-dimensional sets).

From here, we seek a sequence $\{\alpha_i \cdot x \}_i$ where $\alpha_i \in \pi_1(\Sigma_g)$, such that the sequences of distances $\{d_{T_1}(x, \alpha_i \cdot x)\}_i$ and $\{d_{T_2}(x, \alpha_i \cdot x)\}_i$ are both unbounded. We construct a sequence by constructing a path in $S$ beginning at $x$ that does not backtrack significantly when its image is projected to each tree. The backtracking in the $T_i$ is monitored by algebraic intersection with blue diagram curves when projected to $S_i$.

Since $x$ is not on a lifted Heegaard diagram curve, it must be in the interior of a connected component  of the complement in $S$. Call these complementary components. Since $f$ is locally injective, this complementary component cannot have a problem, and therefore, at most two edges of the boundary are sent to the boundary of $\delta_1 \times \delta_2$ and they cannot both be the same colour. The complementary component must have at least four edges, so there must be a way for a path beginning at $x$ to exit the complementary component via a vertex (transversally crossing a red and blue hyperplane simultaneously) without exiting $\hf_1^{-1}(\delta_1) \cap \hf_2^{-1}(\delta_2)$.

The path has now entered a new complementary component in $S$. We now extend the path further, taking care not to back track at each step. This new complementary component must also not have a problem. Therefore, as long as the path does not exit this complementary component by crossing the same two edges it crossed upon entry, the path's image will not back track in $T_1$ and in $T_2$. The complementary component has least four edges, and so again, it is possible to exit this complementary component via a vertex (simultaneously transversally crossing a blue and red hyperplane) without backtracking in $T_1$ and in $T_2$.

We continue in this way, extending the path. Each time the path crosses diagonally over a vertex (simultaneously transversally crossing a blue and red diagram curve), its image goes one unit deeper in $T_1$ and $T_2$.

Let $N$ be the number of complementary components in $\Sigma_g$. We extend the path from $x$ as above until it visits $N+10$ complementary components. At this point, we extend the path by a segment connecting it to the the nearest point in the $\pi_1(\Sigma_g)$ orbit of $x$. This is at most $N$ complementary components away. The completed path now connects $x$ to a point $\alpha_1 \cdot x$ for some $\alpha_1 \in \pi_1(\Sigma_g)$. The image of the path moves distance $N+10$ units deeper into $\delta_1$ and $\delta_2$, and then backtracks at most $N$ units. In total then, $d_{T_j}(x, \alpha_1 \cdot x) \geq 10 $  for $j=1, 2$.

We find $\alpha_{i+1} \cdot x$ by extending our path from $\alpha_i \cdot x$ in the same way as above: first taking care to move $N+10$ unit deeper into both trees, and then proceeding to the nearest point in the orbit of $x$. This constructs a sequence $\{ \alpha_i \}$ so that $d_{T_j}(x, \alpha_i \cdot x) \geq 10i$, for both $j=1,2$. These sequences of distances are unbounded as desired. Therefore, by definition, $\delta_1 \times \delta_2$ is heavy.
\end{proof}

Lemma~\ref{lem_heavy_quad} implies that the Guirardel core contains $\hf (S)$, but it also cannot be much larger than $\hf (S)$ (because every quadrant not intersecting $\hf (S)$ is light).

In fact, the Guirardel core is the minimal (with respect to inclusion), non-empty $\pi_1(\Sigma_g)$--invariant closed connected set such that the fibres of the projections to $T_1$ and $T_2$ are connected (see \cite{guirardel},  section 5). The theorem below then follows and characterizes the relationship between $f$ and the core.\\

\begin{thm}
\label{thm:coreandhatf}
Let $f: \Sigma_g \rightarrow (\vee^g S^1) \times (\vee^g S^1)$ be a locally injective Stallings map describing a Heegaard decomposition. The Guirardel core of the actions of $\pi_1(\Sigma)$ on $T_1$ and $T_2$ is the smallest set $X$ such that $\hat f(S) \subseteq X$ and ${p_1}^{-1} (t)$ is connected for every $t \in T_1$ and ${p_2}^{-1}(t)$ is connected for every $t \in T_2$.

\end{thm}
 \begin{proof}
 Let $X$ be the the smallest set $X$ such that $\hat f(S) \subseteq X$ and ${p_1}^{-1} (t)$ is connected for every $t \in T_1$ and ${p_2}^{-1}(t)$ is connected for every $t \in T_2$. The core contains $\hf(S)$ and the fibres of $p_1$ and $p_2$ are connected in the core, so the core must contain $X$. On the other hand, $X$ is a non-empty,  $\pi_1(\Sigma_g)$--invariant, closed, connected set, and the projections to $T_1$ and $T_2$ are connected by definition, so the core must be contained in $X$. This implies that $X$ is the Guirardel core.
 \end{proof}

 \subsection{Complexity of the Guirardel core and augmented Heegaard diagrams}
Let $f$ be a taut Stallings map describing a Heegaard splitting for a $3$-manifold $M$. We continue to assume that $S^2 \times S^1$ is not a connect summand of $M$. Recall that $f$ factors through a pinched surface $\Sigma^*$. Let the factor map be $g: \Sigma^* \rightarrow \vee^g S^1 \times \vee^g S^1$.  Let $\hat{\Sigma}^*$ be the cover of $\Sigma^*$ that corresponds to the kernel of $g$. Then $\hf$ factors through $\hat{\Sigma}^*$ with factor map $\hat{g}:\hat{\Sigma}^* \rightarrow T_1 \times T_2$.
\[
\begin{tikzcd}
S \arrow[dr] \arrow[rr, "\hf"] \arrow[dd]&					&T_1\times T_2 \arrow[dd]\\
							& \hat{\Sigma}^* \arrow[ru, "\hat{g}"]&	\\
\Sigma_g \arrow[dr] \arrow[rr, near start, "f"]	&					&(\vee^g S^1) \times (\vee^g S^1)\\
							& \Sigma^* \arrow[ru, "g"]	&
\arrow[from=2-2, to=4-2, crossing over]
\end{tikzcd}
\]

Given a Heegaard splitting, all choices of Stallings maps $[f]$ encoding it have the same $\pi_1$-kernel, $f_*$, so the topology of the cover $S$ is fixed. As we pass through the $\text{Out} F_g \times \text{Out} F_g$ orbit of $f$, passing through all the maps $\Sigma_g \rightarrow \vee^g S^1$ that describe the splitting, the square complex structure on $\hat{\Sigma}^*$ changes and $\hat{f}$ changes. If there is a reducing curve in $\Sigma_g$, it lifts to $S$. The disk its image bounds in $T_1 \times T_2$ is in the Guirardel core, because of the coordinate-wise convex condition. A longer reducing curve suggests more squares in the core needed to form a disk that this curve bounds. This leads to the following question.\\

\begin{q}
\label{pinchall}
Does the smallest (in some sense) Guirardel core in the Out$F_g \times$ Out$F_g$ orbit  give an $f$ that is simple (pinches a maximal set of reducing curves, see Definition~\ref{def:simple})?
\end{q}
 
 If the answer to this question is positive, we would have a new characterization for Heegaard diagrams that are connect sums of irreducible diagrams. Then working with Guirardel cores, it may be possible to find a combinatorial algorithm to move within the Out$F_n \times$ Out$F_n$ orbit to find these diagrams.
 
One attempt to compare sizes of Guirardel cores, comes from looking at their quotients. Given a Stallings map $f$ encoding a Heegaard splitting, let $\mathscr{G}_f$ be the quotient of the Guirardel core (of the action of $\pi_1(\Sigma_g)$ on $T_1$ and $T_2$) by the diagonal action of $\pi_1(\Sigma_g)$. The area of $\mathscr{G}_f$, in terms of number of squares, could be used as a measure of the size of the Guirardel core. However even $\mathscr{G}_f$ is likely infinite in many cases. Nevertheless, the following construction gives a partial answer to Question~\ref{pinchall}. We look at a certain subset of the core whose quotient does have finite area.

\subsubsection{Construction of augmented Heegaard diagrams}
Begin with a taut Stallings map for a Heegaard splitting for a $3$-manifold $M$ that does not have $S^2 \times S^1$ as a connect summand. Consider an embedded disk in $T_1 \times T_2$ whose boundary is a subset of $\hat f (S)$, and when pulled back to $S$, and then projected to $\Sigma_g$, the boundary is a simple closed curve $\gamma$. Note that such a disk is in the Guirardel core due to the convexity of fibres. We would like to focus on this part of the Guirardel core.

Take the pinched surface $\hat f$ factors through, $\hat{\Sigma}^*$, and attach an isomorphic disk (tiled with squares in the same way) along the pullback of its boundary. The quotient map $\hat g: \hat{\Sigma}^* \rightarrow T_1 \times T_2$ can be extended to this disk by sending it to the isomorphic disk bounded by $\hat g (\gamma)$ in $T_1 \times T_2$. The unextended map $\hat g$ is locally injective. If the extended $\hat g$ is no longer locally injective, we do not want to include the disk. In this way, the augmented Heegaard diagram is built by attaching disks that do not disrupt the local injectivity. A precise definition will follow shortly.

Given a taut Stallings map $f$ for a Heegaard splitting of $M$, the boundary components of the subsurfaces that are pinched by $f$ provide a set of disjoint reducing curves. We may use the two disks each of these reducing curves bounds (one in each handle body) to decompose $M$ by connect sum. The connect summands come with Heegaard splittings of lower geni (that add to give the original genus). In this way, we look at the closed surface subspaces of $\Sigma^*$ as Heegaard splittings of connect summands of $M$.

We may ask whether each of these are reducible. If $f$ restricted to one of the surface subspaces describes a reducible Heegaard splitting, then by Haken's lemma, there must be a reducing curve on the splitting surface. Since $S^2 \times S^1$ is not a connect summand, this curve must be separating, and thus we may use the two disks it bounds (one in each handle body) to decompose $M$ further by connect sum. We continue in this way until we have a set of disjoint reducing simple closed curves on $\Sigma$ such that the connect components of the complement represent irreducible Heegaard splittings of the prime connect summands of $M$.

There are many choices in the procedure above, and so there are several sets of curves one may end with. For any given possibility, we may construct an augmented Heegaard diagram.\\

\begin{defn}
\label{mycore}
Given a taut Stallings map $f$ for a Heegaard splitting, and a maximal set of disjoint reducing curves $\{\gamma_1, \ldots \gamma_n\}$, constructed as above (namely, so that they include the curves bounding the subsurfaces pinched by $f$), we construct the $\textit{augmented Heegaard diagram}$  by gluing disks to $\Sigma^*$ along each $\gamma_i$.
Beginning with the pinched surface $\Sigma^*$, for each $\gamma_i$, consider its lift to $\hat{\Sigma}^*$,  $\hat{\gamma}_i$ (recall $\hat{\Sigma}^*$ is the pinched surface $\hat{f}$ factors through). The image $\hat{g}(\hat{\gamma}_i)$ in $T_1 \times T_2$ bounds a disk $D$.  We attach an isometric copy of $D$ to $\Sigma^*$ along $\gamma_i$ and extend the map $g$ to send this copy to the projected image of $D$ in $\vee^g S^1 \times \vee^g S^1$. By modifying $\gamma_i$ by a homotopy, and removing some peripheral squares of $D$, it is possible to choose $D$ so that the extended map $g$ remains locally injective. The resulting square complex is an \textit{augmented Heegaard diagram} $\mathscr{C}$.
\end{defn}

Notice that we now have the following commutative diagram, with $\hat{\mathscr{C}}$ being the cover of $\mathscr{C}$ corresponding to $\ker g_*$.
\[
\begin{tikzcd}
S \arrow[dr] \arrow[rr, "\hf"] \arrow[dd]&					&T_1\times T_2 \arrow[dd]\\
							& \hat{\mathscr{C}} \arrow[ru, "\hat{g}"]&	\\
\Sigma_g \arrow[dr] \arrow[rr, near start, "f"]	&					&(\vee^g S^1) \times (\vee^g S^1)\\
							& \mathscr{C} \arrow[ru, "g"]	&
\arrow[from=2-2, to=4-2, crossing over]
\end{tikzcd}
\]

The purpose of only attaching disks that maintain the local injectivity of $\hat g$ is that this ensures that $\hat{\mathscr{C}}$ and $\mathscr{C}$ are clean $VH$-complexes. Note that $\mathscr{C}$ is compact, and therefore virtually special.\\

\begin{rmk}
\label{conditionforirreducibility} Let $M$ be a $3$-manifold that does not have $S^1 \times S^2$ as a connect summand. A Heegaard decomposition of $M$ is irreducible if and only if it has an augmented Heegaard diagram that is a surface.
\end{rmk}

It is also notable that one could define the augmented Heegaard diagram for a Stallings map $f$ that is not taut.  In this case, $f$ would still factor through a pinched surface, however, the factor map would not be locally injective. In this case, the image of $\hat f$ is may not be contained in the Guirardel core. However, one could still attach disks along reducing curves as in Definition~\ref{mycore} to obtain an augmented Heegaard diagram.

\subsubsection{Area properties of the augmented Heegaard diagram}

The fact that one may take irreducible Heegaard diagrams and then take their connect sum to form a Heegaard diagram for the connect sum of the $3$-manifolds they represent implies that certainly there exist Heegaard diagrams that are simple, as in Definition~\ref{def:simple}. Equivalently, there exist Stallings maps that pinch a maximal reducing multicurve. However, as mentioned earlier, it is difficult to detect from the combinatorics of a Heegaard diagram whether it is simple. The following theorem relates the area (number of squares) of an augmented Heegaard diagram to whether it is simple or not.\\

\begin{thm}
\label{thm:augheeg}
Let $M=\mathcal{H}_1 \cup \mathcal{H}_2$ be a Heegaard splitting of a $3$-manifold $M$ that does not have $S^2 \times S^1$ as a connect summand.
%If $f$ represents a taut Stallings map for this splitting, recall $\Sigma^*$ is the pinched surface $f$ factors through, as in Theorem~\ref{thm:pinch}.
In the space of Stallings maps, $B_{M,\mathcal{H}_1, \mathcal{H}_2}$, each has several augmented Heegaard diagrams. Of all of these, the Stallings map with the smallest-area augmented Heegaard diagram $\mathscr{C}$, has $\mathscr{C} = \Sigma^*$. In other words, a Stallings map with the smallest-area augmented Heegaard diagram is simple.
\end{thm}

We will see in the proof below that a Stallings map with the smallest-area augmented Heegaard diagram must also be taut. To make a Stallings map taut is equivalent to simplifying a Heegaard diagram with wave moves along curves that do not intersect the Heegaard diagram. While it had been conjectured by Whitehead and Volodin-Kuznetso-Fomenko that this may be enough to simplify a Heegaard diagram (of $S^3$), it has been shown that this is not enough \cite{morikawa} \cite{vkfcounterexample} \cite{ochiai} \cite{viro}. Theorem~\ref{thm:augheeg} quantifiably describes how to  continue to simplify all the way.

Overall, we see that asking for a Stallings map to be taut ties it to the Guirardel core (Theorem~\ref{thm:coreandhatf}). In order for this Stallings map to be simple, the Guirardel core must have a sense of minimal complexity, which is made precise by minimizing the area over all possible augmented Heegaard diagrams (Theorem~\ref{thm:augheeg}).
%%%%%%START HERE%%%%
\begin{proof}
Assume that $S^2 \times S^1$ is not a connect summand of $M$. Let $[f] \in B_{M, \mathcal{H}_1, \mathcal{H}_2}$. Suppose that $\mathscr{C} \neq \Sigma^*$. We claim that by applying a sequence of handle slides, the area of $\mathscr{C}$ can be reduced. Specifically, we reduce the area by applying wave moves to pinch curves that previously had a disk attached along them. The idea is that this eliminates the need to attach any disks to $\Sigma^*$ to obtain $\mathscr{C}$.

Consider $\gamma$, a simple closed curve
%or a simple path
in $\Sigma^*$, along which a disk is attached in $\mathscr{C}$.
%Recall from the proof of Proposition~\ref{prop: finite}, that if $\gamma$ is a simple path, it suffices to consider the case where it begins and ends at (distinct) vertices in the square complex structure of $\Sigma^*$.
Let $\hat \gamma$ be a lift of $\gamma$ to $S$. Then $\hat f (\hat \gamma)$ bounds a disk in $T_1 \times T_2$.

Consider the intersection word of $\gamma$ with the Heegaard diagram corresponding to $[f]$. Let $b_1, \ldots b_n$ be the blue letters that appear and let $r_1, \ldots r_m$ be the red letters that appear in this intersection word. (In other words,  $b_1, \ldots b_n$ are the blue curves $\gamma$ intersects and $r_1, \ldots r_m$ are the red curves $\gamma$ intersects). The cyclic intersection word of $\gamma$ can be reduced to the trivial word using the relation that $r$'s and $b$'s commute. This implies that each intersection noted by $b_i$ is paired with another noted $b_i^{-1}$ with which it will cancel in the reduction process (respectively for $r$'s). Each cancelling pair bounds a hyperplane in the disk bounded by $\hat f (\hat \gamma)$. Further, two hyperplanes in the disk intersect if and only if their end points on $\hat{f}(\hat \gamma)$ are linked.

%Both $b_i$ and $b_i^{-1}$ denote intersection with a single curve (in opposite directions). Since the Guirardel core has connected fibres, A segment of a hyperplane in the core must connect these two intersections. In fact a one-parameter family of segments of fibres of $p_1$ form a disk that $\hat f (\hat \gamma)$ bounds. Since $T_1 \times T_2$ is aspherical, this must be the same disk formed by a one parameter family of fibres of $p_2$ that is bounded by $\hat f (\hat \gamma)$. Note that for each $i$, if the end points of the hyperplane bounded by the intersections $b_i$ and $b_i^{-1}$ are linked with the end points of the hyperplane bounded by the intersections $r_j$ and $r_j^{-1}$ then the two hyperplane segments these pairs bound must intersect in the disk bounded by $\hat f (\hat \gamma)$. Therefore there is a square in the disk $\hat f (\hat \gamma)$ bounds corresponding to the intersection of $\hat b_i$ and $\hat r_j$. See Figure~\ref{tileddisk}.

\begin{figure}[h]
\centering
\begin{tikzpicture}
\node (image) at (0,0) {\includegraphics[width=4in]{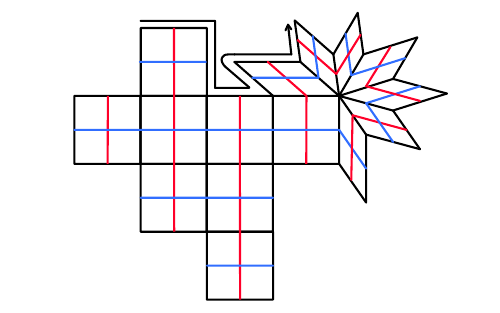}};
\node at (-1.8,3.1) {$\hat f (\hat \gamma)$};
\end{tikzpicture}
\caption{Disk in $T_1 \times T_2$ with $\hat f (\hat \gamma)$ as boundary.}
\label{tileddisk}
\end{figure}

Since $\hat f (\hat \gamma)$ is null homotopic in $T_1 \times T_2$,  there exists an $i$ such that a subsegment of $\gamma$ provides a path from $b_i$ to itself intersecting the same side twice without intersecting any other blue curves (corresponding to the place in the intersection word where cyclic reduction can begin). In other words, a wave move can be applied to $b_i$ along a subsegment of $\gamma$. Note that this segment of $\gamma$ may intersect red curves, and so after applying the wave move, new squares may emerge on $\Sigma^*$ corresponding to the intersection of $b_i$ and these red curves, see Figure~\ref{fig:reducingarea}. 

\begin{figure}[h]
\centering
\begin{tikzpicture}
\node (image) at (0,0) {\includegraphics[width=6.5in]{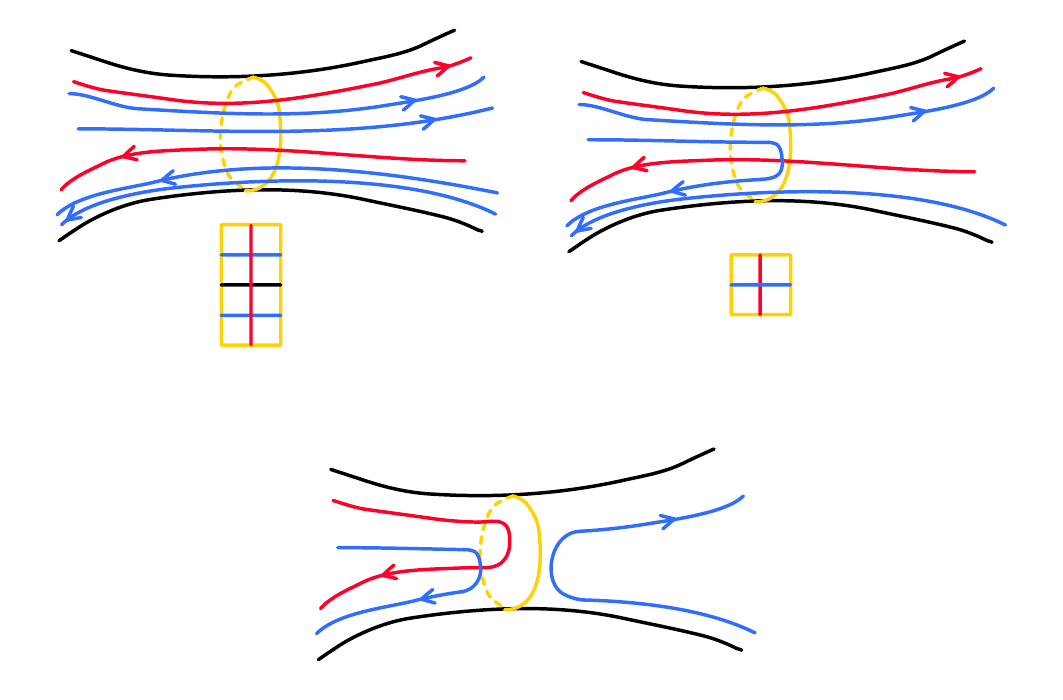}};
\node at (-4.3,4.6) {$ \gamma$};
\node at (3.7,4.5) {$ \gamma$};
\node at (7.2,4.6) {$ r_1$};
\node at (7.5,4) {$ b_2$};
\node at (0.7,2.7) {$ r_1$};
\node at (0.7,3.3) {$ b_1$};
\node at (7.7,1.8) {$ b_2$};
%%%%%%
\node at (-0.8,4.7) {$ r_1$};
\node at (-0.4,4.3) {$ b_2$};
\node at (-7.1,2.9) {$ r_1$};
\node at (-7.3,3.5) {$ b_1$};
\node at (-0.2,2.5) {$ b_1$};
\node at (-0.3,2) {$ b_2$};
%%%%%%%%
\node at (-0.3,-1.8) {$ \gamma$};
\node at (-3.3,-2.3) {$ r_1$};
\node at (-3.3,-3) {$ b_1$};
\node at (2.7,-2.2) {$ b_2$};
%%%%%%%%%%%
\node at (-4.3, -0.1) {$ r_1$};
\node at (-3.7, 1.5) {$ b_2$};
\node at (-3.7, 0.5) {$ b_1$};
\node at (3.6, 0.4) {$ r_1$};
\node at (2.9, 1) {$ b_2$};
%%%%%%%%%%%%%
\node at (-4.3,5.5) {Portion of original Heegaard diagram.};
\node[text width = 2cm] at (-6,1) {Disk tiled by squares attached along $\gamma$ in $\mathscr{C}$.};
\node[text width = 5cm] at (3.7,5.5) {Heegaard diagram after wave move applied to $b_1$ along segment of $\gamma$.};
\node[text width = 2cm] at (5.5,1) {Disk attached along $\gamma$ in $\mathscr{C}$ now has fewer squares.};
\node[text width=5cm] at (0,-0.6) {Several wave moves applied so that Heegaard diagram no longer intersects $\gamma$.};
\node[text width = 4cm] at (6,-3) {Here, $\gamma$ is pinched in $\mathscr{C}$ and does not bound a disk of positive area.};
\end{tikzpicture}
\caption{Before and after several wave moves have been applied. Heegaard diagram curves have been oriented arbitrarily so that intersection with $\gamma$ can be seen more easily. Note that in this case, one new intersection is created in the Heegaard diagram, between $r_1$ an $b_1$.}
\label{fig:reducingarea}
\end{figure}
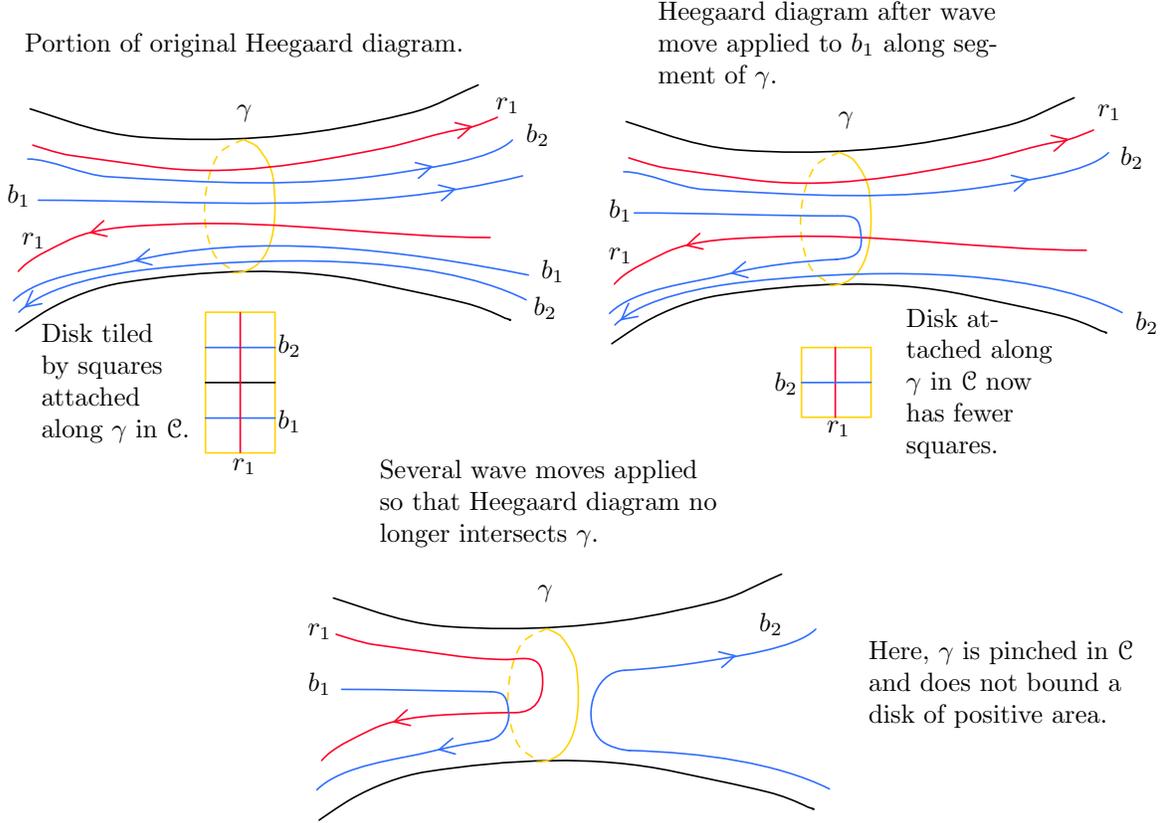

Continue in this way, applying wave moves until $\gamma$ does not intersect any Heegaard diagram curves. At this point $\gamma$ is pinched in $\Sigma^*$. Further, any new squares added to $\Sigma^*$ that remain after this process correspond to pairs of Heegaard diagram curves whose intersection points with $\gamma$ are linked. Thus the new squares are in bijection with a subset of the squares in the disk $\gamma$ bounded in $\mathscr{C}$ before the modifications, see Figure~\ref{fig:reducingarea}. (While this is not always the case, there are instances where strictly fewer new squares are created. One such example is seen in Figure~\ref{fig:reducingarea})

We claim that the segments of the Heegaard diagram curves that were deleted in these wave moves must have had some intersections with other diagram curves.  Thus, the total number of squares in $\mathscr{C}$ has been reduced by this move. This is proven as follows:

Begin by assuming the deleted arcs of Heegaard diagram curves had no intersections. Orient $\gamma$ and choose a base point away from any intersections with diagram curves. Take closed tubular neighbourhoods of the arcs of $b_1, \ldots b_n$ and $r_1, \ldots, r_m$ that are deleted in the process above. Call them $\tau_1 \ldots \tau_n$ and $\nu_1, \ldots \nu_m$ respectively. Choose these neighbourhoods so that they do not intersect the base point of $\gamma$, and so that the boundaries of $\tau_1 \ldots \tau_n$ and $\nu_1, \ldots \nu_m$  that run alongside $b_1, \ldots b_n$ and $r_1, \ldots, r_m$ do not intersect any other curves in the original Heegaard diagram. This is possible because of our assumption that the arcs of $b_1, \ldots b_n$ and $r_1, \ldots, r_m$ that are deleted do not have any intersections.

Consider a new curve $\eta$ given by beginning at the base point of $\gamma$ and following $\gamma$, until the first boundary of $\tau_1 \ldots \tau_n$ or $\nu_1, \ldots \nu_m$ is intersected. Then follow that boundary until it intersects $\gamma$ again. At this point follow $\gamma$ again until another boundary of $\tau_1 \ldots \tau_n$ and $\nu_1, \ldots \nu_m$ is intersected. Then follow this boundary until it returns to $\gamma$. Continue in this way, see Figure~\ref{fig:options}.

The process ends if $\eta$ returns to the base point of $\gamma$. There are only finitely many boundary arcs of $\tau_1 \ldots \tau_n$ and $\nu_1, \ldots \nu_m$, and segments of $\gamma$ between these. If the process does not end, $\eta$ must eventually intersect itself somewhere other than the base point. This is not possible because, for each arc in $\eta$, there is a unique arc that could possibly precede it. Therefore, the process ends and $\eta$ is a simple closed curve.

Further, we have assumed that the boundaries of $\tau_1 \ldots \tau_n$ and $\nu_1, \ldots \nu_m$ do not intersect any Heegaard diagram curves, and so $\eta$ does not intersect any of the curves of the original Heegaard diagram. Therefore, before the modifications, $\eta$ is a curve that must be pinched in $\Sigma^*$, or bound a disk in $\Sigma_g$.

Depending on the choice of base point on $\gamma$, there are many choices for $\eta$. As long as the deleted arcs of $b_1, \ldots b_n$ and $r_1, \ldots, r_m$ are not all on the same side of $\gamma$, it is possible to choose the base point so that $\eta$ intersects $\gamma$ transversally, see Figure~\ref{fig:options}. If $\eta$ is pinched, this is a contradiction since the pinched curves and the curves along which disks are glued in to form $\mathscr{C}$ are disjoint. If $\eta$ bounds a disk in $\Sigma_g$, this disk must not contain any Heegard diagram curves. However, by construction there are Heegaard diagram curves on either side of $\eta$, a contradiction.

\begin{figure}[h]
\centering
\includegraphics[width=4.5in]{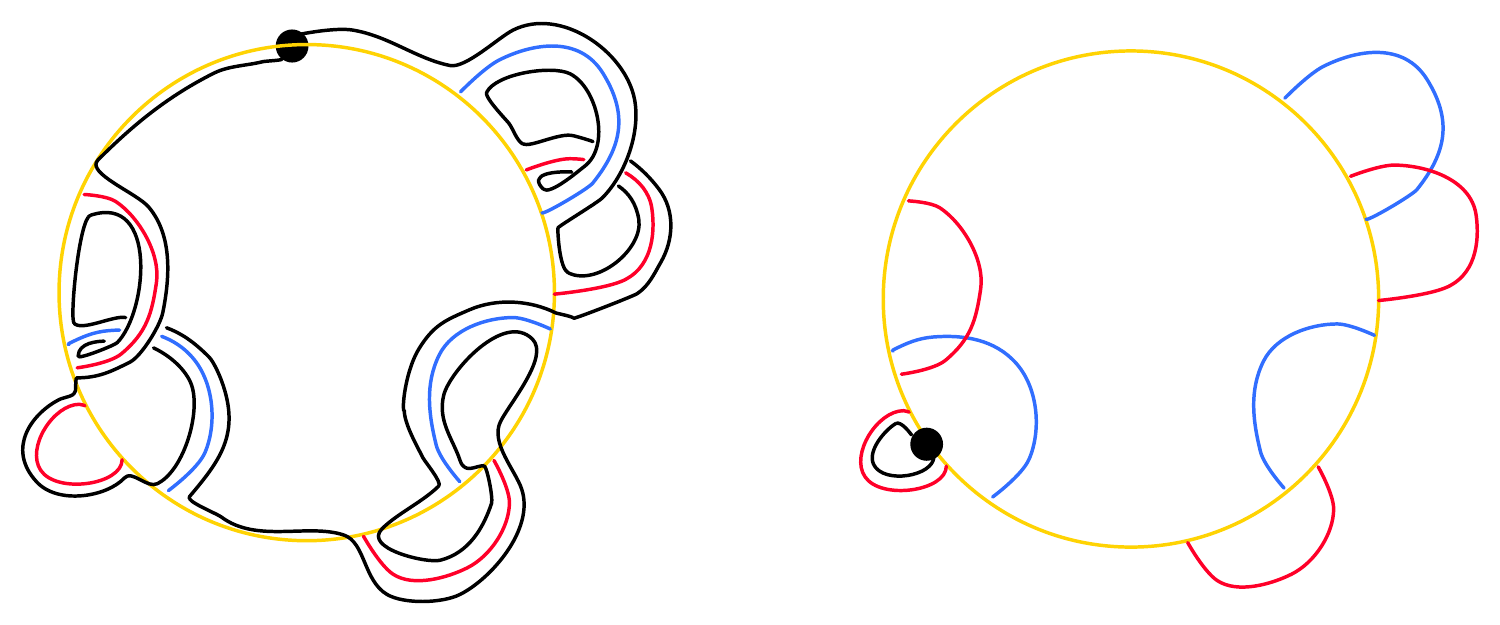}
\caption{By choosing the base point between two deleted arcs of Heegaard diagram curves that are on opposite sides of $\gamma$, $\eta$ intersects $\gamma$ transversally (figure on the left). Otherwise, this may not be the case (figure on the right).}
\label{fig:options}
\end{figure}

In the case where the deleted arcs of $b_1, \ldots b_n$ and $r_1, \ldots, r_m$ are all on the same side of $\gamma$, the assumption that they have no intersections imply that $\gamma$ bounds a disk on this side in $\Sigma^*$ (before attaching disks to form $\mathscr{C}$!). This is because, by definition of $\Sigma^*$, the Heegaard diagrams on each surface subspace of $\Sigma^*$ are filling. This is a contradiction because $\gamma$ must be an essential curve in $\Sigma^*$.

In either case, the deleted arcs of $b_1, \ldots b_n$ and $r_1, \ldots, r_m$ must have had some intersections. This translates to squares being deleted in $\Sigma^*$ when the wave moves are applied. Looking to $\mathscr{C}$, we see that some squares may be added, but these are in bijection with a subset of the squares that are deleted. The above now guarantees that this is a proper subset, and so, applying these wave moves reduces the area of $\mathscr{C}$.
\end{proof}

%Suppose that the minimum-area augmented Heegaard diagram $\mathscr{C}$ is a surface. Recall that all reducing curves on $\Sigma_g$ have null-homotopic image in $\mathscr{C}$. However, for minimum-area augmented Heegaard diagrams, $\mathscr{C}=\Sigma^*$ by Theorem~\ref{thm:augheeg}. Since $\mathscr{C}$ is a surface, this means that $\Sigma^* = \Sigma_g$, and no curves are pinched. Therefore, there are no reducing curves on $\Sigma_g$. In other words, the Heegaard decomposition of $M$ is irreducible.

\begin{cor}
Given a $3$-manifold $M$ that does not have $S^1 \times S^2$ as a connect summand, the minimum-area augmented Heegaard diagram over all possible augmented Heegaard diagrams describing $M$ corresponds to a Heegaard splitting that is not stabilized.
\end{cor}

\begin{proof}
Given a Heegaard splitting of $M$, the minimum-area augmented Heegaard diagram $\mathscr{C}$ for this splitting is equal to $\Sigma^*$ and corresponds to a simple Heegaard diagram.  If the Heegaard splitting was stabilized, one of the surface subspaces $X$ of $\Sigma^*$ must be a surface of genus one that forms an augmented Heegaard diagram for $S^3$. Consider collapsing this subsurface to a point to make a new pinched surface. This new pinched surface must be an augmented Heegaard diagram for a Heegaard splitting of $M$ as well, since it has the same set of connect summands that are not $S^3$. However, the area of the new augmented Heegaard diagram must be strictly less than the original augmented Heegaard diagram. This implies that the original augmented Heegaard diagram was not of minimum area across all possible augmented Heegaard diagrams describing $M$, a contradiction. Thus the original Heegaard splitting must have been stabilized.

\end{proof}

\section{Future directions}
\label{directions}
While the previous section gives a partial answer to Question~\ref{pinchall}, it would be great if one could find a more explicit description of augmented Heegaard diagrams, as complexity of the Guirardel core, as opposed to using the Stallings map. This would allow results and methods for Guirardel cores to be applied to Heegaard splittings more readily.

This paper connects several different structures: Heegaard diagrams, Stallings maps, virtually special square complexes, Guirardel cores of a surface group acting on a pair of trees. One natural question is how the properties of the $3$-manifold all these structures describe is manifested.\\

%\begin{q}
%How can properties of the $3$-manifold (for instance JSJ decomposition) be read from an augmented Heegaard diagram (or Stallings map, or Guirardel core)?\\
%\end{q}

\begin{q}
What methods and theorems from the study of Guirardel cores and virtually special square complexes be used to prove statements about $3$-manifolds, using the connection provided by Theorems~\ref{thm:pinch}, \ref{thm:coreandhatf} and \ref{thm:augheeg}
\end{q}

%%%%%%%%%%%%%%%%%%%%%%%%%%%%%%%%%%%%%%
%%%%%%%%%%%%%%%%%%%%%%%%%%%%%%%%%%%%%%%
%%%%%%%%%%%%%%%%%%%%%%%%%%%%%%%%%%%%

\section{Example: Lens spaces}
\label{lensspaces}
\subsection{Genus 1}
\label{g=1}
Consider genus 1 Heegaard splittings. These are lens spaces. They are all prime, and $S^2 \times S^1$ is the only lens space that is reducible. This is because any essential simple closed curve on the two-dimensional torus does not separate the surface. Therefore the existence of a curve that bounds a disk in both handle bodies implies the existence of a two-dimensional sphere embedded in the lens space that does not separate. It is well-known that an embedded, non-separating $2$-sphere implies that $S^2 \times S^1$ is a connect summand of the manifold. This is a contradiction, as we have excluded $S^2 \times S^1$ in our hypothesis and lens spaces are prime. Therefore, there cannot be a reducing curve and the Heegaard splitting is irreducible. Note that this means that the genus $1$ Heegaard splitting of $S^3$ is considered irreducible. We restrict attention to irreducible lens spaces in what follows (we do not consider $S^2 \times S^1$).\\

A lens space can be described by two integer parameters $p$ and $q$, and is denoted $L(p,q)$. Genus one Heegaard splittings for Lens spaces are pairs of simple essential curves on a two-dimensional torus. For simplicity, we define the $p$ and the $q$ in terms of Heegaard diagrams. This is not how it is usually done, but this view ties in best with this note. It is left as an exercise for the reader to show that these definitions are equivalent to the usual ones.

The first integer, $p$ is equal to the number of intersections of a Heegaard diagram for $L(p,q)$ (notice there are no handle slides for the single genus case). The second integer, $q$ is coprime to $p$, and is found by starting at a point on a diagram curve and numbering the intersection points as one travels along the curve (pick a direction). Then starting at a point on the other curve, read off the numbers while travelling along it (pick a direction). The consecutive numbers on the list obtained will differ by a constant, modulo $p$. This constant is $q$, and $L(p, q)=L(p, -q (\text{mod}p))=L(p,q^{-1} (\text{mod}p))$. Recall, the fundamental group of the lens space $L(p, q)$ is $\mathbb{Z}/\mathbb{Z}_p$.\\

Consider the domain and the target for Stallings maps that encode genus one Heegaard splittings. The splitting surface is the two-dimensional torus, and the target space is the product of two bouquets of a single circle. This is also the two-dimensional torus. Now Theorem~\ref{locinj} implies that there is a locally injective representative map that encodes the Heegaard splitting. Locally injective endomorphisms of $T^2$ are covering maps. And so, the irreducible lens spaces can be encoded by finite sheeted covering maps of the torus.\\ 

\begin{prop}
The degree of the covering map is $p$.
\end{prop}
\begin{proof}
Begin with a covering map $f:T^2 \longrightarrow S^1 \times S^1$ that describes $L(p,q)$. Recall, that one can obtain a Heegaard diagram for the encoded splitting by taking the curves that are the coordinate pre-images of a point on each circle coordinate of $S^1 \times S^1$. In Figure~\ref{torustarget}, we see two points whose pre-image should be considered to obtain a diagram. The pre-image of the blue points should be coloured blue and the pre-image of the red points should be coloured red.\\

\begin{figure}
\centering \includegraphics[width=4in]{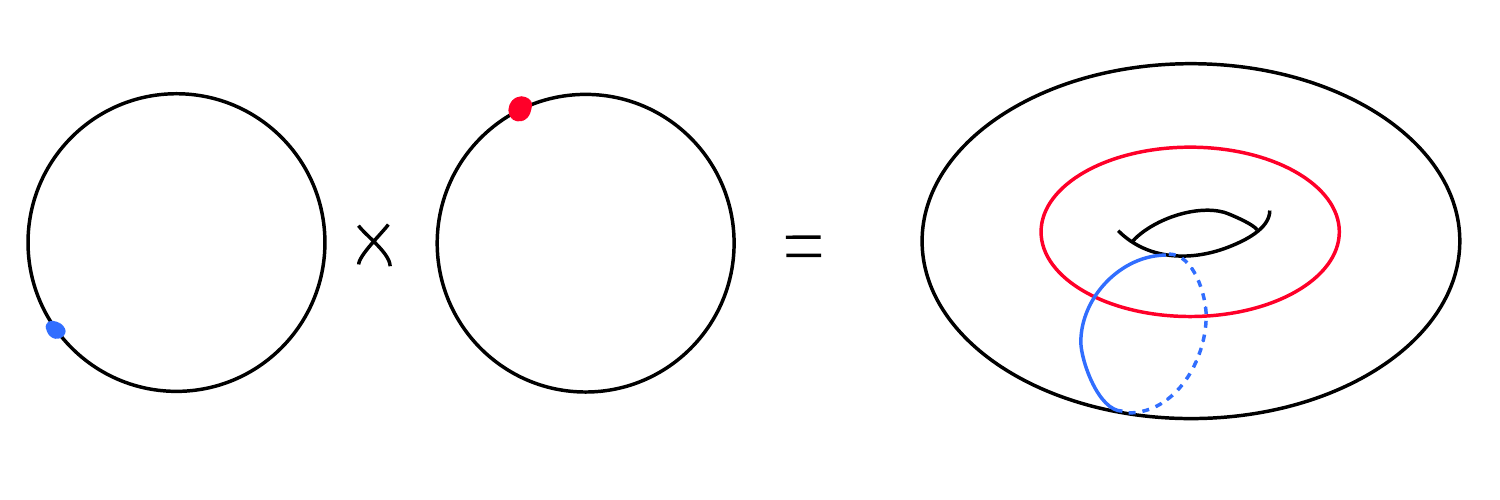}
\caption{$S^1 \times S^1$ shown with a blue point and a red point selected on each coordinate circle.}
\label{torustarget}
\end{figure}

Note that the fiber above the blue point and the fiber above the red point intersect once in the target space (shown in Figure~\ref{torustarget}). Call the intersection point $y$. Now, in the domain $T^2$, the diagram we obtained has $p$ intersection points. Therefore there are $p$ pre-images of $y$. This shows that the degree of $f$ is $p$.
\end{proof}

In the example in Figure~\ref{coveringmap}, using the notation of the proof above, $f$ is the map between tori and it describes $L(5,2)$.

\begin{figure}
\centering \includegraphics[width=\textwidth]{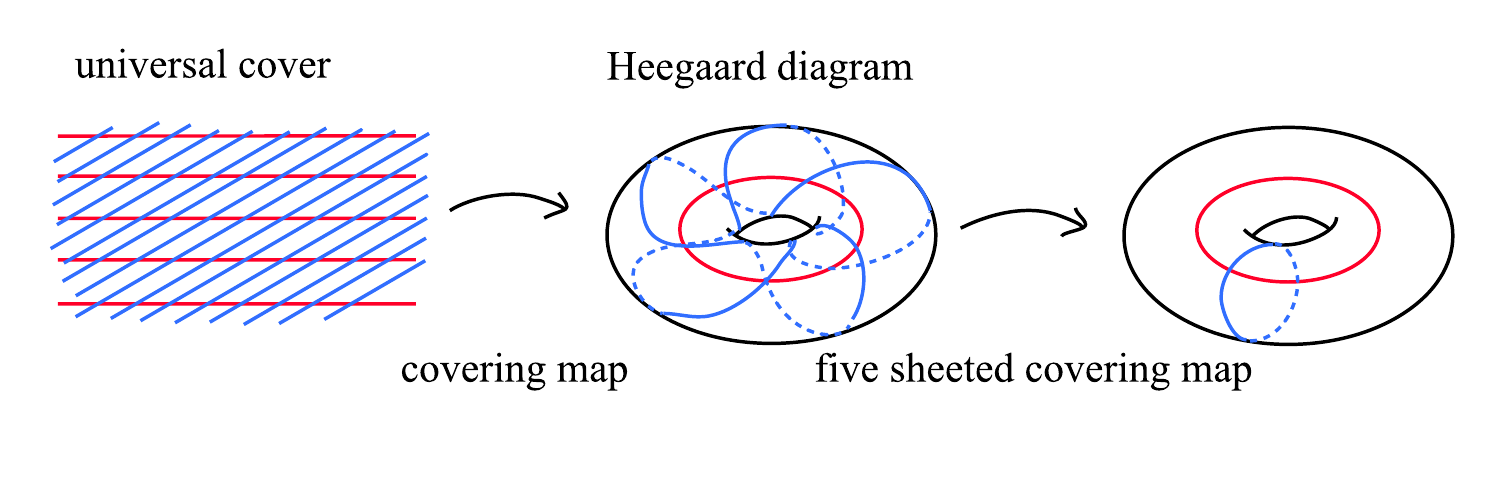}
\caption{Pre-images of the red and blue points.}
\label{coveringmap}
\end{figure}

The other integer, $q$, can be seen as follows. Recall that the squares can be combinatorially drawn as the dual to the Heegaard diagram. This shows how the $p$ squares are attached together to give a torus.\\ \\

\begin{rmk}
The square complex structure of the domain $T^2$ is given by attaching $p$ squares to make a cylinder of circumference $p$ and height one, and then attaching the remaining boundary components with a $\frac{q}{p}$ twist. There is a distinct square complex structure for each lens space. These can each be thought of as a point in the moduli space of conformal structures on $T^2$.
\end{rmk}

A $\frac{q}{p}$ twist is needed because otherwise, there would not be a single Heegaard diagram curve of each colour after gluing. The countable set of points in the moduli space that encode Heegaard diagrams of Lens spaces is unbounded since $L(p,1)$ always has a simple closed curve of length $\sqrt{2}$ and a transversal curve of length $p$ (taking the squares to have unit side lengths). So as $p \rightarrow \infty$, the first curve becomes conformally short.

These Stallings maps are simple, since there are no reducing curves and $\Sigma = \Sigma^* = \mathscr{C}$. The space $T_1 \times T_2$ is homeomorphic to $\mathbb{R}^2$, and the Guirardel core is the entire plane. 

\subsection{Genus 2}
\label{osborne}
\begin{figure}[h]
\centering 
\begin{tikzpicture}
\node (image) at (0,0) {\includegraphics[width=\textwidth]{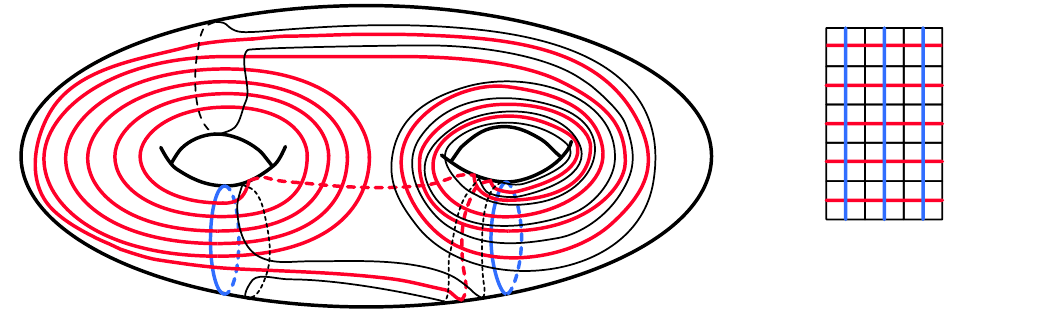}};
\node at (-2.8,1) {$a$};
\node at (-6.2, 2.1) {$b$};
\draw (-5.4,1.7) to [out=90, in =0] (-6, 2.1);
\node at (-5.1, -1.9) {$x$};
\node at (0.45, -2.3) {$y$};
\draw (-0.5,-1.95) to [out=0, in =120] (0.3, -2.2);
\node at (-2.4, -1.4) {$\eta$};
\node at (4.4, -0.7) {$a$};
\node at (4.4, -0.1) {$a$};
\node at (4.4, 0.5) {$a$};
\node at (4.4, 1.1) {$a$};
\node at (4.4, 1.8) {$a$};
\node at (4.9, 2.2) {$y$};
\node at (5.5, 2.2) {$y$};
\node at (6.1, 2.2) {$y$};
\end{tikzpicture}
\caption{A taut and filling Heegaard diagram for a Lens space with fundamental group $\mathbb{Z}_{13}$. The curves $a$ and $b$ are red, while the curves $x$ and $y$ are blue. The curve $\eta$ is reducing, and bounds the disk shown at right in the augmented Heegaard diagram $\mathscr{C}$. }
\label{area18}
\end{figure}
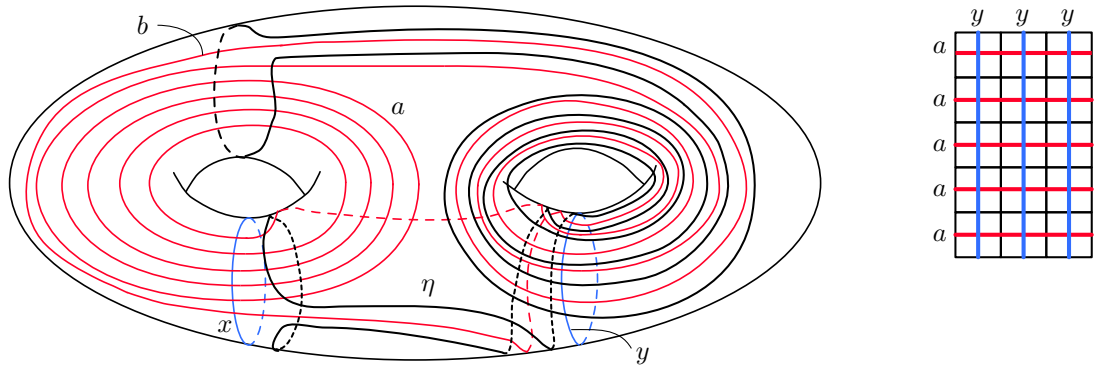

Any genus 2 Heegaard splitting of a Lens space is reducible. However, some Stallings maps encoding genus two Heegaard splittings of Lens spaces are taut. Further in these examples, the square complex $\Sigma^*$  has smaller total area than their genus 1 counterparts. These examples were found by Osborne, and they give a counterexample to a converse to Theorem~\ref{locinj}. Figure~\ref{Osborne} shows one such Heegaard diagram of a Lens space with fundamental group $\mathbb{Z}_{13}$. A Heegaard decomposition of genus one for this manifold would have an augmented Heegaard diagram of area 13 (see Section~\ref{g=1}). The genus two diagram shown in Figure~\ref{Osborne} corresponds to a Stallings map $h$ such that $\Sigma_g= \Sigma^*$ has area 11.  However, $\mathscr{C}$ must also contain disks of squares attached along  a maximal disjoint set of reducing curves. A reducing curve which bounds a disk of area 15 is shown in Figure~\ref{area18}. This brings the area of $\mathscr{C}$ well above its genus one counterpart.

\bibliography{Heegaard_splittings_and_square_complexes} 
\bibliographystyle{plain}

\end{document}